\documentclass{article}

\usepackage{ifpdf}
\usepackage{microtype}
\usepackage{graphicx}
\usepackage{booktabs} 

\usepackage{epstopdf}
\usepackage{hyperref}

\usepackage[accepted]{icml2018_modified}

\usepackage{url}
\usepackage{graphicx}
\usepackage{subfig}
\usepackage{epstopdf}
\usepackage{listings}
\usepackage{amsmath}
\usepackage{amssymb}
\usepackage{mathtools}
\mathtoolsset{showonlyrefs}
\usepackage{hyperref}
\usepackage[utf8]{inputenc}
\usepackage[english]{babel}
\usepackage{amsthm}

\usepackage[font=small,skip=5pt]{caption}

\newtheorem{theorem}{Theorem}
\newtheorem{corollary}{Corollary}[theorem]
\newtheorem{lemma}[theorem]{Lemma}
\newtheorem{assumption}{Assumption}

\usepackage[colorinlistoftodos,bordercolor=orange,backgroundcolor=orange!20,linecolor=orange,textsize=scriptsize]{todonotes}

\newcommand{\diag}[1]{{\mathrm{Diag}}\left(#1\right)}

\newcommand{\norm}[1]{\left\Vert #1 \right\Vert}
\newcommand{\trace}[1]{\mathrm{Tr} \left(#1\right)}
\newcommand{\Sum}[2]{\sum\limits_{#1}^{#2}}

\newcommand{\abs}[1]{\left\vert #1 \right\vert}

\newcommand{\Set}[1]{\left\{#1\right\}}

\newcommand{\eqdef}{:=}
\newcommand{\R}{\mathbb{R}}
\newcommand{\Exp}{\mathbb{E}}
\newcommand{\E}[1]{{\mathbb{E}\left[#1\right] }}    

\newcommand{\cD}{{\cal D}}

\newcommand{\cO}{{\cal O}}

\newcommand{\mA}{{\bf A}}
\newcommand{\mB}{{\bf B}}
\newcommand{\mC}{{\bf C}}
\newcommand{\mE}{{\bf E}}

\newcommand{\mH}{{\bf H}}
\newcommand{\mI}{{\bf I}}

\newcommand{\mM}{{\bf M}}

\newcommand{\mQ}{{\bf Q}}

\newcommand{\mS}{{\bf S}}

\newcommand{\mU}{{\bf U}}
\newcommand{\mV}{{\bf V}}
\newcommand{\mW}{{\bf W}}

\icmltitlerunning{Stochastic Spectral and Conjugate Descent Methods}

\begin{document}

\twocolumn[
\icmltitle{Stochastic Spectral and Conjugate Descent Methods}




\begin{icmlauthorlist}
\icmlauthor{Dmitry Kovalev}{mipt}
\icmlauthor{Eduard Gorbunov}{mipt}
\icmlauthor{Elnur Gasanov}{mipt}
\icmlauthor{Peter Richt\'{a}rik}{kaust,ed,mipt}
\end{icmlauthorlist}

\icmlaffiliation{kaust}{King Abdullah University of Science and Technology, Thuwal,  Saudi Arabia}
\icmlaffiliation{ed}{University of Edinburgh, Edinburgh,  United Kingdom}
\icmlaffiliation{mipt}{Moscow Institute of Physics and Technology, Dolgoprudny, Russia}

\icmlcorrespondingauthor{Dmitry Kovalev}{dakovalev1@mail.ru}
\icmlcorrespondingauthor{Peter Richt\'{a}rik}{peter.richtarik@kaust.edu.sa, peter.richtarik@ed.ac.uk}

\icmlkeywords{spectral methods, randomized coordinate descent, acceleration, mini-batching}

\vskip 0.3in
]



\printAffiliationsAndNotice{}  

\begin{abstract}
The state-of-the-art methods for solving optimization problems in big dimensions are variants of randomized coordinate descent (RCD).  In this paper we introduce a fundamentally  new  type of acceleration strategy for RCD based on the augmentation of the set of coordinate directions by  a few {\em spectral} or {\em conjugate} directions. As we increase the number of extra directions to be sampled from, the rate of the method improves, and interpolates between the linear rate of RCD and a linear rate {\em independent of the condition number}. We develop and analyze also inexact variants of these methods where the spectral  and conjugate directions are allowed to be approximate only. We motivate the above development by proving several negative results which highlight the limitations of  RCD with importance sampling.
\end{abstract}

\section{Introduction}\label{sec:introduction}

An increasing array of learning and training tasks reduce to optimization problem in very large dimensions. The state-of-the-art algorithms in this regime are based on {\em randomized coordinate descent (RCD)}. Various acceleration strategies were proposed for RCD in the literature in recent years, based on techniques such as Nesterov's momentum \cite{nesterov1983method, LeeSidford, APPROX, NU_ACDM, NestStich-accelerated2017}, heavy ball momentum \cite{polyak1964, loizou2017momentum},   importance sampling \cite{nesterov2012efficiency, NSYNC},   adaptive sampling \cite{AdaSDCA}, random permutations \cite{randperm2016}, greedy rules \cite{Nutini2015}, mini-batching \cite{PCDM}, and locality breaking \cite{breaking2017}. These techniques enable faster rates in theory and practice.  

In this paper we introduce a fundamentally  new  type of acceleration strategy for RCD  which relies on the idea of {\em enriching} the set of (unit) coordinate directions $\{e_1,e_2,\dots,e_n\}$ in $\R^n$, which are used in RCD as directions of descent, via the addition of a few {\em spectral} or {\em conjugate} directions. The  algorithms we develop and analyze in this paper randomize over this enriched  larger set of directions.

\subsection{The problem}
For simplicity\footnote{Many of our results can be extended to convex functions of the form $f(x)=\phi(\mA x) - b^\top x$, where $\phi$ is a smooth and strongly convex function. However, due to space limitations, and the fact that we already have a lot to say in the special case $\phi(y) = \tfrac{1}{2}\|y\|^2$, we leave these more general developments to a follow-up paper.}, we focus on quadratic minimization  \begin{equation} \label{eq:quad_opt} \min_{x\in \R^n} f(x) = \frac{1}{2}x^\top \mA x - b^\top x, \end{equation} where $\mA$ is an $n\times n$ symmetric and positive definite matrix. The optimal solution is unique, and equal to $x_*=\mA^{-1}b$. 

\subsection{Randomized coordinate descent}

 Applied to \eqref{eq:quad_opt},  RCD performs the iteration
\begin{equation} 
\label{eq:alg_line_RCD}
x_{t+1} = x_t - \frac{\mA_{:i}^\top x_t - b_i}{\mA_{ii}} e_i,
\end{equation}
where at each iteration, $i$ is chosen with probability $p_i>0$. It was shown by \citet{leventhal2010} that if the probabilities are proportional to  the diagonal elements of $\mA$ (i.e., $p_i\sim \mA_{ii}$), then the random iterates of RCD satisfy
\[ \Exp [\|x_t - x_* \|_{\mA}^2] \leq (1-\rho)^t \|x_0 - x_*\|_\mA^2,\]
where $\rho = \tfrac{\lambda_{\min}(\mA)}{{\rm Tr}(\mA)}$ and $\lambda_{\min}(\mA)$ is the minimal eigenvalue of $\mA$. That is,  as long as the number of iterations $t$ is at least \begin{equation} \label{eq:09hf09fs}\cO\left( \frac{{\rm Tr}(\mA)}{\lambda_{\min}(\mA)} \log \tfrac{1}{\epsilon} \right),\end{equation} we have $\Exp [\|x_t - x_* \|_{\mA}^2] \leq \epsilon$. Note that ${\rm Tr}(\mA)/\lambda_{\min}(\mA) \geq n$, and that this can be arbitrarily larger that $n$.

{\small
\begin{table*}[t]
\centering
\begin{tabular}{|l|c|c|c|}
\hline
Method Name &    Algorithm  &  Rate &  Reference \\
\hline
stochastic descent (SD)          & \eqref{alg_lin_sd},  Algorithm~\ref{alg:SD}  & \eqref{eq:rate-line-SD}, Lemma~\ref{lem:rate_of_SD} &  \citet{gower2015randomized} \\
stochastic spectral descent (SSD) & Algorithm~\ref{alg:SSD} & \eqref{eq:rate-SSD}, Theorem~\ref{thm:SSD} & NEW \\
stochastic conjugate descent (SconD) & read Section~\ref{sec:8ys089h0df} & Theorem~\ref{thm:SSD} & NEW \\
randomized coordinate descent (RCD)          &  \eqref{eq:alg_line_RCD}, Algorithm~\ref{alg:RCD}      & \eqref{eq:09hf09fs}, \eqref{eq:rate-line-RCD-p}  &    \citet{gower2015randomized} \\
stochastic spectral  coordinate descent (SSCD) & Algorithm \ref{alg:SSCD} & \eqref{eq:rate-line-SSCD}, Theorem~\ref{thm:SSCD} & NEW \\
mini-batch SD (mSD) & Algorithm \ref{alg:Par_SD} & Lemma~\ref{lem:rate_of_parallel_SD} &  \citet{richtarik2017stochastic} \\
mini-batch SSCD (mSSCD) & Algorithm \ref{alg:Par_SSCD} & Theorem~\ref{thm:Par_SSCD} & NEW \\
inexact SconD (iSconD) & Algorithm~\ref{alg:iSconD} & Theorem~\ref{thm:iSconD} & NEW \\
inexact SSD (iSSD) & Algorithm~\ref{alg:iSSD} & see Section~\ref{sec:iSSD-conv} & NEW \\
\hline
\end{tabular}
\caption{Algorithms described in this paper.}
\end{table*}
}

\subsection{Stochastic descent}

Recently,  \citet{gower2015randomized}  developed an iterative ``sketch and project'' framework for solving linear systems and quadratic optimization problems; see also \cite{gower2015stochastic} for extensions. In the context of problem \eqref{eq:quad_opt}, and specialized to sketching matrices with a single column, their method takes the form
\begin{equation} \label{alg_lin_sd} x_{t+1} =  x_t - \frac{s_t^\top (\mA x_t-b)}{s_t^\top \mA s_t} s_t, \end{equation}
where $s_t\in \R^n$ is a random vector sampled from some fixed distribution $\cD$.  In this paper we will refer to this method by the name  {\em stochastic descent (SD)}. 

Note that $x_{t+1}$ is obtained from $x_t$ by minimizing $f(x_t + h s_t)$ for $h\in \R$ and setting $x_{t+1} = x_t + h s_t$. Further, note that RCD arises as a special case with $\cD$ being a discrete probability distribution over the set $\{e_1,\dots,e_n\}$. However, SD converges for virtually any distribution $\cD$, including discrete and continuous distributions. In particular,  \citet{gower2015randomized} show  that as long as  $\Exp_{s\sim \cD} [\mH]$ is invertible, where $\mH \eqdef \frac{s s^\top}{s^\top \mA s} $, then SD converges as \begin{equation}\label{eq:rate-line-SD}\cO\left(\frac{1}{\lambda_{\min}(\mW)} \log \tfrac{1}{\epsilon}\right),\end{equation} where $\mW\eqdef \Exp_{s\sim \cD}[\mA^{1/2} \mH \mA^{1/2} ]$ (see Lemma~\ref{lem:rate_of_SD} for a more refined result due to  \citet{richtarik2017stochastic}). Rate of RCD in \eqref{eq:09hf09fs} can be obtained as a special case of \eqref{eq:rate-line-SD}.

\subsection{Stochastic spectral descent}

The starting point of this paper is the new observation that stochastic descent obtains the rate \begin{equation} \label{eq:rate-SSD}\cO\left(n \log \frac{1}{\epsilon}\right)\end{equation} in the special case when $\cD$ is chosen to be  the uniform distribution over the eigenvectors of $\mA$ (see Theorem~\ref{thm:SSD}). For obvious reasons, we refer to this new method as {\em stochastic spectral descent (SSD).} 

To the best of our knowledge, SSD was not explicitly considered in the literature before. We should note that SSD is fundamentally different from {\em spectral gradient descent}  \cite{spectralPGD,BB},  which refers to a family of gradient descent methods with a special choice of stepsize depending on the spectrum of the Hessian of  $f$.  

The rate \eqref{eq:rate-SSD} does not merely provide an  improvement on the rate of RCD given in \eqref{eq:09hf09fs}; what is  remarkable is that  this rate is completely independent of the properties (such as conditioning) of $\mA$. Moreover, we show that this method is {\em optimal} among the class of  stochastic descent methods \eqref{alg_lin_sd} parameterized by the choice of the distribution $\cD$ (see Theorem~\ref{thm:SSCD}).  Despite the attractiveness of its rate, SSD is not a practical method. This is because once we have the eigenvectors of $\mA$ available, the optimal solution $x_*$ can be assembled directly without the need for an iterative method.

\subsection{Stochastic conjugate descent}

We extend all results discussed above for SSD, including the rate \eqref{eq:rate-SSD}, to the more general class of methods we call {\em stochastic conjugate descent (SconD)}, for which $\cD$ is the uniform distribution over vectors $v_1,\dots,v_n$ which are  mutually {\em $\mA$ conjugate}: $v_i^\top \mA v_j = 0$ for $i\neq j$ and  $v_i^\top \mA v_i = 1$.

{\small
\begin{table*}[t]
\centering
\begin{tabular}{|l|c|}
\hline
Result & Theorem \\  
\hline
Uniform probabilities are optimal for $n=2$ &  \ref{thm:n=2}\\
Uniform probabilities are optimal for any $n\geq 2$ as long as  $\mA$ is diagonal & \ref{thm:unif_prob_can_be_opt}\\ 
``Importance sampling'' $p_i\sim \mA_{ii}$ can lead to an arbitrarily worse rate than uniform probabilities & \ref{thm:imp_prob_can_be_bad} \\
``Importance sampling'' $p_i\sim \|\mA_{i:}\|^2$ can lead to an arbitrarily worse rate than uniform probabilities & \ref{thm:imp_prob_can_be_bad} \\
For every $n\geq 2$ and  $T>0$, there is $\mA$ such that the rate of RCD with optimal probabilities is $\cO(T \log \tfrac{1}{\epsilon})$ & \ref{thm:opt_probs_are_bad_UPPER} \\
For every $n\geq 2$ and  $T>0$, there is $\mA$ such that the rate of RCD with optimal probabilities is $\Omega(T \log \tfrac{1}{\epsilon})$   & \ref{thm:opt_probs_are_bad_LOWER} \\
\hline
\end{tabular}
\caption{Summary of results on importance and optimal sampling in RCD.}
\label{tbl:rcd_results}
\end{table*}
}

\subsection{Optimizing probabilities in RCD}

The idea of speeding up RCD via the use of non-uniform probabilities was pioneered by   \citet{nesterov2012efficiency} in the context of smooth convex minimization, and later built on by many authors \cite{NSYNC, ALPHA,  NU_ACDM}. In the case of non-accelerated RCD, and in the context of smooth convex  optimization, the most popular choice of probabilities is to set $p_i \sim L_i$, where $L_i$ is the Lipschitz constant of the gradient of the objective corresponding to coordinate $i$ \cite{nesterov2012efficiency, NSYNC}.  For  problem \eqref{eq:quad_opt}, we have  $L_i = \mA_{ii}$.  \citet{gower2015randomized} showed that the optimal probabilities for \eqref{eq:quad_opt} can in principle be computed  through semidefinite programming (SDP); however, no theoretical properties of the optimal solution of the SDP were given. 


As a warm-up, we first ask the following question: how important is importance sampling? More precisely, we investigate  RCD with probabilities $p_i\sim \mA_{ii}$, and RCD with probabilities $p_i\sim \|\mA_{i:}\|^2$, considered as RCD with ``importance sampling'', and compare these with the baseline RCD with uniform probabilities. 
Our result (see Theorem~\ref{thm:imp_prob_can_be_bad}) contradicts conventional ``wisdom''. In particular, we show that for every $n$ there is a matrix $\mA$ such that diagonal probabilities lead to the best rate. Moreover, the rate of RCD with ``importance'' can be arbitrarily worse than the rate of RCD with uniform probabilities. The same result applies to probabilities proportional to the square of the norm of the $i$th row of $\mA$.

We then switch gears, and motivated by the nature of SSD,  we ask the following question:  in order to obtain a condition-number-independent rate such as \eqref{eq:rate-SSD}, do we {\em have to} consider new (and hard to compute) descent directions, such as eigenvectors of $\mA$,  or can a similar effect be obtained using RCD with a better selection of probabilities? We give two negative results to this question (see Theorems~\ref{thm:opt_probs_are_bad_UPPER} and \ref{thm:opt_probs_are_bad_LOWER}). First, we show that for any $n\geq 2$ and any $T>0$, there is a matrix $\mA$ such that the rate of RCD with {\em any probabilities} (including the optimal probabilities) is $\cO(T \log \tfrac{1}{\epsilon})$. Second, we give a similar but much stronger statement where we reach the same conclusion, but for the {\em  lower bound} as opposed to the upper bound. That is,  $\cO$ is replaced by $\Omega$.

As a by-product of our investigations into importance sampling, we establish that for $n=2$, {\em uniform probabilities} are optimal for all matrices $\mA$ (see Theorem~\ref{thm:n=2}).  For a summary of all these results, see Table~\ref{tbl:rcd_results}.

\subsection{Interpolating between RCD and SSD}\label{sec:ibns98g9db9((}

{\scriptsize
\begin{table*}[t]
\centering
\begin{tabular}{|l|c|c|c|}
\hline
&  general spectrum & 
\begin{tabular}{c} $n-k$ largest eigvls are $\gamma$-clustered \\
$c \leq \lambda_i \leq \gamma c$ for $k+1 \leq  i \leq n$ 
 \end{tabular} 
  & $\alpha$-exp decaying eigvls \\
\hline
RCD  ($p_i\sim \mA_{ii}$)             &  $\tilde{\cO}\left(\frac{\sum_i \lambda_i}{\lambda_1}\right)$ & $\tilde{\cO}\left(\frac{\gamma n c}{\lambda_1}\right)$  & $\tilde{\cO}\left(\frac{1}{\alpha^{n-1}}\right)$ \\
SSCD  & $\tilde{\cO}\left( \frac{(k+1)\lambda_{k+1} + \sum_{i=k+2}^n \lambda_i}{\lambda_{k+1}}\right)$ &  $\tilde{\cO}\left(  \gamma n \right)$ &  $\tilde{\cO}\left(\frac{1}{\alpha^{n-k-1}}\right)$ \\
SSD    & $\tilde{\cO}(n)$ & $\tilde{\cO}(n)$& $\tilde{\cO}(n)$\\
\hline
\end{tabular}
\caption{Comparison of complexities of RCD, SSCD (with parameter $0\leq k \leq n-1$) and SSD under various regimes on the spectrum of $\mA$. The $\tilde{\cO}$ notation supresses a $\log \tfrac{1}{\epsilon}$ term.}
\label{tbl:regimes}
\end{table*}
}

RCD and SSD lie on opposite ends of a continuum of stochastic descent methods for solving \eqref{eq:quad_opt}. RCD ``minimizes'' the work per iteration without any regard for the number of iterations, while SSD minimizes the number of iterations  without any regard for the cost per iteration (or pre-processing cost). Indeed, one step of RCD costs $\cO(\|\mA_{i:}\|_0)$ (the number of nonzero entries in the $i$th row of $\mA$), and hence RCD can be implemented very efficiently for sparse $\mA$. If uniform probabilities are used, no pre-processing (for computing probabilities) is needed. These advantages are paid for by the rate \eqref{eq:09hf09fs}, which can be arbitrarily high. On the other hand, the rate of SSD does not depend on $\mA$. This advantage is paid for by a high pre-processing cost: the computation of the eigenvectors.  This pre-processing cost makes the method utterly impractical.

One of the main contributions of this paper is the development of a new {\em parametric family of algorithms that in some sense interpolate between RCD and SSD}.  

In particular, we consider the stochastic descent algorithm \eqref{alg_lin_sd} with $\cD$  being a discrete distribution over the search directions $\{e_1,\dots,e_n \} \cup \{u_1,\dots, u_k\}$, where $u_i$ is the eigenvectors of $\mA$ corresponding to the $i$th smallest eigenvalue of $\mA$. We refer to this new method by the name {\em stochastic spectral coordinate descent (SSCD)}.  

We compute the optimal probabilities of this distribution, which turn out to be unique, and show that  for $k\geq 1$ they depend on the $k+1$ smallest eigenvalues of $\mA$: $0<\lambda_1\leq \lambda_2 \leq \cdots \leq \lambda_{k+1}$. In particular, we prove (see Theorem~\ref{thm:SSCD}) that the rate of SSCD with optimal probabilities is
\begin{equation}\label{eq:rate-line-SSCD}\cO\left( \frac{(k+1)\lambda_{k+1} + \sum_{i=k+2}^n \lambda_i}{\lambda_{k+1}} \log \tfrac{1}{\epsilon}\right) .\end{equation}
For $k=0$, SSCD reduces to RCD with $p_i\sim \mA_{ii}$, and the rate \eqref{eq:rate-line-SSCD} reduces to \eqref{eq:09hf09fs}. For $k=n-1$,  SSCD {\em does not} reduce to SSD. However, the rates match. Indeed, in this case the rate \eqref{eq:rate-line-SSCD} reduces to \eqref{eq:rate-SSD}. Moreover, the rate improves monotonically as $k$ increases, from $\cO(\tfrac{{\rm Tr}(\mA)}{\lambda_{\min}(\mA)} \log \tfrac{1}{\epsilon})$ (for $k=0$) to $\cO(n \log \tfrac{1}{\epsilon})$ (for $k=n-1$).

\vspace{-.5em}
\paragraph{SSCD removes the effect of the $k$ smallest eigenvalues.}
Note that the rate \eqref{eq:rate-line-SSCD} does {\em not depend} on the $k$ smallest eigenvalues of $\mA$. That is, by adding the eigenvectors $u_1,\dots,u_k$ corresponding to the $k$ smallest eigenvalues to the set of descent directions, we have removed the effect of these eigenvalues.

\vspace{-.5em}
\paragraph{Clustered eigenvalues.}
Assume that the $n-k$ largest eigenvalues are clustered: $c \leq \lambda_i \leq \gamma c$  for some $c>0$ and $\gamma>1$, for all $k+1 \leq i \leq n$. In this case, the rate \eqref{eq:rate-line-SSCD} can be estimated as a function of the clustering ``tightness'' parameter $\gamma$:
$\cO\left(  \gamma n \log \tfrac{1}{\epsilon}\right). $ See Table~\ref{tbl:regimes}.

This can be arbitrarily better than the rate of RCD, even for $k=1$.
In other words, there are situations where by enriching the set of directions used by RCD by a single eigenvector only, the resulting method accelerates dramatically. To give  a concrete and simplified example to illustrate this, assume that $\lambda_1=\delta>0$, while $\lambda_2=\cdots = \lambda_n = 1$. In this case, RCD has the rate $\cO((1+\tfrac{n-1}{\delta})\log \tfrac{1}{\epsilon})$, while SSCD with $k=1$ has the rate $\cO(n \log \tfrac{1}{\epsilon})$. So, SSCD is $\tfrac{1}{\delta}$ times better than RCD, and the difference grows to infinity as $\delta$ approaches zero even for fixed dimension $n$. 

\vspace{-0.5em}
\paragraph{Exponentially decaying eigenvalues.} If the eigenvalues of $\mA$ follow an exponential decay with factor $0<\alpha<1$, then the rate of RCD is $\cO(\tfrac{1}{\alpha^{n-1}} \log \tfrac{1}{\epsilon})$, while the rate of SSCD is $\cO(\tfrac{1}{\alpha^{n-k-1}} \log \tfrac{1}{\epsilon})$. This is an improvement by the factor $\tfrac{1}{\alpha^k}$, which can be very large even for small $k$ if $\alpha$ is small. See Table~\ref{tbl:regimes}. For an experimental confirmation of this prediction, see Figure~\ref{fig:expdecay}.

\vspace{-0.5em}
\paragraph{Adding a few ``largest'' eigenvectors does not help.} We show that in contrast with the situation above, adding a few of the ``largest'' eigenvectors to the coordinate directions of RCD does not help. This is captured formally in the appendix as Theorem~\ref{thm:last_eigs}.




\vspace{-0.5em}
\paragraph{Mini-batching.} We extend  SSCD to a mini-batch setting; we call the new method {\em mSSCD}. We show that the rate of mSSCD interpolates between the rate of mini-batch RCD and rate of SSD. Moreover, we show that  mSSCD is optimal among a certain parametric family of methods,  and that its rate improves as $k$ increases. See Theorem~\ref{thm:Par_SSCD}.

\subsection{Inexact Directions}
Finally, we relax the need to compute exact eigenvectors or $\mA$- conjugate vectors, and analyze the behavior of our methods for inexact directions. Moreover, we propose and analyze an inexact variant of SSD which does {\em not} arise as a special case of SD. See Sections~\ref{sec:inexact_methods} and \ref{sec:iSSD}.

\section{Stochastic Descent}\label{sec:SD}

The stochastic descent method was described in \eqref{alg_lin_sd}. We now  formalize it as Algorithm~\ref{alg:SD}, and equip it with a stepsize, which will be useful in  Section~\ref{sec:90hs0909ff}, where we study mini-batch version of SD.

\begin{algorithm}[h]
   \caption{Stochastic  Descent (SD)}
   \label{alg:SD}
\begin{algorithmic}
   \STATE {\bfseries Parameters:} Distribution $\cD$; Stepsize parameter $\omega>0$
   \STATE {\bfseries Initialize:} Choose $x_0 \in \R^n$
   \FOR{$t=0,1,2,\dots$}
  \STATE Sample search direction $s_t\sim \cD$
  \STATE Set $x_{t+1}=x_t-\omega\frac{ s_t^\top (\mA x_t-b)}{s_t^\top \mA s_t} s_t$
   \ENDFOR
\end{algorithmic}
\end{algorithm}

In order to guarantee convergence of SD, we  restrict our attention to the class of {\em proper} distributions, defined next.

\begin{assumption} \label{ass:regular} Distribution $\cD$ is {\em proper} with respect to $\mA$. That is, $\Exp_{s\sim \cD}[\mH]$ is invertible, where \begin{equation}\label{eq:H} \mH   \eqdef \frac{s s^\top}{s^\top \mA s}.\end{equation}
\end{assumption}
Next we present the main convergence result for SD.

\begin{lemma}[Convergence of stochastic descent \cite{gower2015randomized,richtarik2017stochastic}] \label{lem:rate_of_SD}
Let $\cD$ be proper with respect to $\mA$, and let $0<\omega<2$. Stochastic descent (Algorithm~\ref{alg:SD}) converges linearly  in expectation. In particular, we have
\begin{equation}\label{eq:rate_lower}(1-\omega(2-\omega)\lambda_{\max}(\mW))^t \|x_0-x_*\|_{\mA}^2 \leq \Exp[\|x_t-x_*\|_{\mA}^2] \end{equation}
and
\begin{equation}\label{eq:rate_upper} \Exp[\|x_t-x_*\|_{\mA}^2] \leq (1-\omega(2-\omega)\lambda_{\min}(\mW))^t \|x_0-x_*\|_{\mA}^2,\end{equation}
where
\begin{equation}\label{eq:W}\mW\eqdef  \mathbb{E}_{s \sim\mathcal{D}}[\mA^{1/2}\mH \mA^{1/2} ].\end{equation}
Finally, the statement remains true if we replace $\|x_t-x_*\|_\mA^2$ by $f(x_t)-f(x_*)$ for all $t$.
\end{lemma}

It is easy to observe that the stepsize choice $\omega=1$ is optimal. This is why we have decided to present the SD method  \eqref{alg_lin_sd} with this choice of stepsize. Moreover, notice that due to linearity of expectation, \begin{eqnarray*}{\rm Tr}(\mW) &\overset{\eqref{eq:W}}{=}&   \Exp[{\rm Tr}(\mA^{1/2}\mH\mA^{1/2})]\\ & \overset{\eqref{eq:H}}{=}&  \Exp\left[{\rm Tr}\left(\frac{z z^\top }{z^\top z} \right)\right] \\
&=& \Exp\left[{\rm Tr}\left(\frac{z^\top z  }{z^\top z} \right)\right]\\
&  =& 1,\end{eqnarray*}
where $z = \mA^{1/2}s$. Therefore, \[0 < \lambda_{\min}(\mW) \leq  \frac{1}{n} \leq \lambda_{\max}(\mW) \leq 1.\]

\subsection{Stochastic Spectral Descent}\label{sec:SSD}

Let $\mA  = \sum_{i=1}^{n} \lambda_i u_i u_i^\top$ be the eigenvalue decomposition of $\mA$. That is, 
$
0<\lambda_1 \leq \lambda_2 \leq \ldots \leq \lambda_n$ are the eigenvalues of $\mA$ and $u_1,\dots,u_n$ are the corresponding orthonormal eigenvectors. Consider now the SD method with $\cD$ being the uniform distribution over the set $\{u_1,\dots,u_n\}$, and $\omega=1$. This gives rise to a new variant of SD which we call  {\em stochastic spectral descent (SSD)}.

\begin{algorithm}[h]
   \caption{Stochastic Spectral Descent (SSD)}
   \label{alg:SSD}
\begin{algorithmic}
   \STATE {\bfseries Initialize:} $x_0 \in \R^n$; $(u_1,\lambda_1), \dots (u_n,\lambda_n)$: eigenvectors and eigenvalues of $\mA$
   \FOR{$t=0,1,2,\dots$}
  \STATE Choose $i\in [n]$ uniformly at random
  \STATE Set $x_{t+1} = x_t - \left(u_i^\top x_t - \frac{u_i^\top b}{ \lambda_i}\right) u_i$
   \ENDFOR
\end{algorithmic}
\end{algorithm}

For SSD we can establish an unusually  strong convergence result, both in terms of speed and tightness.

\begin{theorem}[Convergence of stochastic spectral descent] \label{thm:SSD} Let $\{x_k\} $ be the sequence of random iterates produced by stochastic spectral descent (Algorithm~\ref{alg:SSD}). Then \begin{equation}\label{eq:ug9fg09si8} \Exp[\|x_t-x_*\|_{\mA}^2  ] = \left(1-\frac{1}{n}\right)^t \|x_0-x_*\|_{\mA}^2.\end{equation}
\end{theorem}

The above theorem implies the rate \eqref{eq:rate-SSD} mentioned in the introduction. It means that up to a  logarithmic factor, SSD only needs $n$ iterations to converge. Notice that \eqref{eq:ug9fg09si8}  is an {\em identity}, and hence the rate is not improvable.

\subsection{Stochastic Conjugate Descent} \label{sec:8ys089h0df}

The same rate as in  Theorem~\ref{thm:SSD}  holds for the {\em stochastic conjugate descent} (SconD) method, which arises as a special case of stochastic descent for $\omega=1$ and $\cD$ being a uniform distribution over a set of $\mA$-orthogonal (i.e., conjugate) vectors. The proof follows by combining  Lemmas~\ref{lem:rate_of_SD} and \ref{lemma:A_orthogonal_basic_method}.

\subsection{Randomized Coordinate Descent} \label{sec:opt_prob}

RCD (Algorithm~\ref{alg:RCD}) arises as a special case of SD with unit stepsize ($\omega=1$) and distribution $\cD$ given by $s_t=e_i$ with probability $p_i> 0$. 

\begin{algorithm}[h]
   \caption{Randomized Coordinate Descent (RCD)}
   \label{alg:RCD}
\begin{algorithmic}
   \STATE {\bfseries Parameters:}    probabilities $p_1,\dots,p_n>0$
   \STATE {\bfseries Initialize:}  $x_0 \in \R^n$
   \FOR{$t=0,1,2,\dots$}
  \STATE Choose $i \in [n]$ with probability $p_i>0$
  \STATE Set $x_{t+1}=x_t-\frac{\mA_{i:} x_t-b_i}{\mA_{ii}} e_i$
   \ENDFOR
\end{algorithmic}
\end{algorithm}

The rate of RCD (Algorithm~\ref{alg:RCD}) can therefore be deduced from Lemma~\ref{lem:rate_of_SD}. Notice that in view of \eqref{eq:H}, we have \[\Exp[\mH] =  \sum_{i=1}^n p_i \frac{e_i e_i^\top }{\mA_{ii}} = {\rm Diag}\left(\frac{p_1}{\mA_{11}},\dots, \frac{p_n}{\mA_{nn}}\right).\] So, as long as all probabilities are positive,  Assumption~\ref{ass:regular} is satisfied. Therefore, Lemma~\ref{lem:rate_of_SD} applies and RCD enjoys the rate
\begin{equation}\label{eq:rate-line-RCD-p}\cO\left( \frac{1}{\lambda_{\min}\left(\mA {\rm Diag}\left(\tfrac{p_i}{\mA_{ii}}\right)\right)}\log \frac{1}{\epsilon}\right).\end{equation}

\paragraph{Uniform probabilities can be optimal.}

We first prove that uniform probabilities are optimal  in 2D. 

\begin{theorem} \label{thm:n=2} Let $n=2$ and consider RCD  (Algorithm~\ref{alg:RCD}) with probabilities $p_1> 0$ and $p_2> 0$, $p_1+p_2=1$.  Then the choice $p_1=p_2=\tfrac{1}{2}$ optimizes the rate of RCD in \eqref{eq:rate-line-RCD-p}. 
\end{theorem}

Next we claim that uniform probabilities are optimal in any dimension $n$  as long as the matrix $\mA$ is diagonal.

\begin{theorem}\label{thm:unif_prob_can_be_opt} Let $n\geq 2$ and let $\mA$ be diagonal. Then uniform probabilities ($p_i=\tfrac{1}{n}$ for all $i$) optimize the rate  of RCD in \eqref{eq:rate-line-RCD-p}. 
\end{theorem}

\paragraph{``Importance'' sampling can be unimportant.}

In our next result we contradict conventional wisdom about typical choices of ``importance sampling'' probabilities. In particular, we claim that diagonal and row-squared-norm probabilities can lead to an arbitrarily worse performance than uniform probabilities.

\begin{theorem}\label{thm:imp_prob_can_be_bad} For every $n\geq 2$ and  $T>0$, there exists $\mA$ such that: (i) The rate of RCD with $p_i\sim \mA_{ii}$ is $T$ times worse than the rate of RCD with uniform probabilities. (ii)
 The rate of RCD with $p_i\sim \|\mA_{i:}\|^2$ is $T$ times worse than the rate of RCD with uniform probabilities.

\end{theorem}

\paragraph{Optimal probabilities can be bad.}

Finally, we show that there is no hope for adjustment of probabilities in RCD to lead to a rate independent of the data $\mA$, as is the case for SSD. Our first result states that such a result can't be obtained from the generic rate \eqref{eq:rate-line-RCD-p}.

\begin{theorem} \label{thm:opt_probs_are_bad_UPPER}  For every $n\geq 2$ and  $T>0$, there exists  $\mA$ such that the number of iterations (as expressed by formula \eqref{eq:rate-line-RCD-p}) of RCD with any choice of probabilities $p_1,\dots,p_n>0$ is  $\cO(T \log(1/\epsilon))$.
\end{theorem}

However, that does not mean, by itself, that such a result can't be possibly obtained via a different analysis. Our next result shatters these hopes as we establish a {\em lower bound} which can be arbitrarily larger than the dimension $n$.

\begin{theorem} \label{thm:opt_probs_are_bad_LOWER}  For every $n\geq 2$ and  $T>0$, there exists an $n\times n$ positive definite matrix $\mA$ and starting point $x_0$, such that the number of iterations of RCD with any choice probabilities $p_1,\dots,p_n>0$ is  $\Omega(T \log(1/\epsilon))$.
\end{theorem}

\section{Interpolating Between RCD and SSD}

Assume now that we have some partial spectral information available. In particular, fix $k\in \{0,1,\dots,n-1\}$ and assume we know eigenvectors $u_i$ and eigenvalues $\lambda_i$ for $i=1,\dots,k$.   We now define a parametric distribution $\cD(\alpha,\beta_1,\dots,\beta_k)$ with parameters   $\alpha>0$ and $ \beta_1,\dots, \beta_k \geq 0$ as follows. Sample
$s\sim \cD(\alpha,\beta_1,\dots,\beta_k)$ arises through the process
\begin{equation}\label{fam:family_interp_methods}
  s = \begin{cases}
e_i & \text{with probability\;}  p_i = \frac{\alpha \mA_{ii}}{C_k}, \; i \in [n],\\
u_i & \text{with probability\;} p_{n+i} = \frac{\beta_i}{C_k}, \;  i \in [k],
\end{cases}
\end{equation}
where $C_k \eqdef \alpha {\rm Tr}(\mA) + \sum_{i=1}^k \beta_i$ is a normalizing factor ensuring that the probabilities sum up to 1.

\subsection{SSCD} \label{sec:iugd8998ds}

Applying the SD method with the distribution $\cD=\cD(\alpha,\beta_1,\dots,\beta_k)$ gives rise to a new specific method which we call {\em stochastic spectral coordinate descent (SSCD)}.

\begin{algorithm}[h]
   \caption{Stochastic Spectral  Coordinate Descent (SSCD)}
   \label{alg:SSCD}
\begin{algorithmic}
    \STATE {\bfseries Parameters:} Distribution $\cD(\alpha,\beta_1,\dots,\beta_k)$
   \STATE {\bfseries Initialize:} $x_0 \in \R^n$
   \FOR{$t=0,1,2,\dots$}
  \STATE Sample $s_t \sim \cD(\alpha,\beta_1,\dots,\beta_k)$
  \STATE Set $x_{t+1}=x_t- \frac{s_t^\top (\mA x_t-b)}{s_t^\top \mA s_t} s_t$
   \ENDFOR
\end{algorithmic}
\end{algorithm}


\begin{theorem} \label{thm:SSCD} Consider Stochastic Spectral Coordinate Descent (Algorithm~\ref{alg:SSCD}) for fixed $k\in \{0,1,\dots,n-1\}$. The method converges linearly for all positive $\alpha>0$ and nonnegative $\beta_i$. The best rate is obtained for parameters $\alpha=1$ and $\beta_i = \lambda_{k+1}-\lambda_i$; and this is the unique choice of parameters leading to the best rate. In this case,
\[ \Exp[\|x_t-x_*\|_{\mA}^2  ] \leq \left(1-\frac{\lambda_{k+1}}{C_k} \right)^t \|x_0-x_*\|_{\mA}^2,\]
where \[C_k =  (k+1) \lambda_{k+1} + \sum_{i={k+2}}^n \lambda_i.\] Moreover, the rate improves as $k$ grows, and we have
\[\frac{\lambda_1}{{\rm Tr}(\mA)} = \frac{\lambda_1}{C_0} \leq  \cdots \leq \frac{\lambda_{k+1}}{C_k} \leq \cdots \leq \frac{\lambda_n}{C_{n-1}} = \frac{1}{n}.\]
\end{theorem}

If $k=0$, SSCD reduces to  RCD (with diagonal probabilities). Since $\tfrac{\lambda_{1}}{C_0} = \tfrac{\lambda_1}{{\rm Tr}(\mA)}$, we recover the rate of RCD of \citet{leventhal2010}. With the choice $k=n-1$ our method does {\em not} reduce to SSD. However, the rates match. Indeed, $\tfrac{\lambda_n}{C_{n-1}} = \tfrac{\lambda_n}{n \lambda_n} = \tfrac{1}{n}$ (compare with Theorem~\ref{thm:SSD}).

\vspace{-0.8em}
\paragraph{``Largest'' eigenvectors do not help.}
It is natural to ask whether there is any benefit in considering  a few ``largest'' eigenvectors instead. Unfortunately, for  the same parametric family as in Theorem~\ref{thm:SSCD}, the answer is negative. The optimal parameters suggest that RCD has better rate without these directions.
See Theorem~\ref{thm:last_eigs} in the appendix.

\subsection{Mini-batch SD} \label{sec:90hs0909ff}

A mini-batch version of SD was developed by \citet{richtarik2017stochastic}. Here we restate the method as  Algorithm~\ref{alg:Par_SD}. 

\begin{algorithm}[h]
	\caption{Mini-batch Stochastic Descent (mSD)}
	\label{alg:Par_SD}
	\begin{algorithmic}
		\STATE {\bfseries Parameters:} Distribution $\cD$; stepsize parameter $\omega >0$; mini-batch size  $\tau\geq 1$
		\STATE {\bfseries Initialize:} $x_0 \in \R^n$
		\FOR{$t=0,1,2,\dots$}
		\FOR{$i = 1,2,\dots,\tau$}
		\STATE Sample $s_{ti} \sim \cD$
		\STATE Set $x_{t+1,i}=x_t- \omega \frac{s_{ti}^\top (\mA x_t-b)}{s_{ti}^\top \mA s_{ti}} s_{ti}$
		\ENDFOR
		\STATE Set $x_{t+1}=\frac{1}{\tau}\sum\limits_{i=1}^\tau x_{t+1,i}$
		\ENDFOR
	\end{algorithmic}
\end{algorithm}

\begin{lemma}[Convergence of mSD \cite{richtarik2017stochastic}] \label{lem:rate_of_parallel_SD}
	Let $\cD$ be proper with respect to $\mA$, and let $0<\omega<\frac{2}{\xi(\tau)}$, where $\xi(\tau) \eqdef \frac{1}{\tau}+\left(1-\frac{1}{\tau}\right)\lambda_{\max}(\mW)$. Then
	 \begin{equation}\label{eq:par_sd_rate_upper} \Exp[\|x_t-x_*\|_{\mA}^2] \leq \left(\rho(\omega,\tau)\right)^t \|x_0-x_*\|_{\mA}^2,\end{equation}
	where \[\rho(\omega,\tau) = 1 - \omega[2-\omega\xi(\tau)]\lambda_{\min}(\mW).\] For any fixed $\tau \geq 1$, the optimal stepsize choice is $\omega(\tau) = \frac{1}{\xi(\tau)}$ and the associated optimal rate is
\[
		\rho(\omega(\tau),\tau) = 1 - \frac{\lambda_{\min}(\mW)}{\frac{1}{\tau}+\left(1-\frac{1}{\tau}\right)\lambda_{\max}(\mW)}.
\]
\end{lemma}

\subsection{Mini-batch SSCD}\label{sec:Par_SSCD}

Specializing mSD to the distribution $\cD=\cD(\alpha,\beta_1,\dots,\beta_k)$ gives rise to a new specific method which we call {\em mini-batch stochastic spectral coordinate descent (mSSCD)}, and formalize as Algorithm~\ref{alg:Par_SSCD}. 

\begin{algorithm}[h]
	\caption{Mini-batch Stochastic Spectral Coordinate Descent (mSSCD)}
	\label{alg:Par_SSCD}
	\begin{algorithmic}
		\STATE {\bfseries Parameters:} Distribution $\cD(\alpha,\beta_1,\dots,\beta_k)$; relaxation parameter $\omega\in \R$; mini-batch size  $\tau\geq 1$ 
		\STATE {\bfseries Initialize:} $x_0 \in \R^n$
		\FOR{$t=0,1,2,\dots$}
		\FOR{$i = 1,2,\dots,\tau$}
		\STATE Sample $s_{ti} \sim \cD(\alpha,\beta_1,\dots,\beta_k)$
		\STATE Set $x_{t+1,i}=x_t- \omega \frac{s_{ti}^\top (\mA x_t-b)}{s_{ti}^\top \mA s_{ti}}  s_{ti}$
		\ENDFOR
		\STATE Set $x_{t+1}=\frac{1}{\tau}\sum\limits_{i=1}^\tau x_{t+1,i}$
		\ENDFOR
	\end{algorithmic}
\end{algorithm}

The rate of mSSCD is governed by the following result.

\begin{theorem} \label{thm:Par_SSCD} Consider mSSCD (Algorithm~\ref{alg:Par_SSCD}) for fixed $k\in \{0,1,\dots,n-1\}$ and optimal stepsize parameter $\omega(\tau) = \frac{1}{\xi(\tau)}$. The method converges linearly for all positive $\alpha>0$ and nonnegative $\beta_i$. The best rate is obtained for parameters $\alpha=1$ and $\beta_i = \lambda_{k+1}-\lambda_i$; and this is the unique choice of parameters leading to the best rate. In this case,
	\[ \Exp[\|x_t-x_*\|_{\mA}^2  ] \leq \left(1-\frac{\lambda_{k+1}}{F_k} \right)^t \|x_0-x_*\|_{\mA}^2,\]
	where \[F_k \eqdef  \frac{1}{\tau}\left((k+1) \lambda_{k+1} + \sum_{i={k+2}}^n \lambda_i\right) + \left(1-\frac{1}{\tau}\right)\lambda_n.\] Moreover, the rate improves as $k$ grows, and we have
\[
		\frac{\lambda_1}{\frac{1}{\tau}{\rm Tr}(\mA)+\left(1-\tfrac{1}{\tau}\right)\lambda_n} = \frac{\lambda_1}{F_0} \leq   \cdots 
\leq \frac{\lambda_{k+1}}{F_k} \]
and
\[\frac{\lambda_{k+1}}{F_k}
 \leq  \cdots
		\leq  \frac{\lambda_n}{F_{n-1}} 
		= \frac{1}{\frac{n-1}{\tau}+1}.
\]
\end{theorem}

If $k=0$, mSSCD reduces to mini-batch RCD (with diagonal probabilities). Since $\tfrac{\lambda_{1}}{F_0} = \tfrac{\lambda_1}{\frac{1}{\tau}{\rm Tr}(\mA)+\left(1-\frac{1}{\tau}\right)\lambda_n}$, we recover the rate of mini-batch RCD \cite{richtarik2017stochastic}. With the choice $k=n-1$ our method does {\em not} reduce to mSSD. However, the rates match.

\begin{figure*}[t!]
	\centering
	\subfloat{\includegraphics[width=0.25\linewidth]{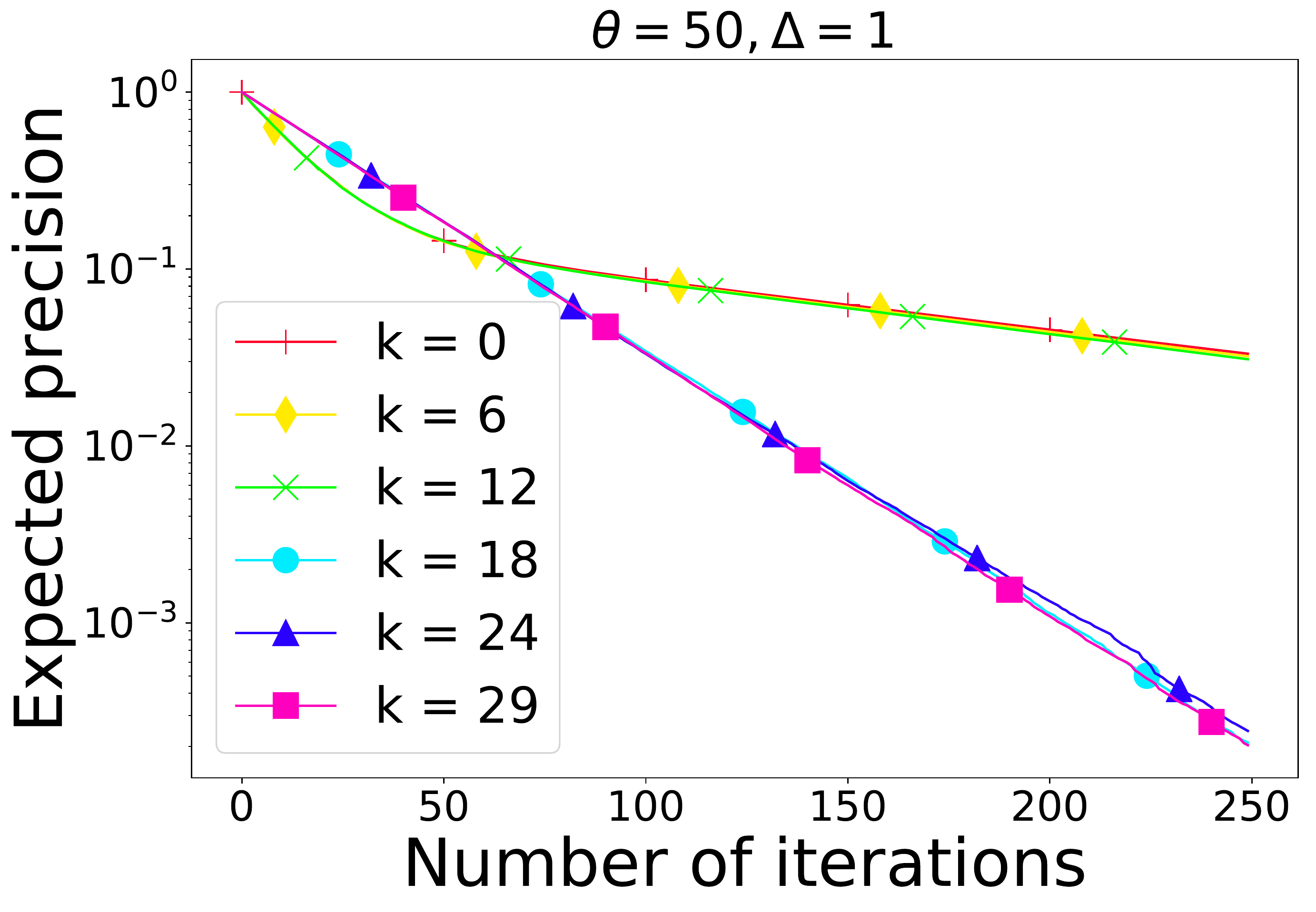}}
	\subfloat{\includegraphics[width=0.25\linewidth]{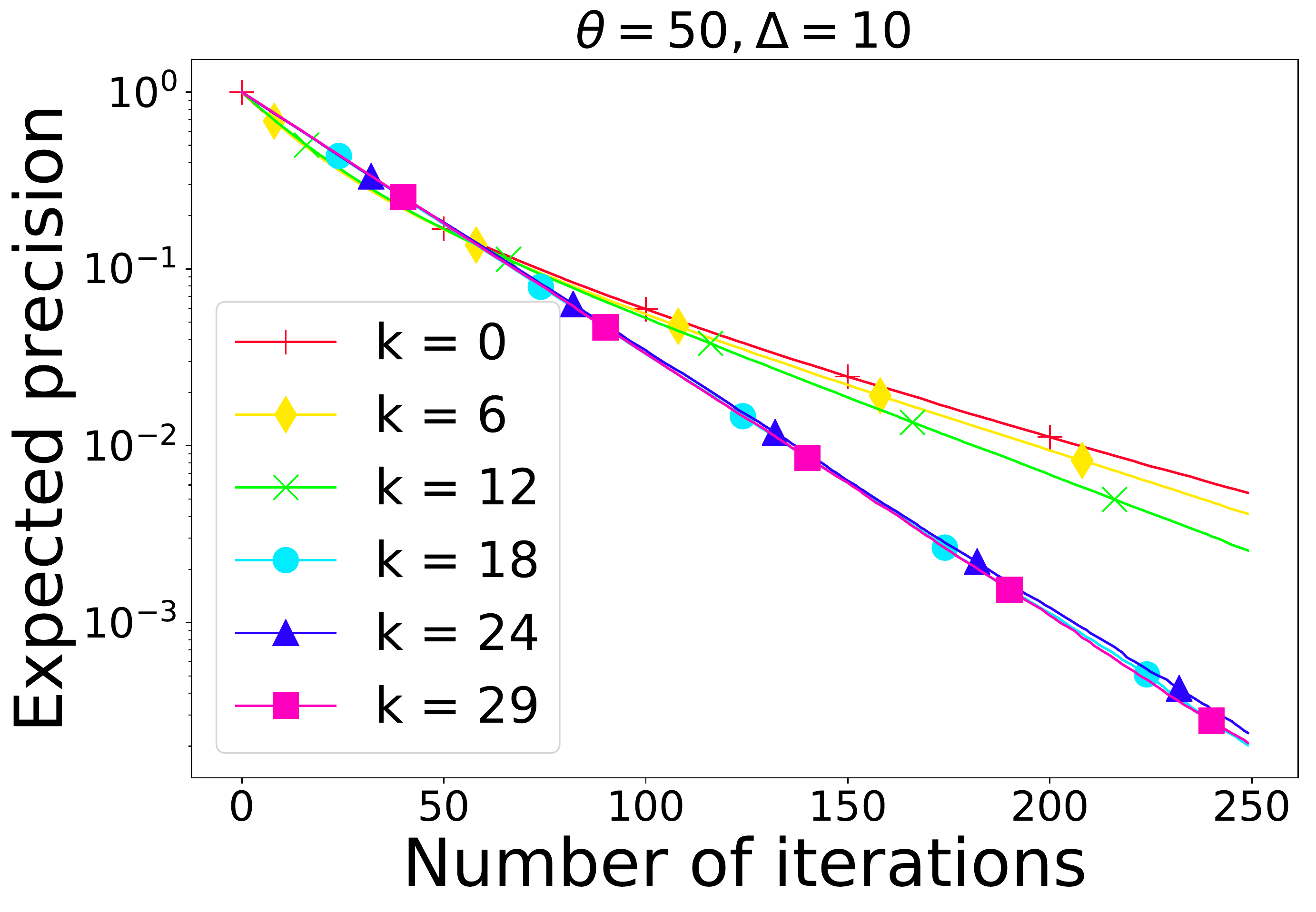}}
	\subfloat{\includegraphics[width=0.25\linewidth]{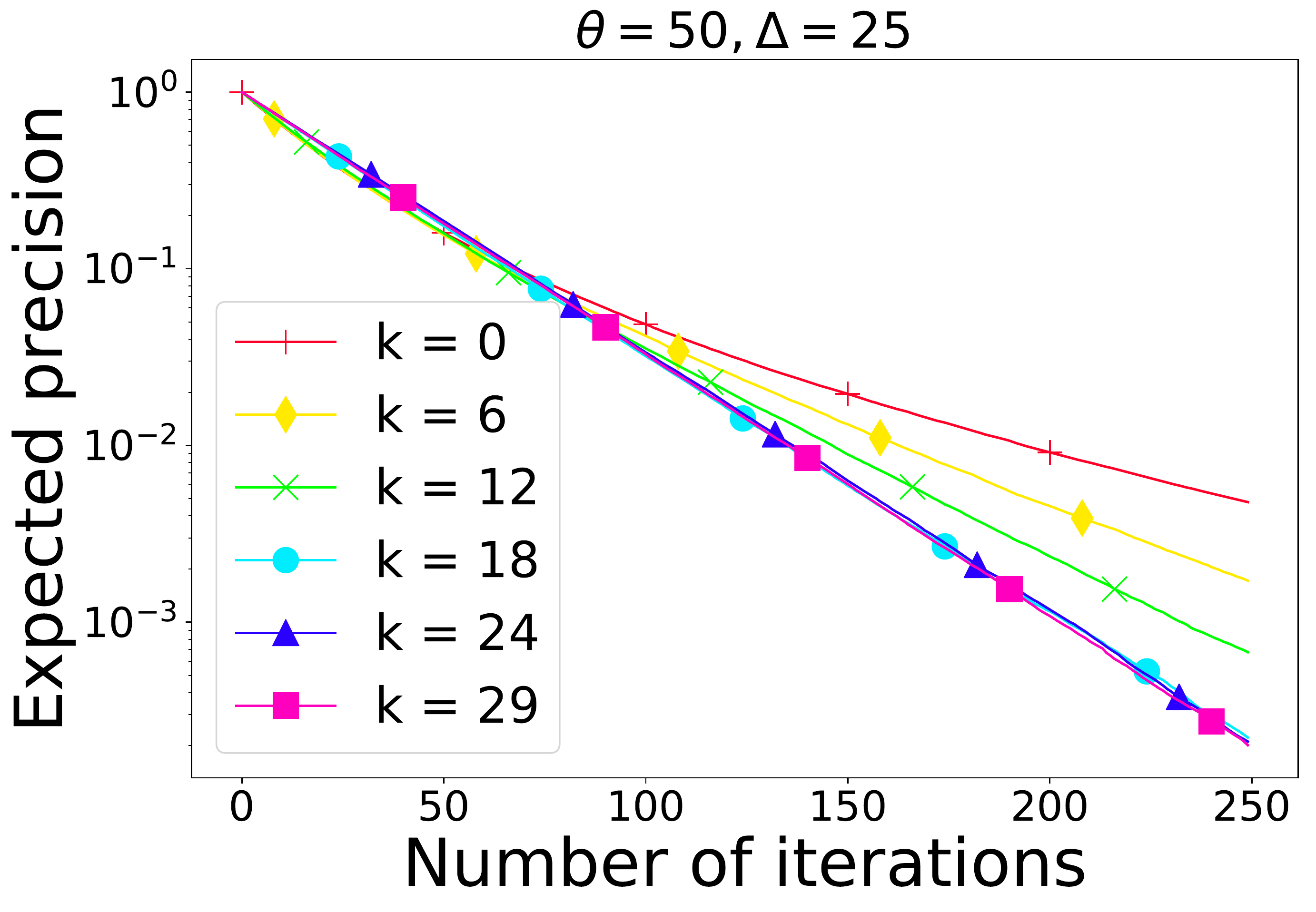}}
	\subfloat{\includegraphics[width=0.25\linewidth]{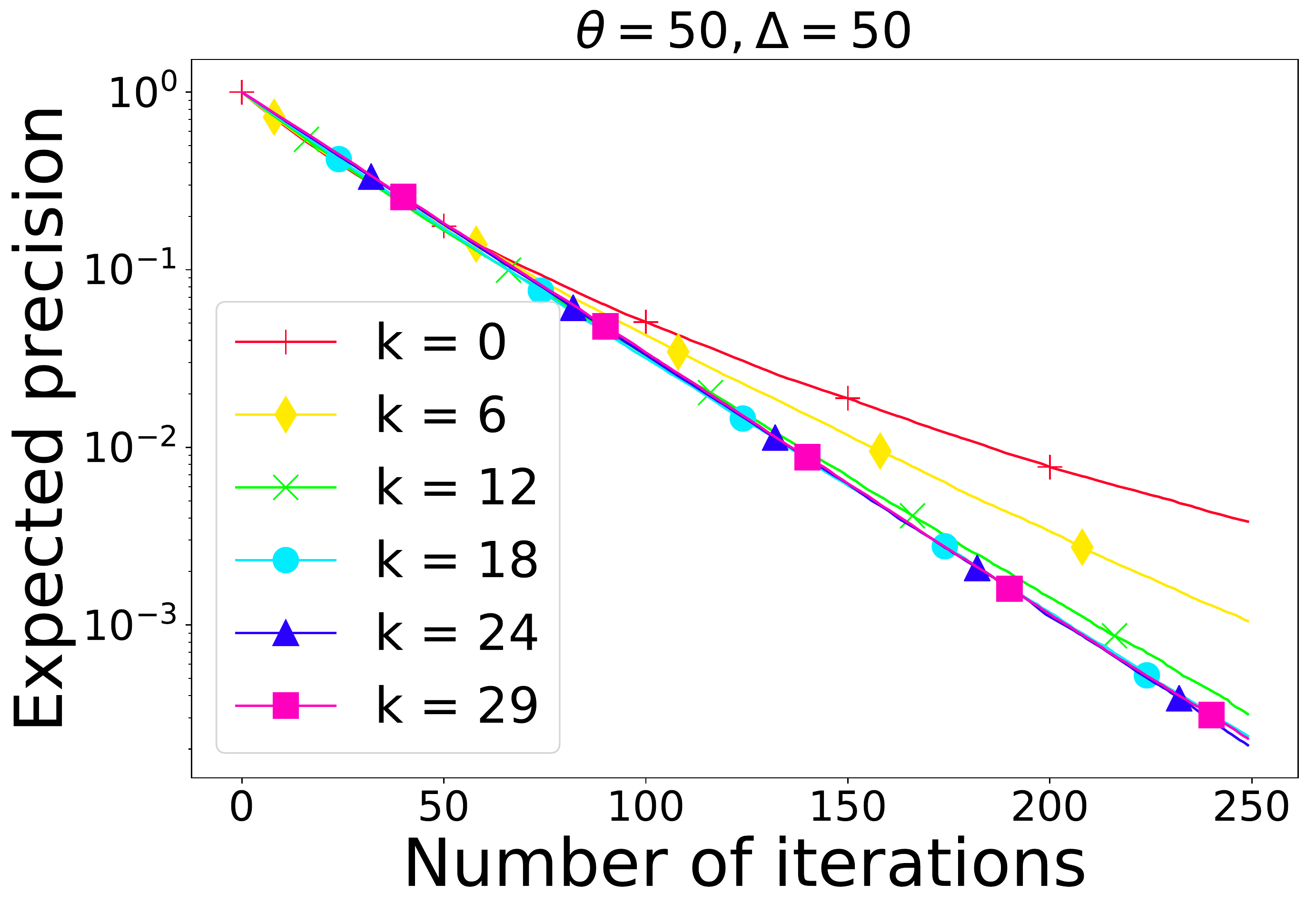}}\\
	\subfloat{\includegraphics[width=0.25\linewidth]{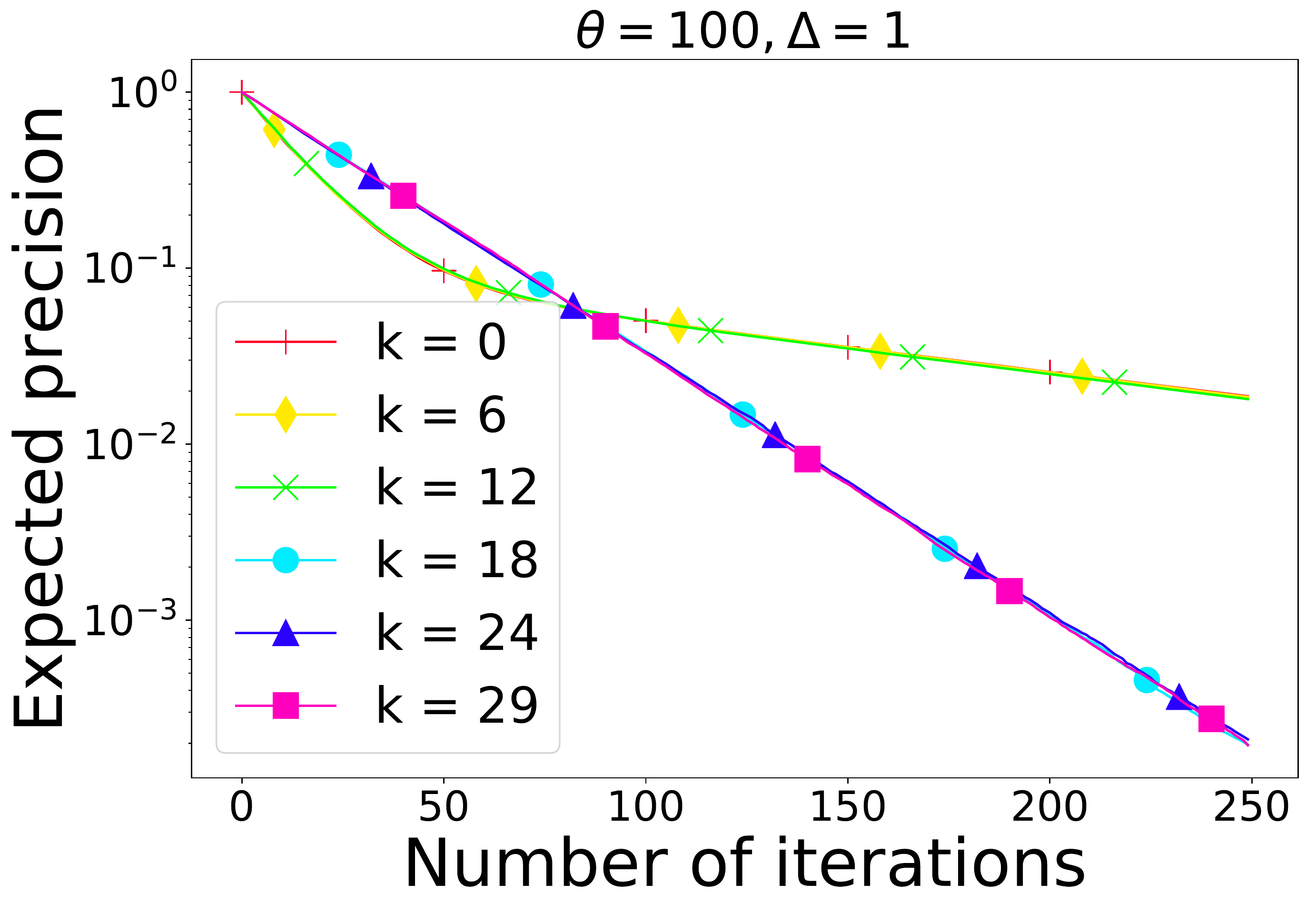}}
	\subfloat{\includegraphics[width=0.25\linewidth]{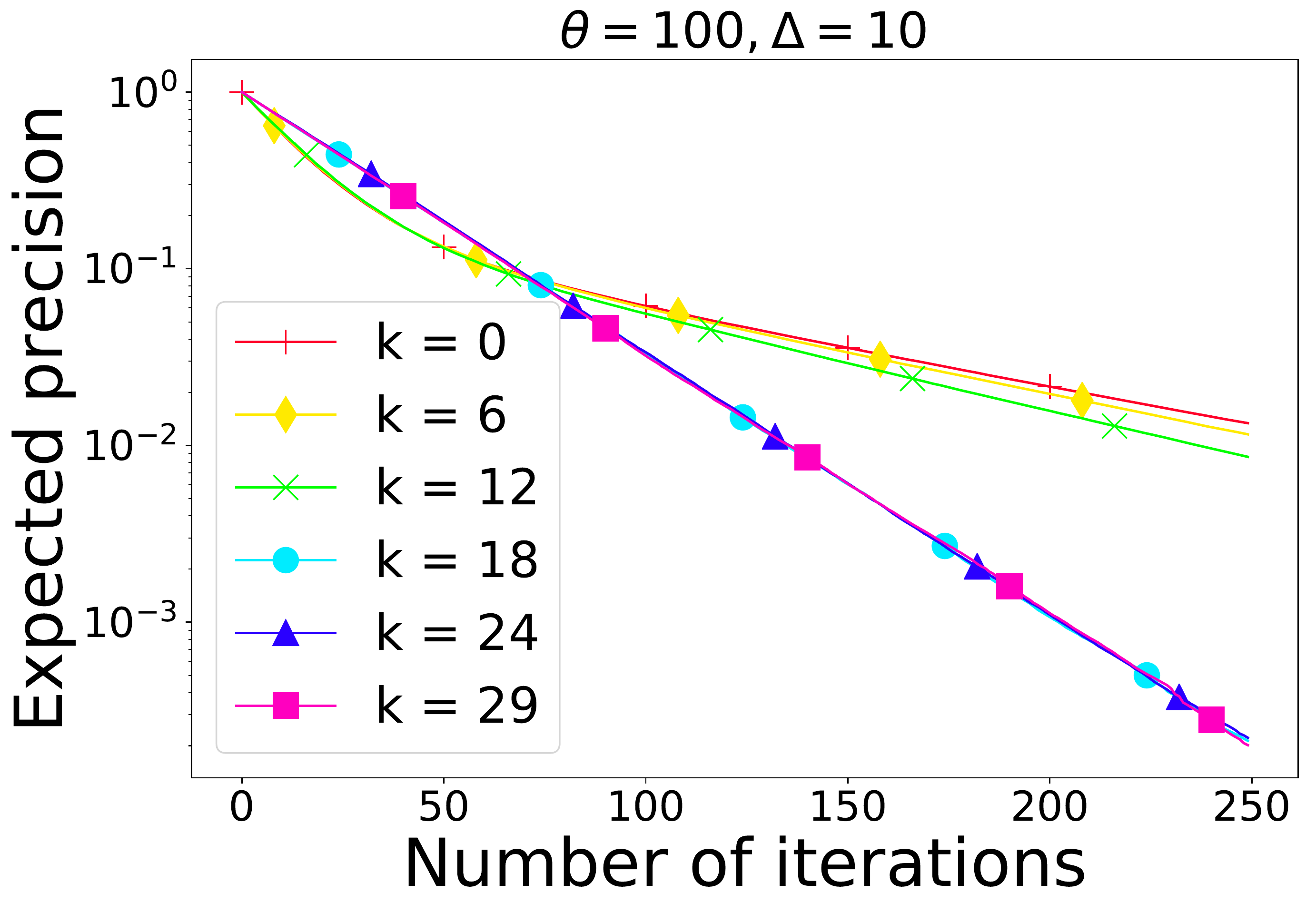}}
	\subfloat{\includegraphics[width=0.25\linewidth]{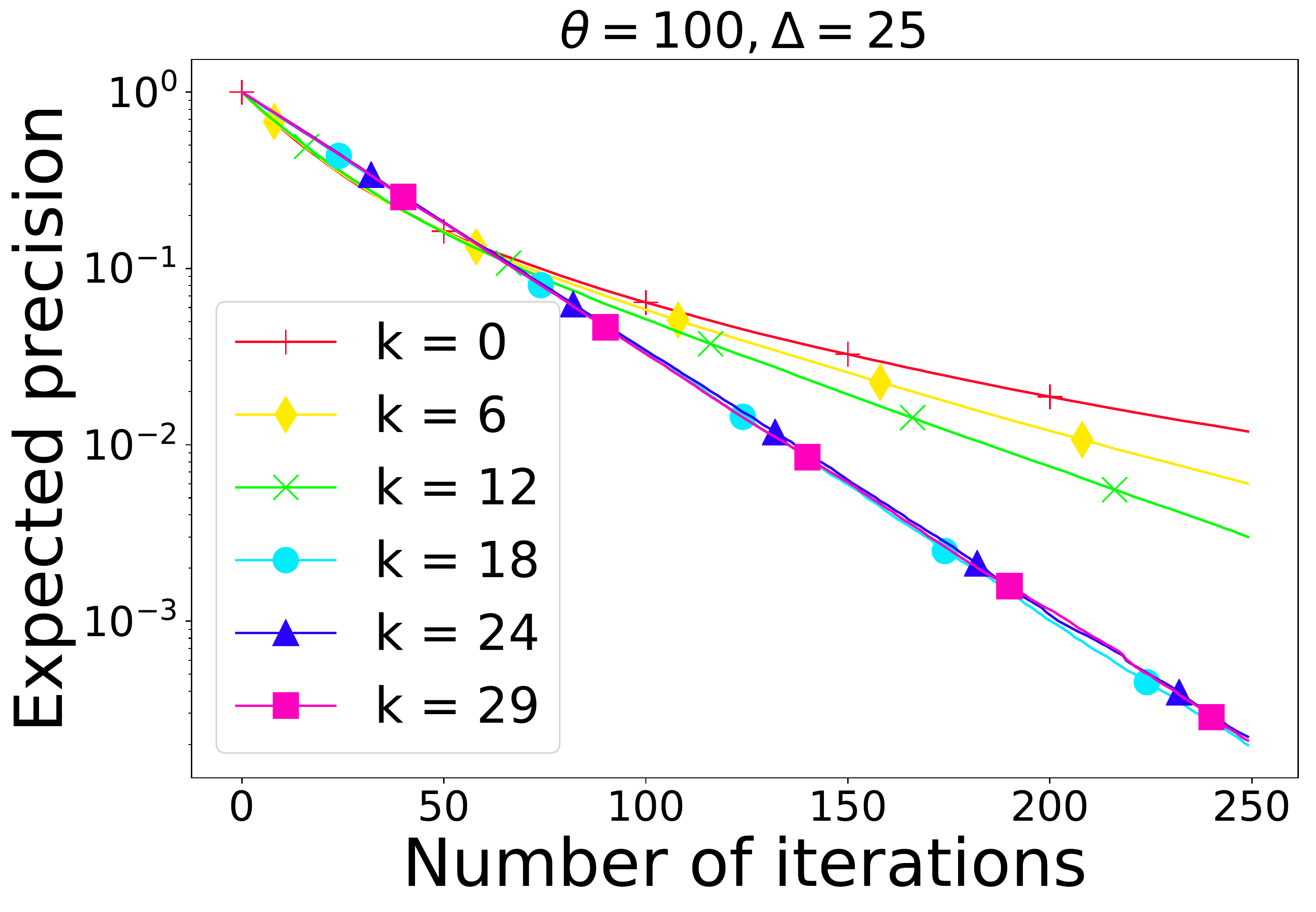}}
	\subfloat{\includegraphics[width=0.25\linewidth]{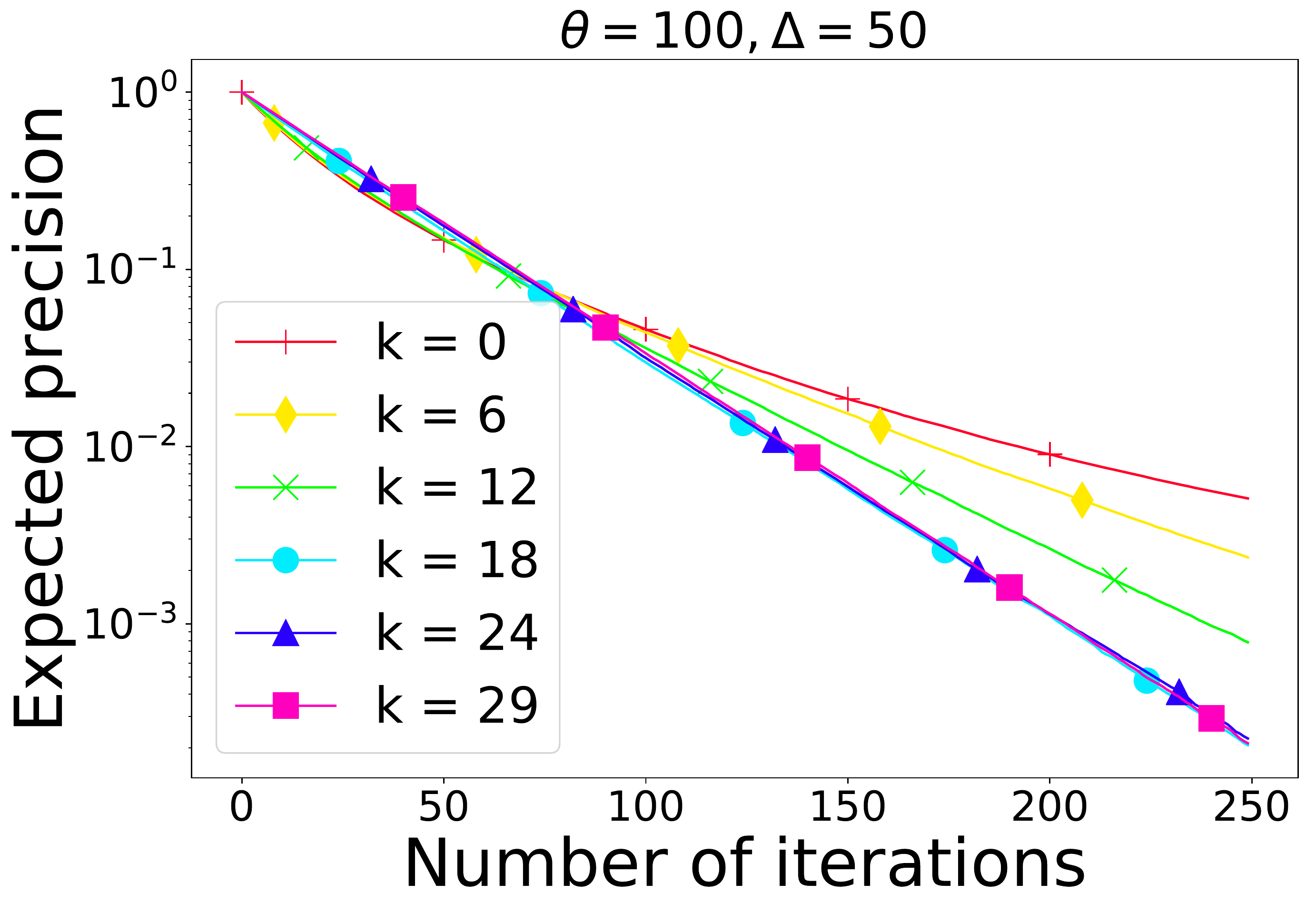}}\\
	\subfloat{\includegraphics[width=0.25\linewidth]{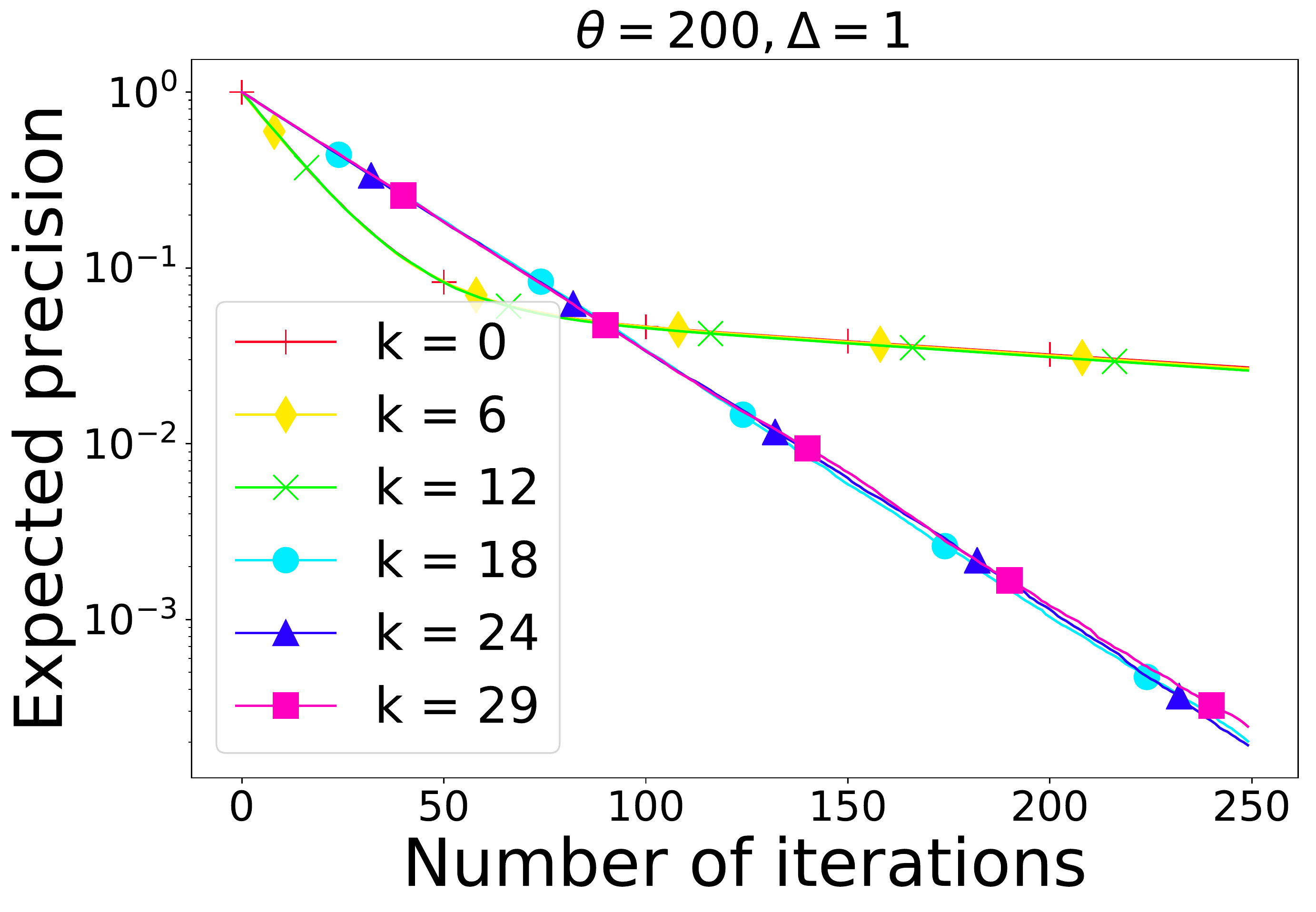}}
	\subfloat{\includegraphics[width=0.25\linewidth]{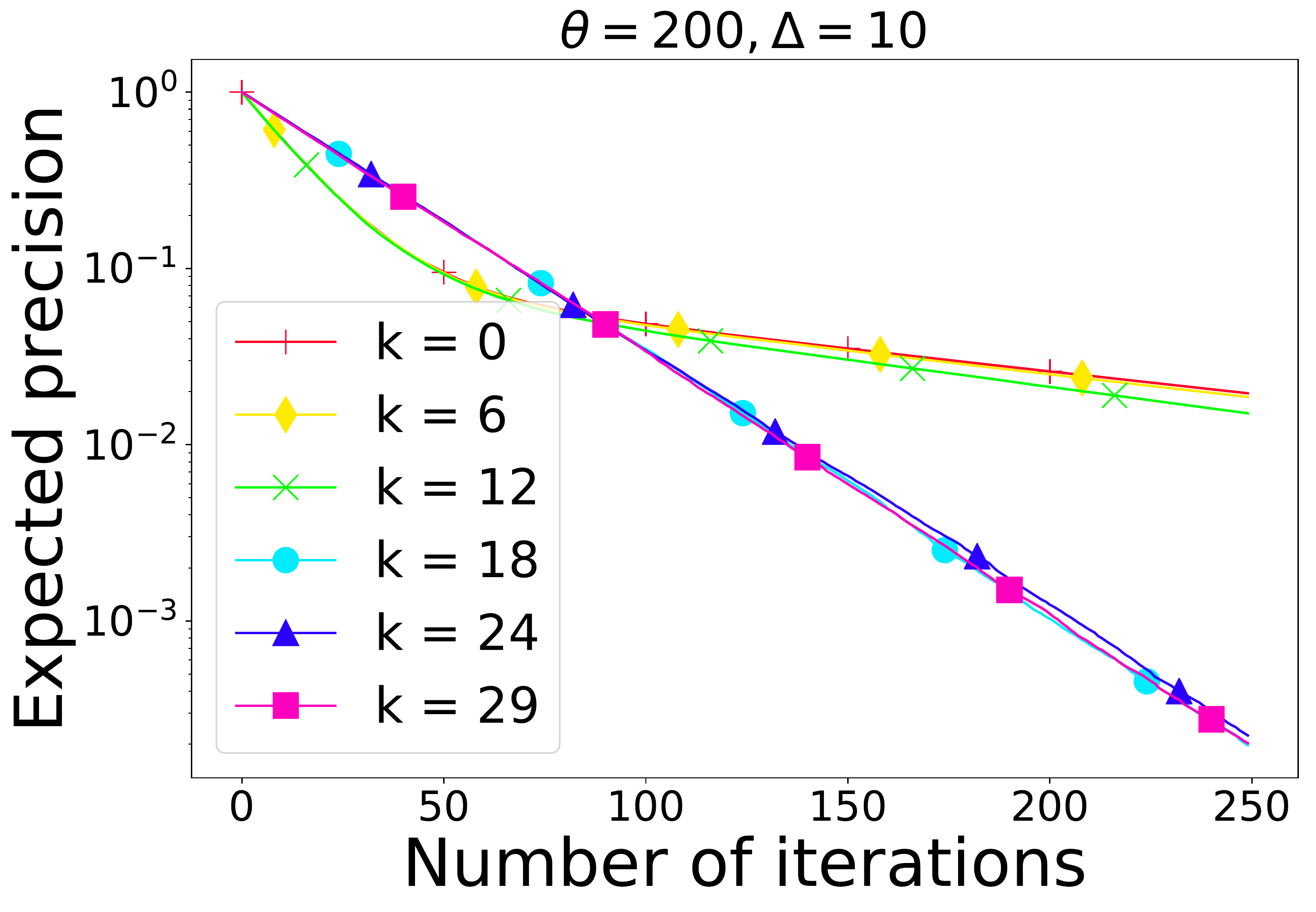}}
	\subfloat{\includegraphics[width=0.25\linewidth]{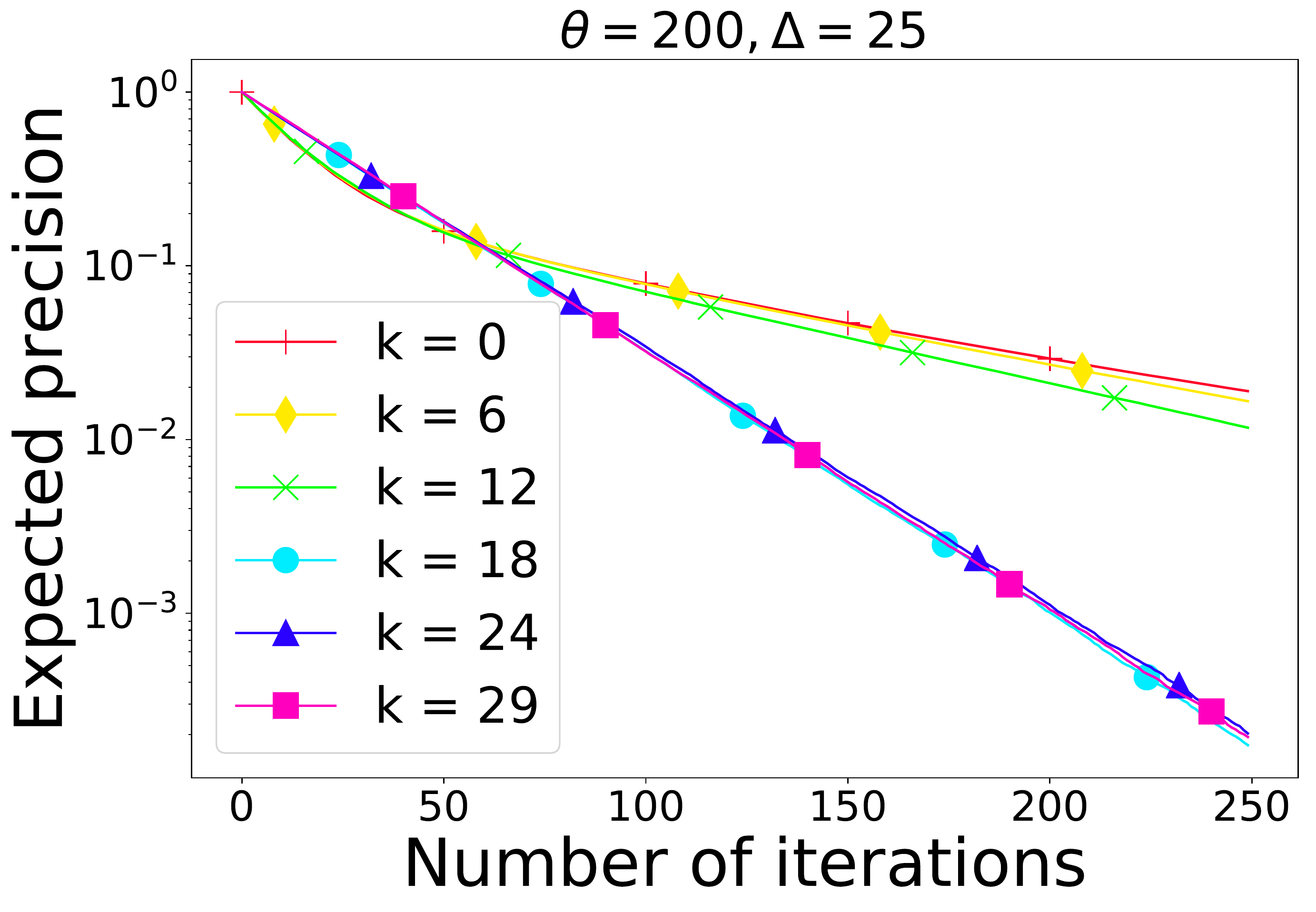}}
	\subfloat{\includegraphics[width=0.25\linewidth]{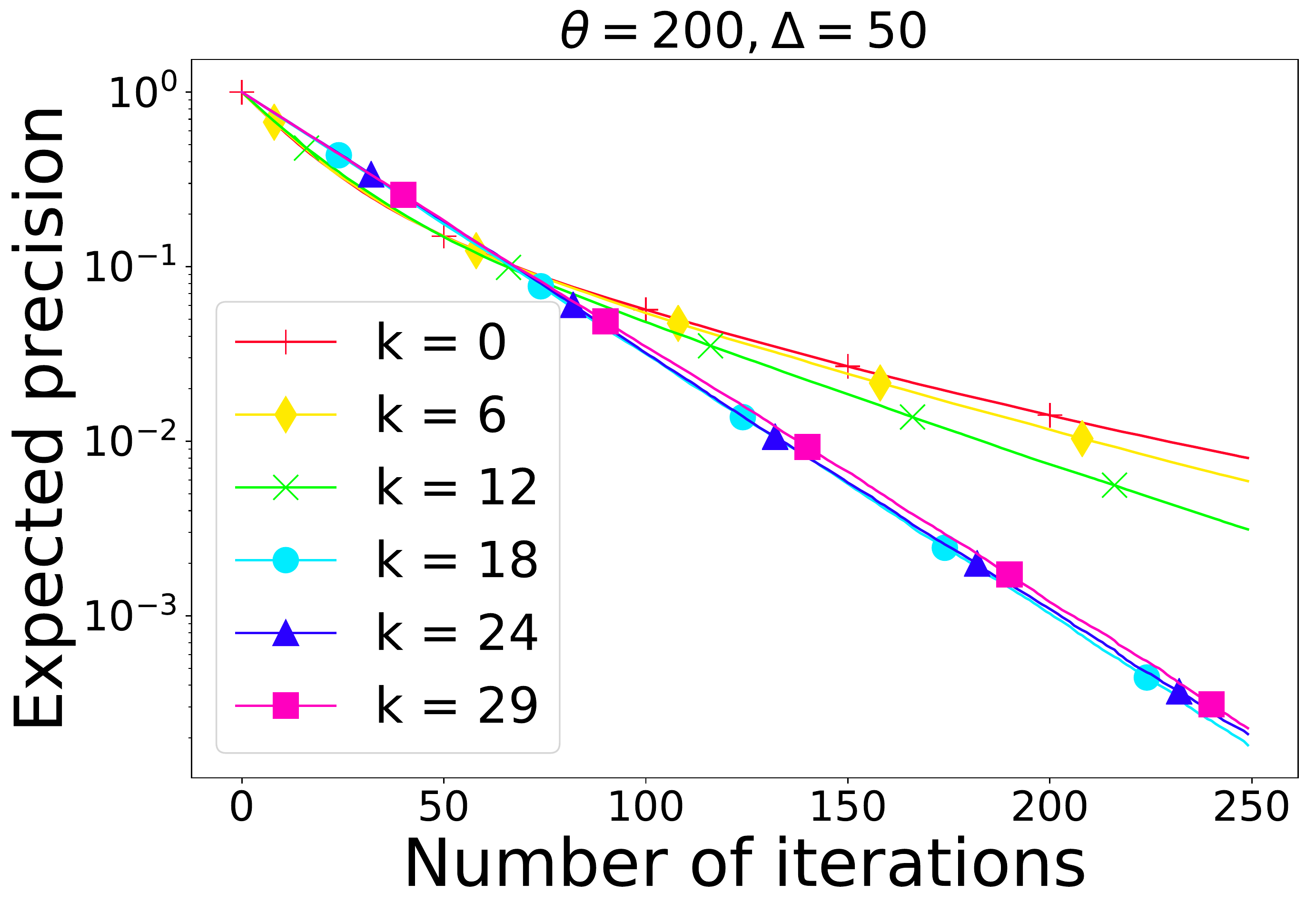}}\\
	\subfloat{\includegraphics[width=0.25\linewidth]{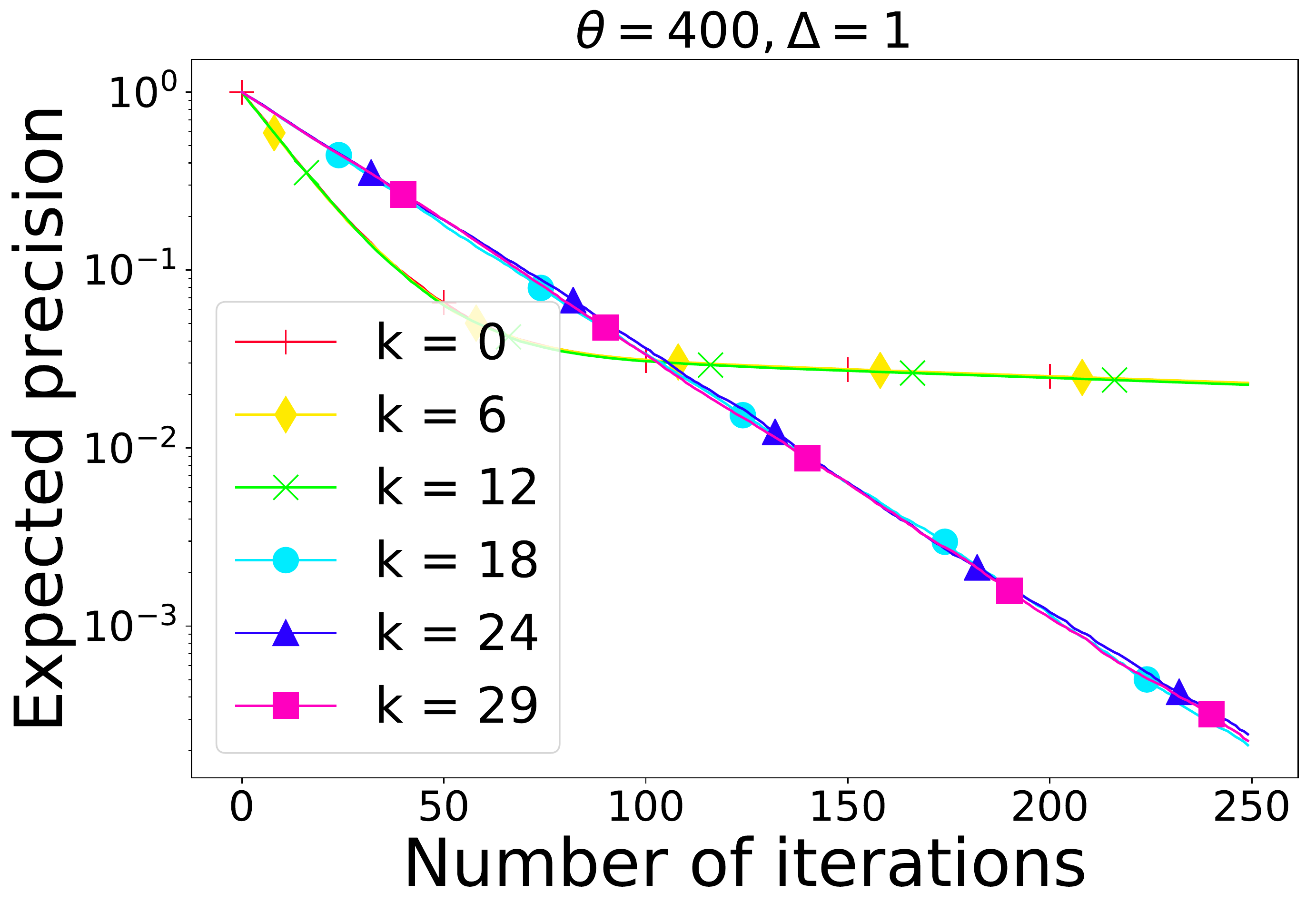}}
	\subfloat{\includegraphics[width=0.25\linewidth]{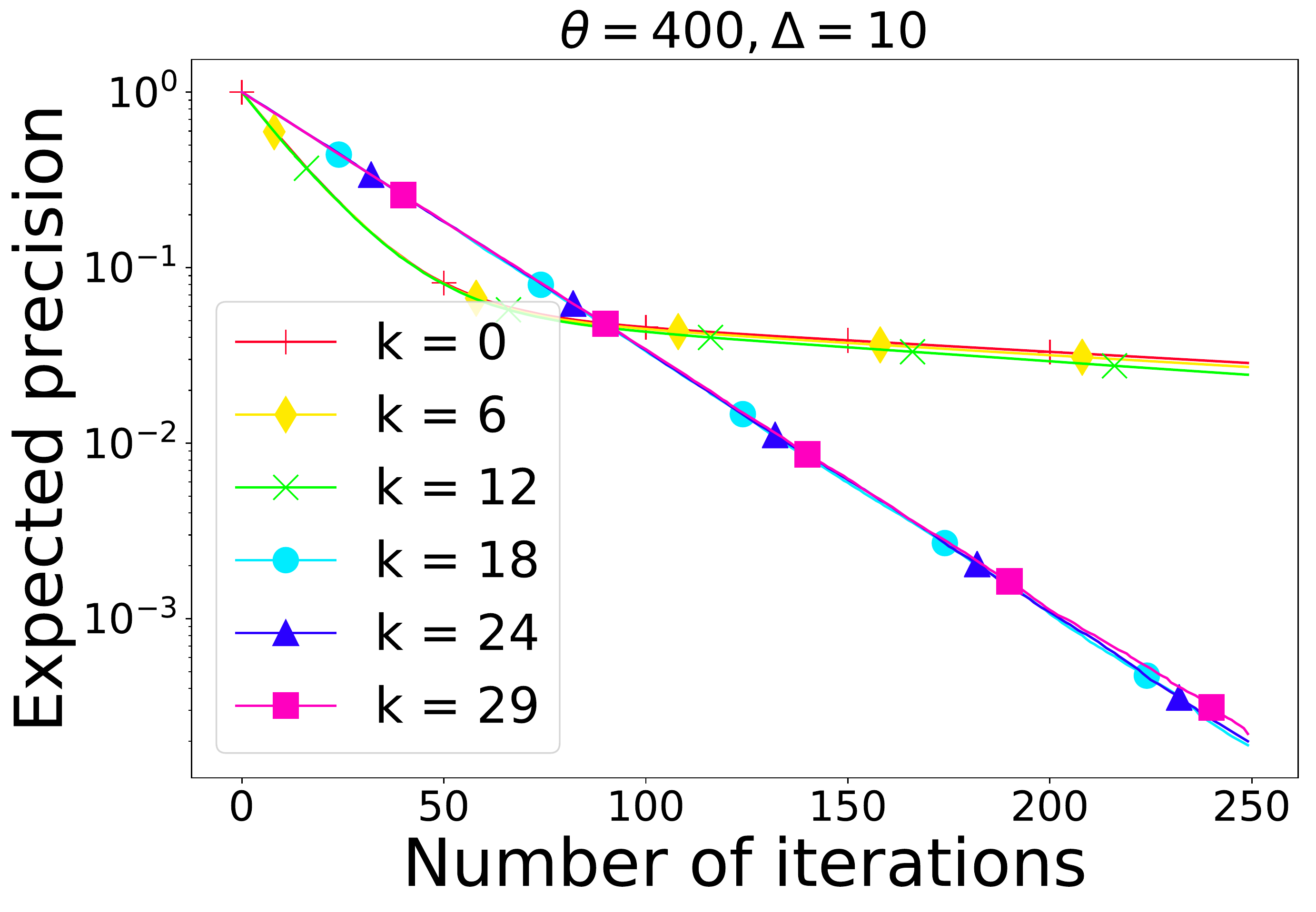}}
	\subfloat{\includegraphics[width=0.25\linewidth]{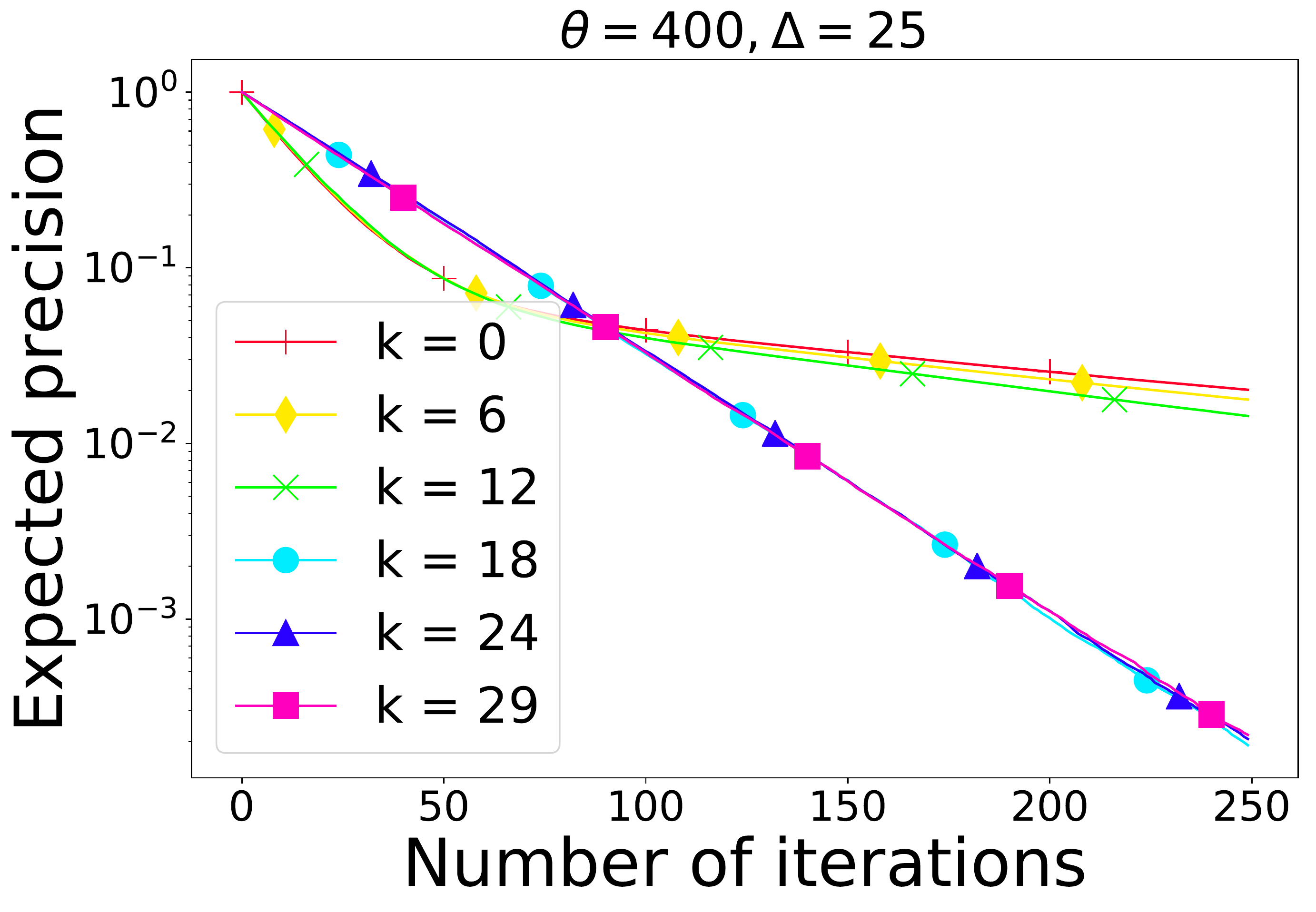}}
	\subfloat{\includegraphics[width=0.25\linewidth]{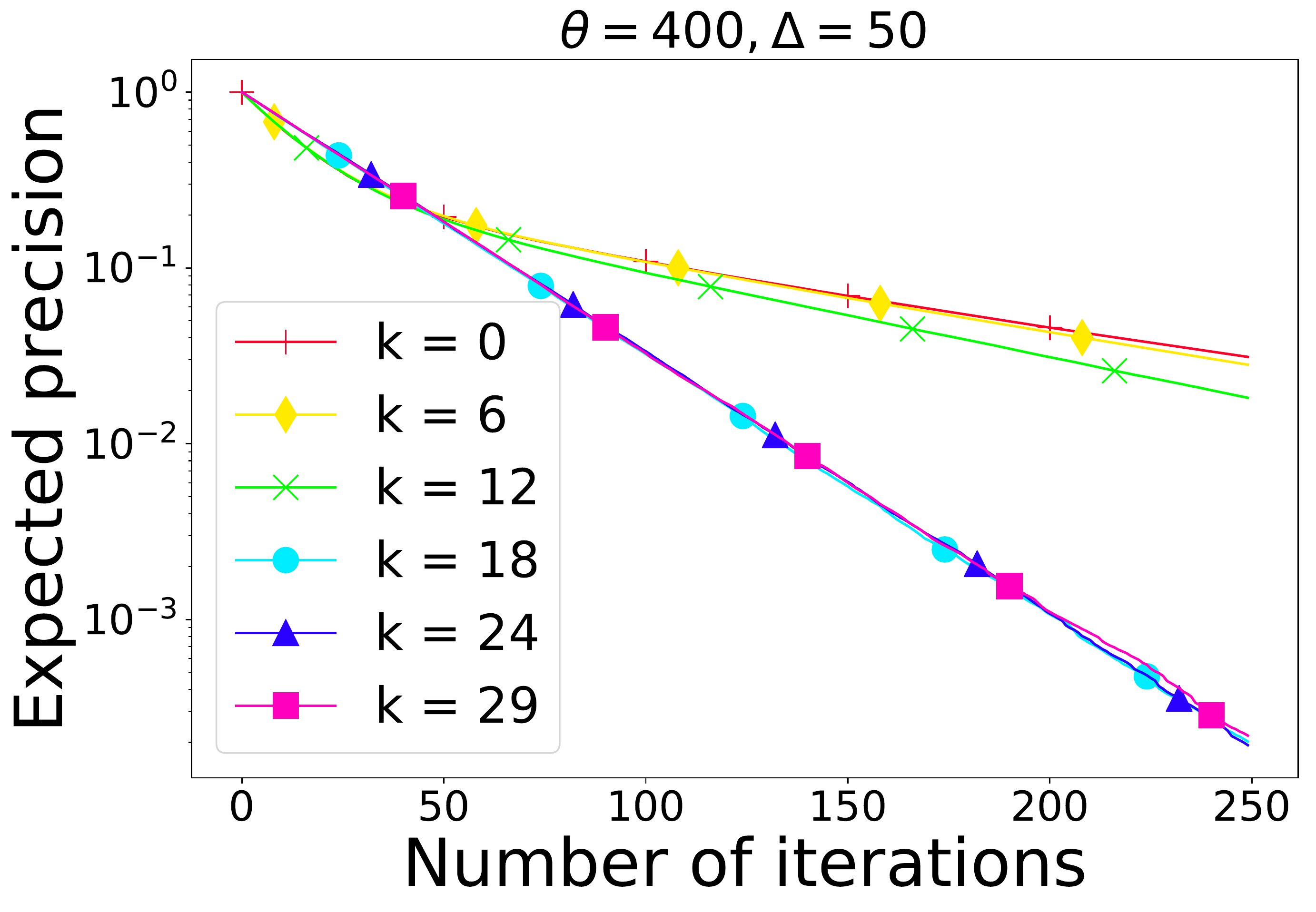}}
	\caption{Expected precision $\Exp\left[\|x_t - x_\ast\|^2_\mA / \|x_0 - x_\ast\|^2_\mA \right]$ versus \# iterations of SSCD for symmetric positive definite matrices $\mA$ of size $30\times 30$ with different structures of spectra. The spectrum of $\mA$ consists of 2 equally sized clusters of eigenvalues; one in the interval $(5, 5 + \Delta)$, and the other in the interval $(\theta, \theta + \Delta)$.  } 
\end{figure*}

\begin{figure*}[t!]
	\centering
	\subfloat{\includegraphics[width=0.25\linewidth]{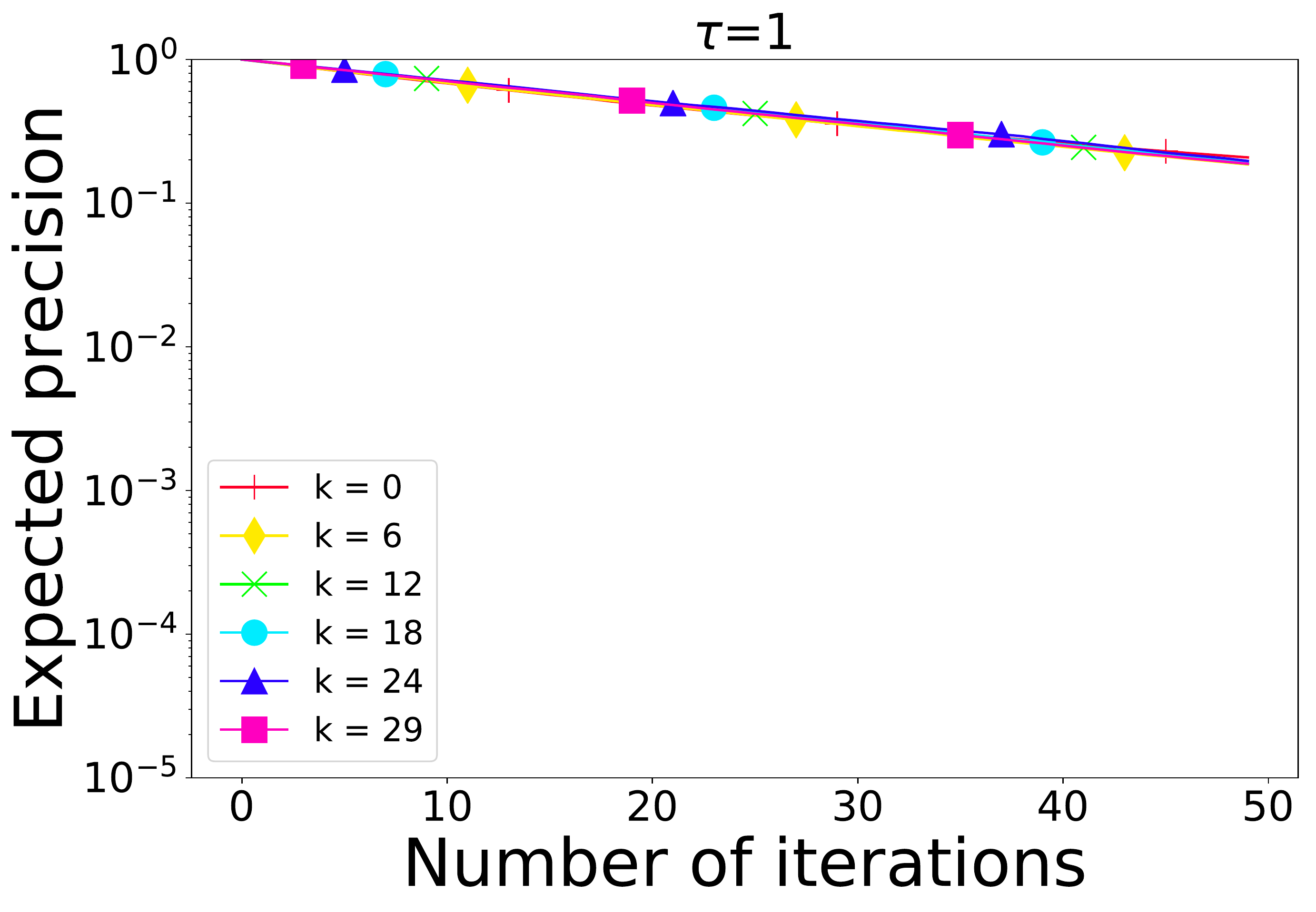}}
	\subfloat{\includegraphics[width=0.25\linewidth]{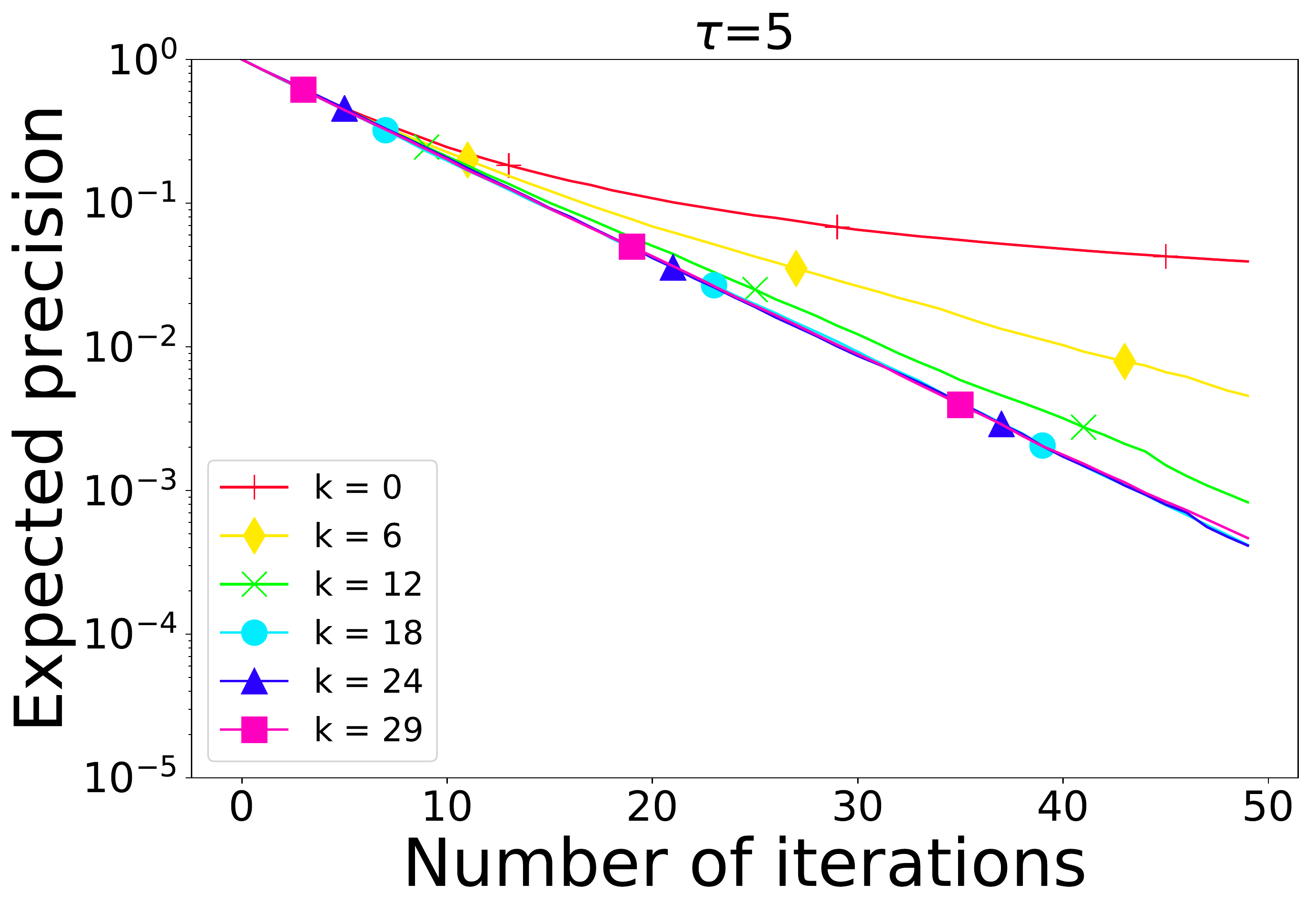}}
	\subfloat{\includegraphics[width=0.25\linewidth]{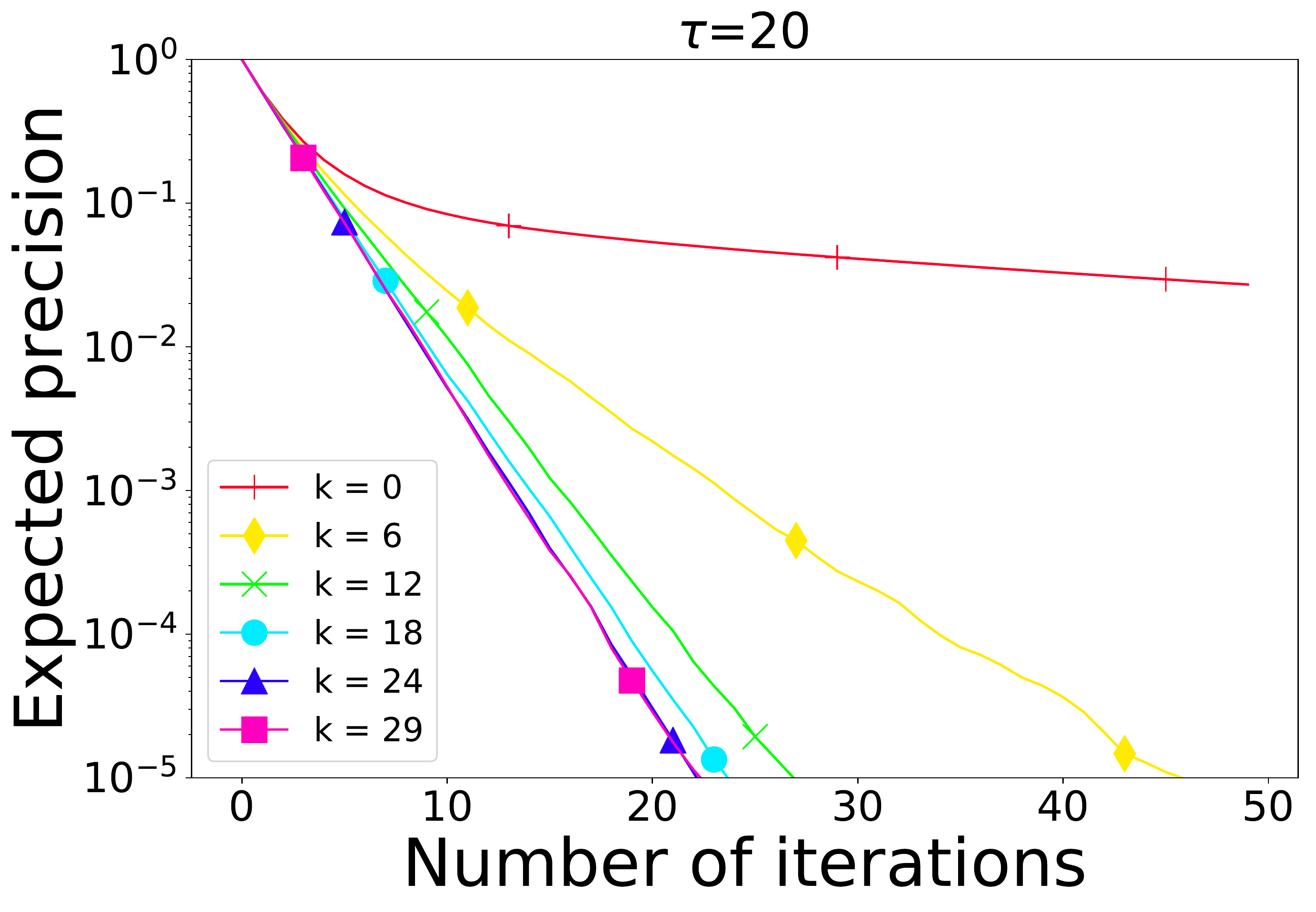}}
	\subfloat{\includegraphics[width=0.25\linewidth]{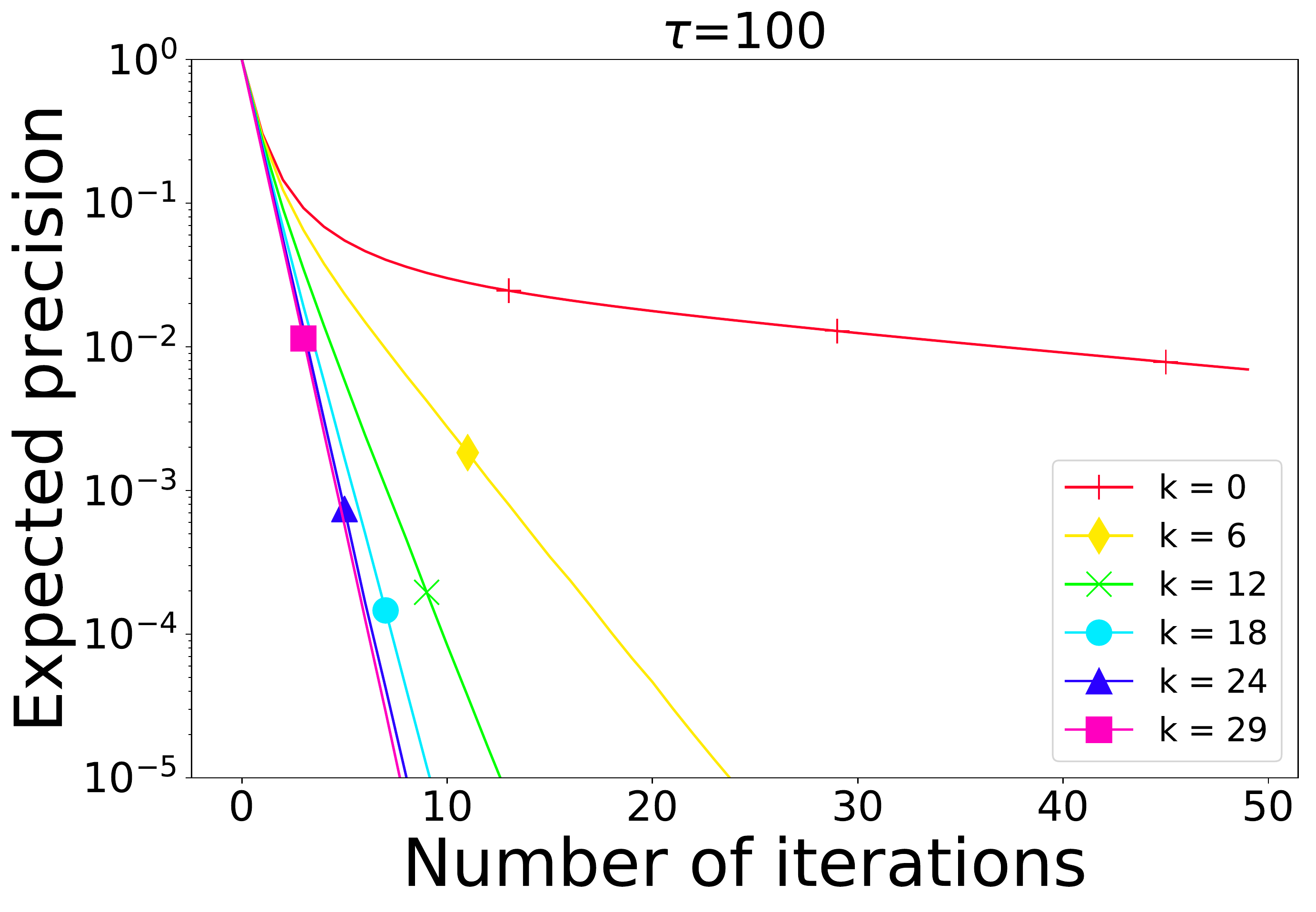}}
	\caption{Expected precision $\Exp\left[ \|x_t - x_\ast\|^2_\mA / \|x_0 - x_\ast\|^2_\mA \right]$ versus \# iterations of mini-batch SSCD for $\mA \in \R^{30\times 30}$ and several choices of mini-batch size $\tau$.  The spectrum of $\mA$ was chosen as a uniform discretization of the interval $[1,60]$. }
	\label{fig:minibatch}
\end{figure*}

\section{Experiments}\label{sec:exp}


\subsection{Stochastic spectral coordinate descent (SSCD)} \label{sec:8hs98h89dhfffKK}

In our first experiment we study how the practical behavior of SSCD (Algorithm~\ref{alg:SSCD}) depends on the choice of $k$. What we study here does not depend on the dimensionality of the problem ($n$), and hence it suffices to perform the experiments on small dimensional problems ($n=30$). 


In this  experiment we consider  the regime of {\em clustered eigenvalues} described in Section~\ref{sec:ibns98g9db9((} and summarized in Table~\ref{tbl:regimes}. In particular, we construct a synthetic  matrix $\mA \in \R^{30\times 30}$ with the smallest 15 eigenvalues clustered in the interval $(5, 5+\Delta)$ and the largest 15 eigenvalues clustered in the interval $(\theta,\theta+\Delta)$. We vary the {\em tightness} parameter $\Delta$ and the {\em separation} parameter $\theta$, and study the performance of SSCD for various choices of $k$. See Figure~\ref{fig:thetadelta}. 

Our first finding is a confirmation of the {\em phase transition} phenomenon predicted by our theory.  Recall that the rate of SSCD (see Theorem~\ref{thm:SSCD}) is
\[\tilde{\cO}\left( \frac{(k+1)\lambda_{k+1} + \sum_{i=k+2}^n \lambda_i}{\lambda_{k+1}}\right) .\]
If $k<15$, we know $\lambda_i \in (5,5+\Delta)$ for $i=1,2,\dots,k+1$, and $\lambda_i \in (\theta,\theta+\Delta)$ for $i=k+2, \dots,n$. Therefore,  the rate  can be estimated as
\[r_{small} \eqdef \tilde{\cO}  \left( k+1 + \frac{(n-k-1) (\theta+\Delta)}{5} \right)  .\] On the other hand, if $k\geq 15$,  we know that $\lambda_i \in (\theta,\theta+\Delta)$ for $i=k+1, \dots,n$, and hence
the rate can be estimated as
\[r_{large} \eqdef \tilde{\cO}  \left( k+1 + \frac{(n-k-1) (\theta+\Delta)}{\theta} \right).\] Note that if the separation $\theta$ between the two clusters is large, the rate $r_{large}$ is much better than the rate $r_{small}$. Indeed, in this regime, the rate $r_{large}$ becomes $\tilde{\cO}(n)$, while $r_{small}$ can be arbitrarily large.

Going back to Figure~\ref{fig:thetadelta}, notice that this can  be observed in the experiments. There is a clear {\em phase transition} at $k=15$, as predicted be the above analysis. Methods using $k\in \{0,6,12\}$  are relatively slow (although still enjoying a linear rate), and tend to have similar behaviour, especially when $\Delta$ is small. On the other hand, methods using $k\in \{18,24,29\}$ are much faster, with a behaviour nearly independent of $\theta$ and $\Delta$. Moreover, as $\theta$ increases, the difference in the rates between the {\em slow} methods using $k\in \{0,6,12\}$ and the {\em fast} methods using $k\in \{18,24,29\}$ grows.

We have performed additional experiments with three clusters; see Figure~\ref{fig:3clusters} in the appendix.

\subsection{Mini-batch SSCD}

In Figure~\ref{fig:minibatch} we report on the behavior of mSSCD, the mini-batch version of SSCD, for four choices of the mini-batch parameter $\tau$, and several choices of $k$. Mini-batch  of size $\tau$ is processed in parallel on $\tau$ processors, and the cost of a single iteration of mSSCD is (roughly) the same for all $\tau$. 

For $\tau=1$, the method reduces to SSCD, considered in previous experiment (but on a different dataset). Since the number of iterations is small, there are no noticeable differences across using different values of $k$.  As $\tau$ grows, however, all methods become faster. Mini-batching seems to be more useful as $k$ is larger. Moreover, we can observe that  acceleration through mini-batching starts more aggressively for small values op $k$, and its added benefit for increasing values of $k$ is getting smaller and smaller. This means that even for relatively small values of $k$, mini-batching  can be expected to lead to substantial speed-ups.

\subsection{Matrix with 10 billion entries}

In Figure~\ref{fig:thetadelta} we report on an experiment using a synthetic problem with data matrix $\mA$ of dimension $n=10^5$ (i.e., potentially with  $10^{10}$ entries). As all experiments were done on a laptop, we worked with sparse matrices with $10^6$ nonzeros only.

In the first row of Figure~\ref{fig:thetadelta} we consider matrix $\mA$ with all eigenvalues distributed uniformly on the interval $[1,100]$.  We observe that SSCD with $k=10^4$ (just 10\% of $n$) requires about an {\em order of magnitude} less iterations  than SSCD with $k=0$ (=RCD).  

In the second row we consider a scenario where $l$ eigenvalues are small, contained in $[1,2]$, with the rest of the eigenvalues contained in $[100,200]$. We consider $l=10$ and $l=1000$ and study the behaviour of SSCD with $k=l$. We see that for $l=10$, SSCD performs dramatically better than RCD: it is able to  achieve machine precision while RCD struggles to reduce the initial error by a factor larger than $10^6$. For $l=1000$, SSCD achieves error $10^{-9}$ while RCD struggles to push the error below $10^{-4}$. These tests show that in terms of \# iterations, {\em SSCD has the capacity to accelerate on RCD by many orders of magnitude.}

\begin{figure}[!h]

	\centering
	\includegraphics[width=0.48\linewidth]{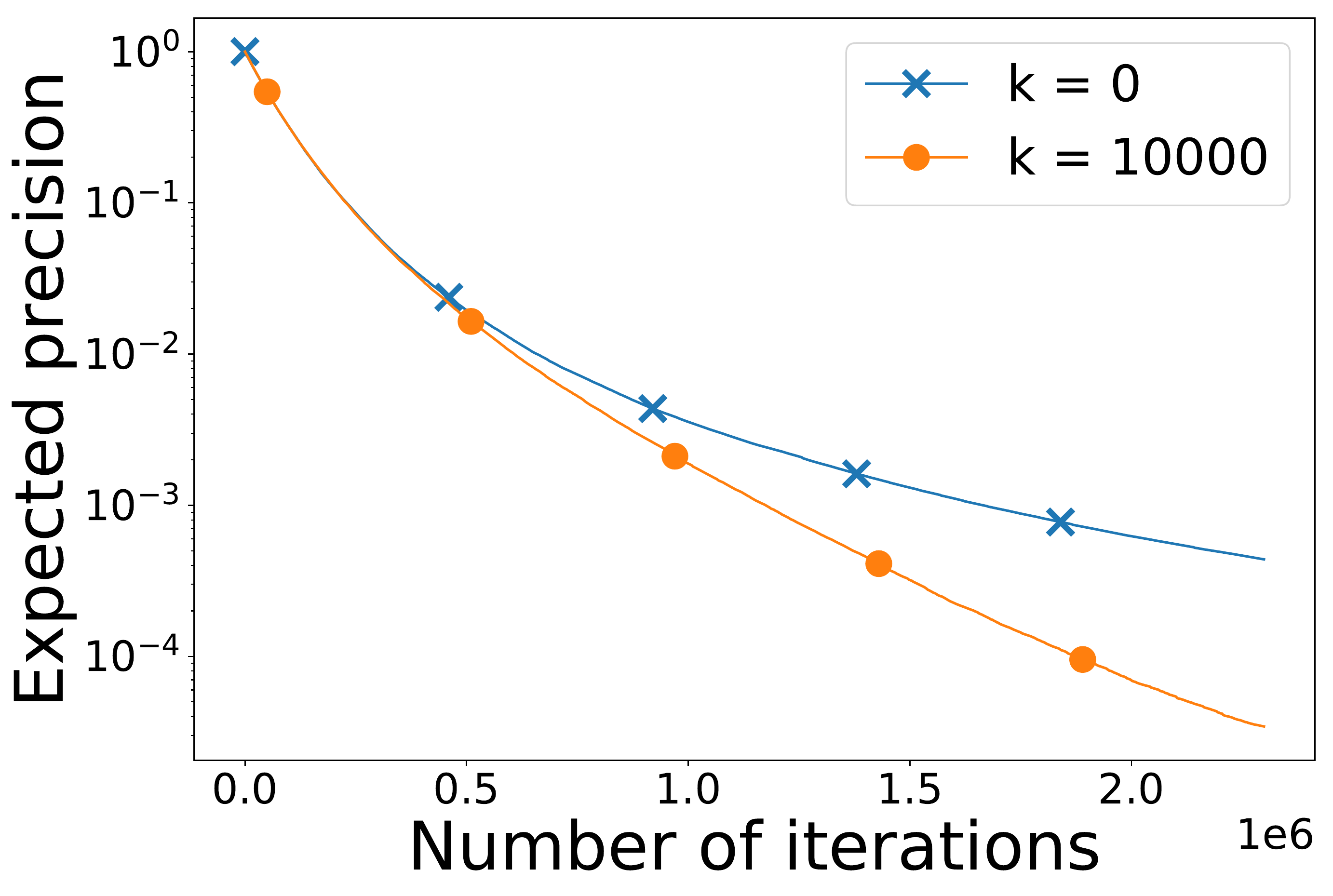}\\
\includegraphics[width=0.48\linewidth]{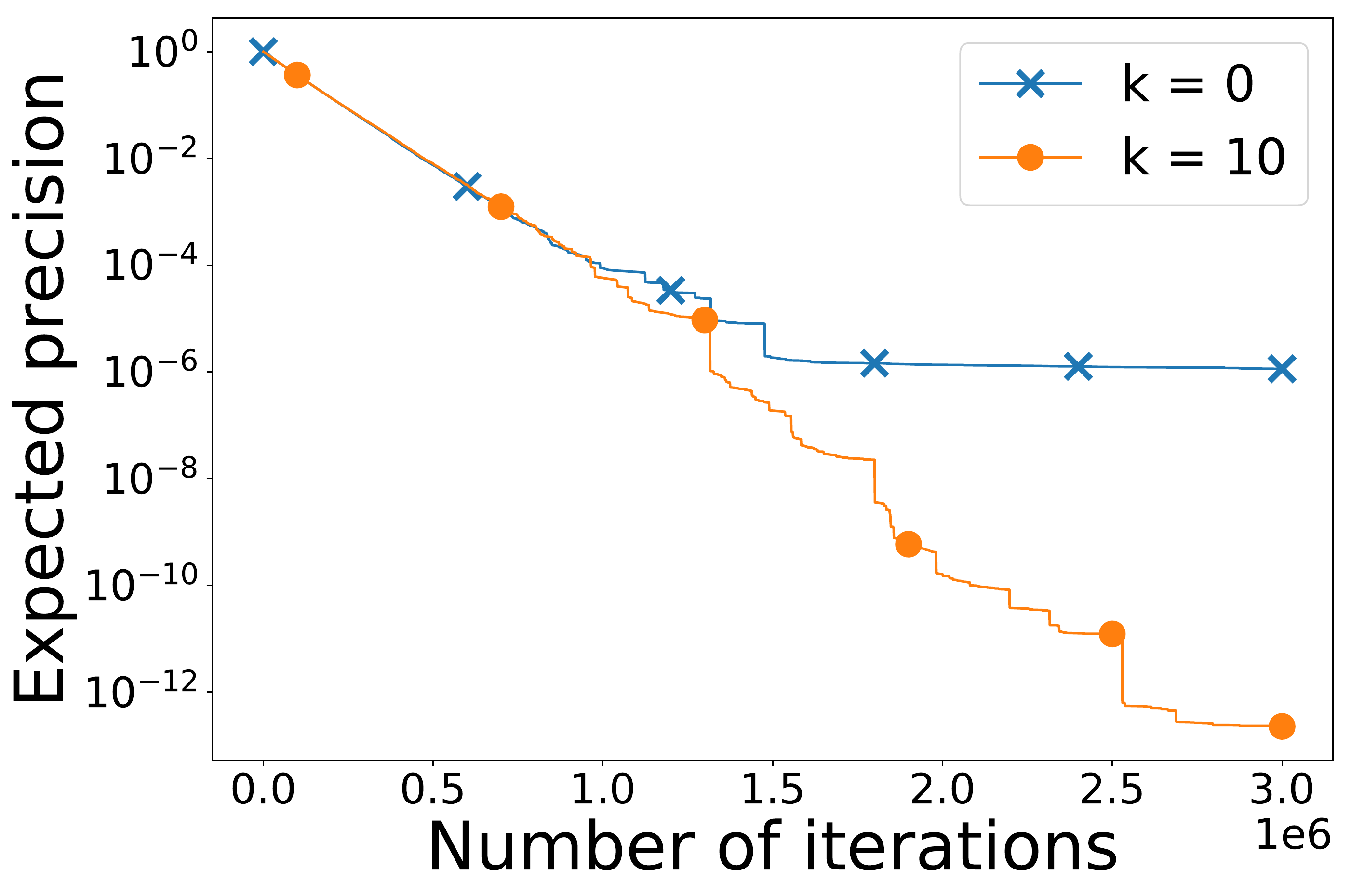}
\includegraphics[width=0.48\linewidth]{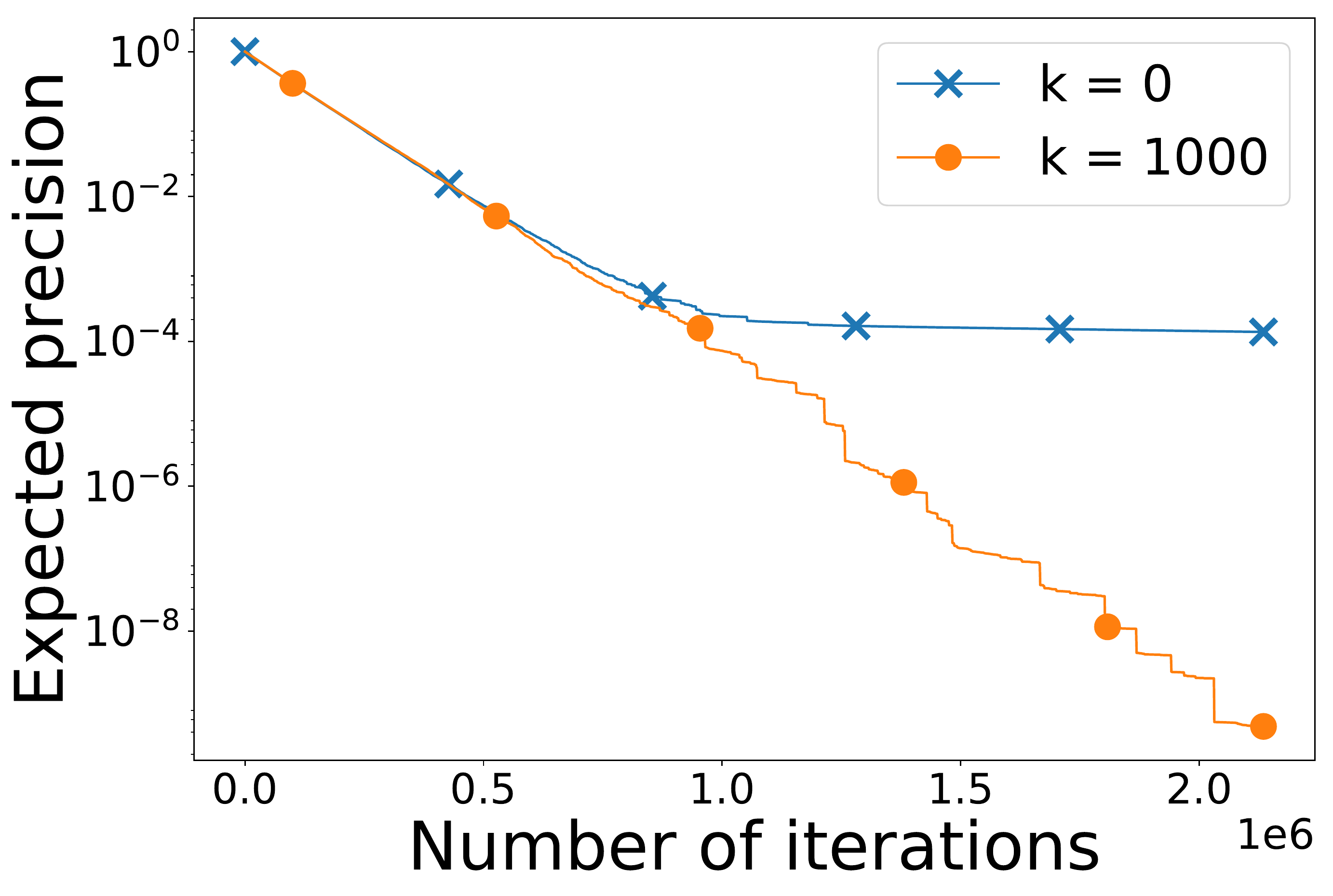}

	\caption{Expected precision $\Exp\left[\|x_t - x_\ast\|^2_\mA/ \|x_0 - x_\ast\|^2_\mA \right]$ versus \# iterations of SSCD for a matrix $\mA \in \R^{10^5 \times 10^5}$. Top row:  spectrum of $\mA$ is uniformly distributed on $[1,100]$. Bottom row:
	spectrum contained in two clusters: $[1, 2]$ and $[100, 200]$.}
	\label{fig:thetadelta}
\end{figure}


\section{Extensions}

Our algorithms and convergence results can be extended to eigenvectors and conjugate directions which are only computed {\em approximately}. Some of this development can be found in the appendix (see Section~\ref{sec:inexact_methods}).
Finally, as mentioned in the introduction, our results can be extended to the more general problem of minimizing $f(x) = \phi(\mA x)$, where $\phi$ is smooth and strongly convex.


\bibliography{references}
\bibliographystyle{icml2018}


\clearpage
\onecolumn

\icmltitle{Appendix}

\section{Extra Experiments}

In this section we report on some additional experiments which shed more light on the behaviour of our methods.

\subsection{Performance on SSCD on $\mA$ with three clusters eigenvalues}

\begin{figure*}[!h]
	\subfloat{\includegraphics[width=0.25\textwidth]{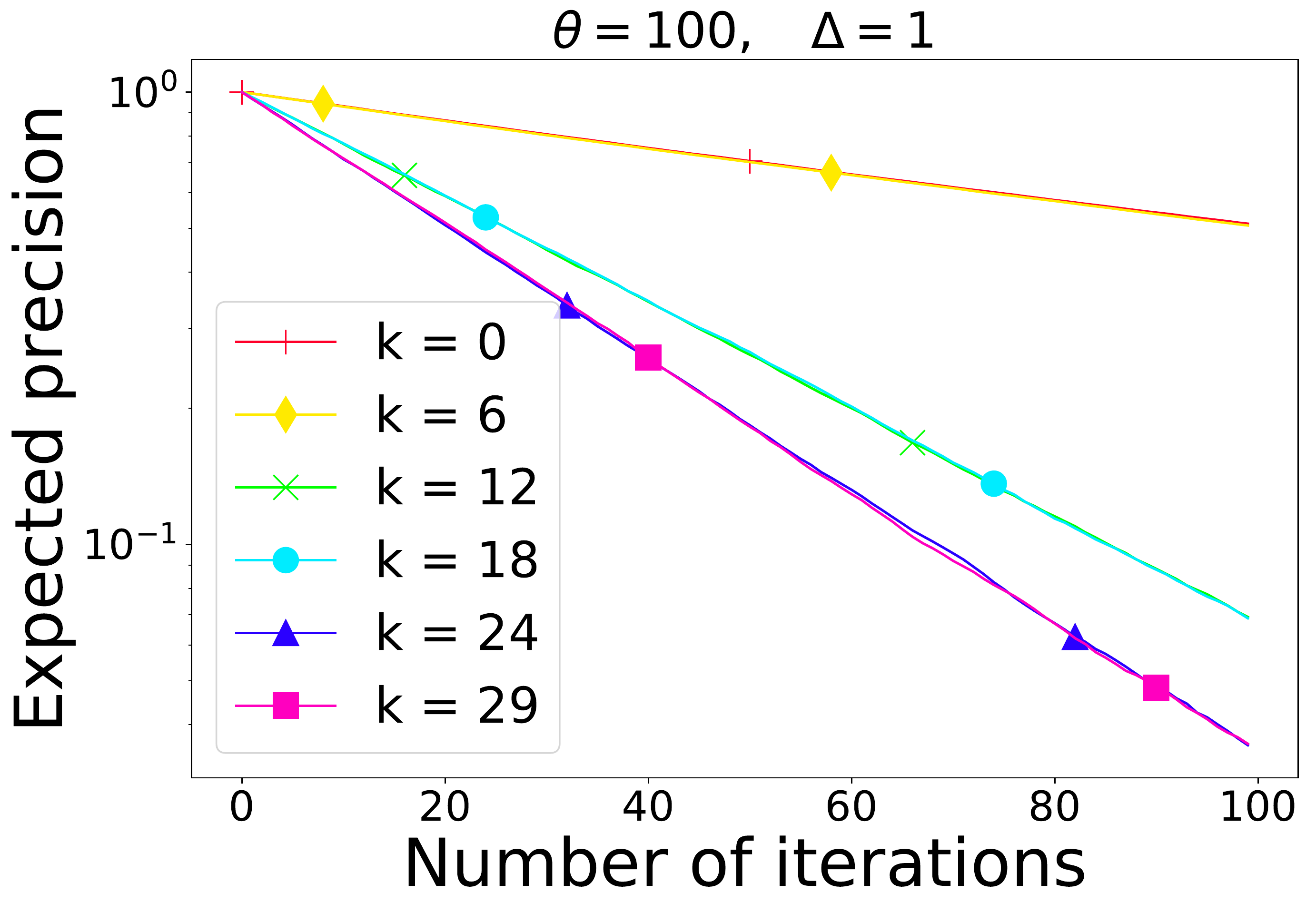}}
	\subfloat{\includegraphics[width=0.25\textwidth]{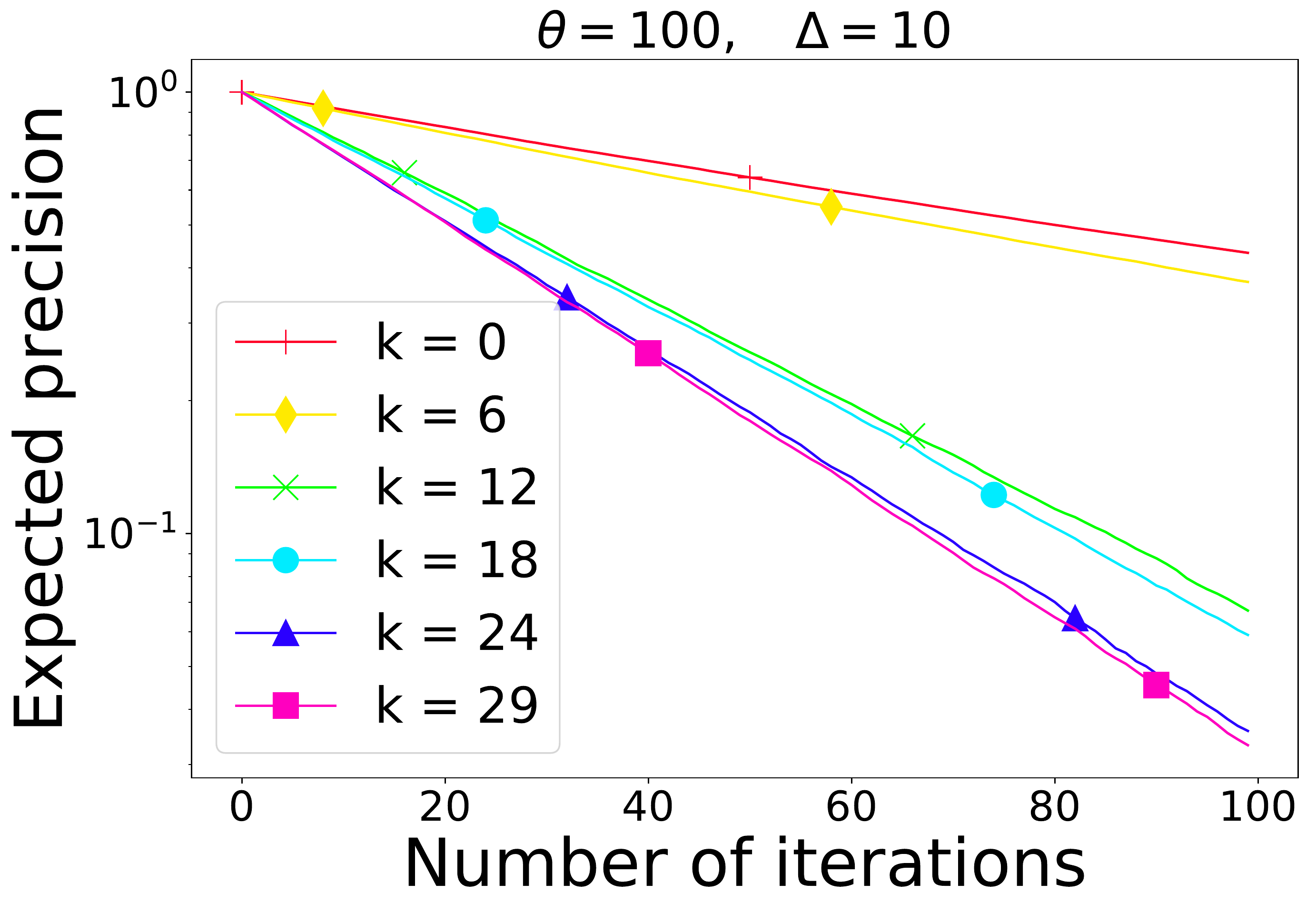}}
	\subfloat{\includegraphics[width=0.25\textwidth]{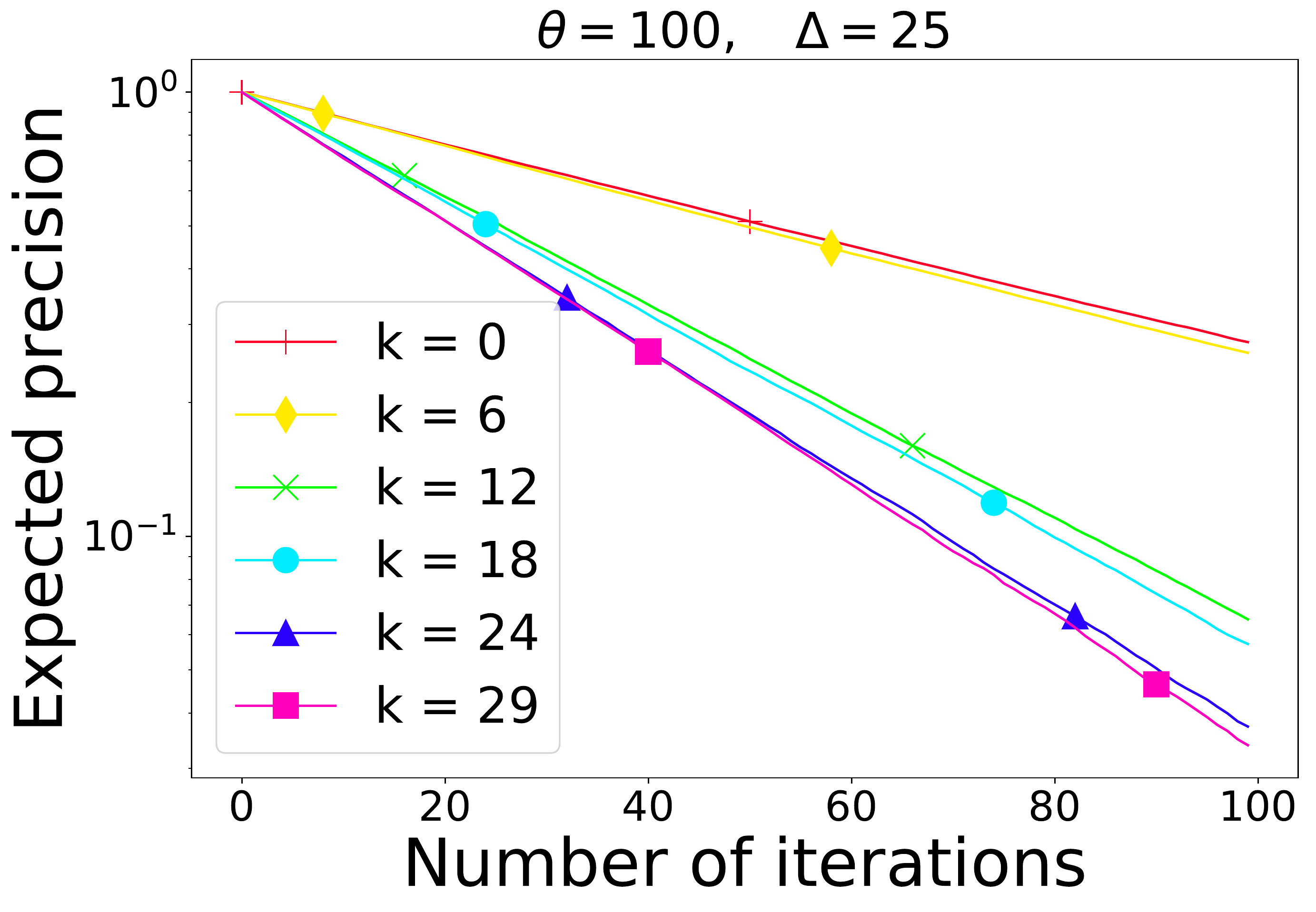}}
	\subfloat{\includegraphics[width=0.25\textwidth]{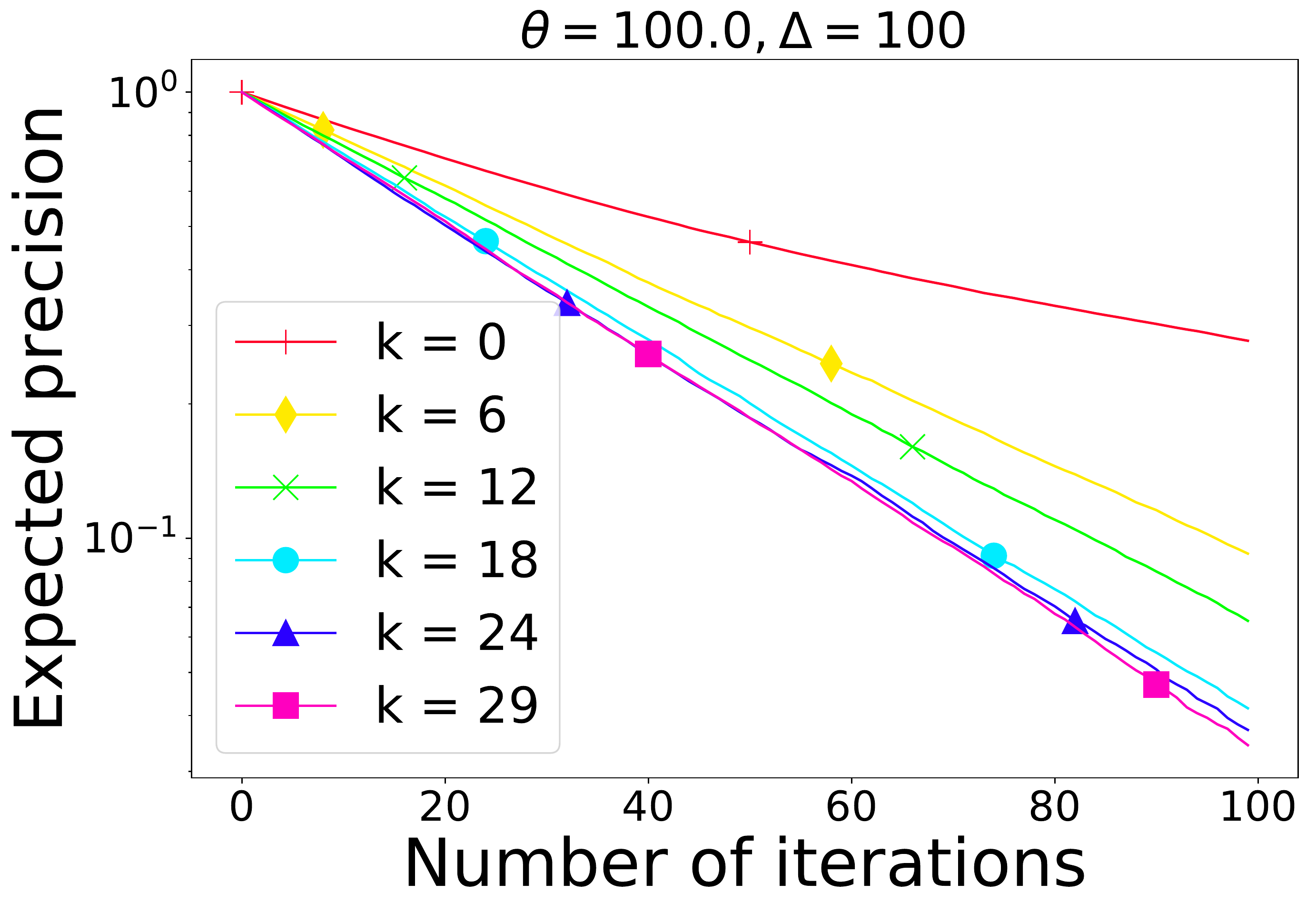}}\\
	\subfloat{\includegraphics[width=0.25\textwidth]{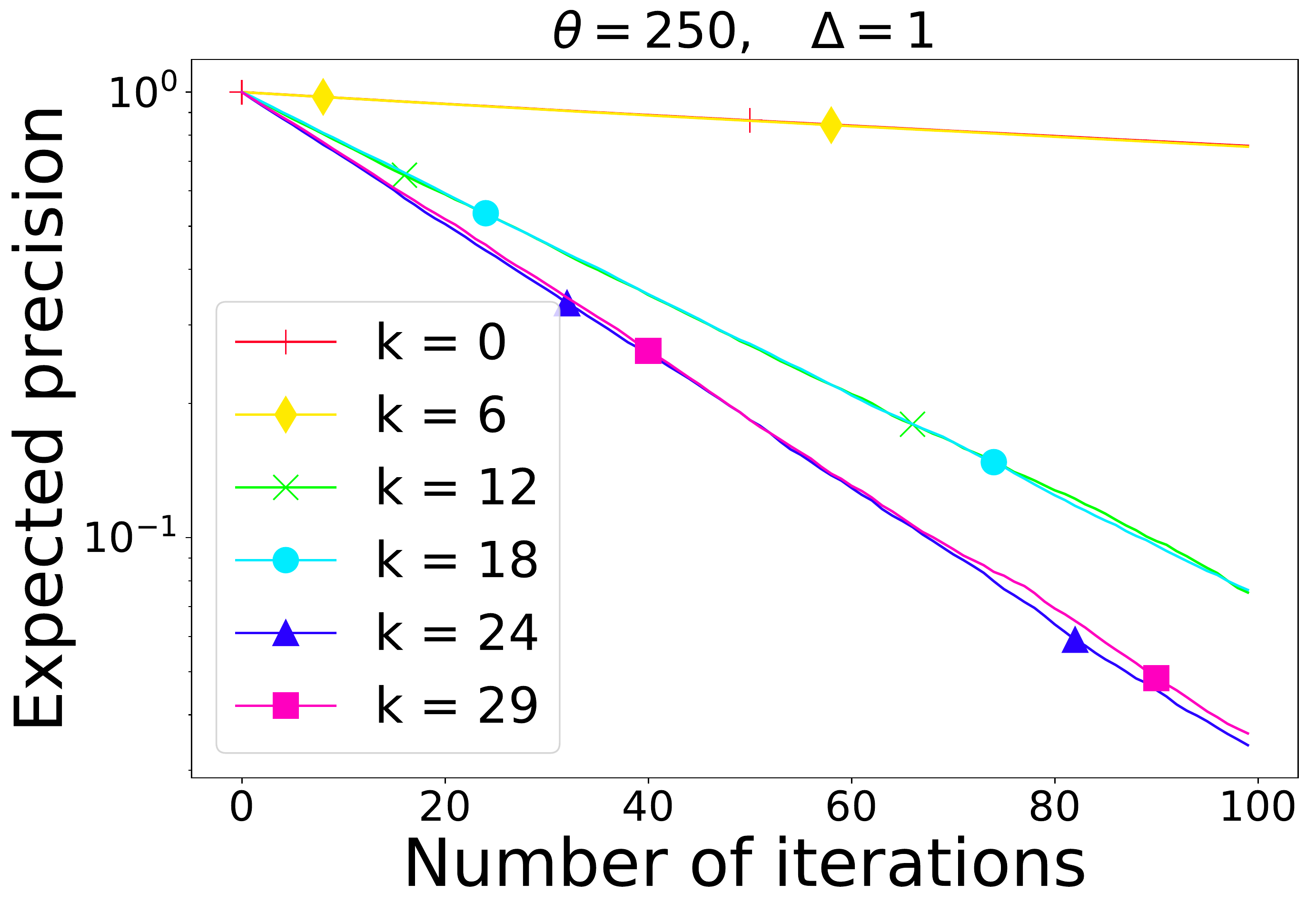}}
	\subfloat{\includegraphics[width=0.25\textwidth]{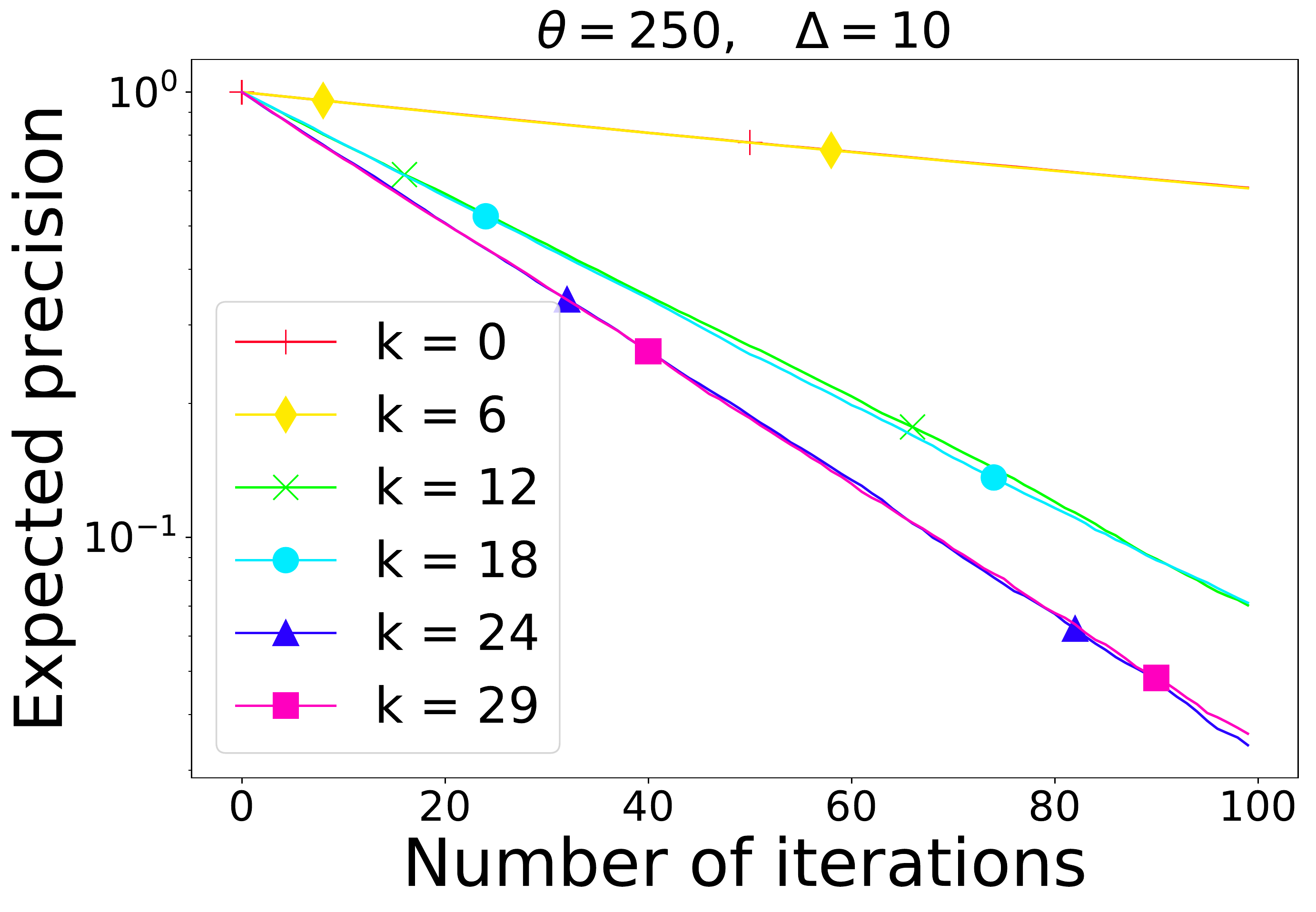}}
	\subfloat{\includegraphics[width=0.25\textwidth]{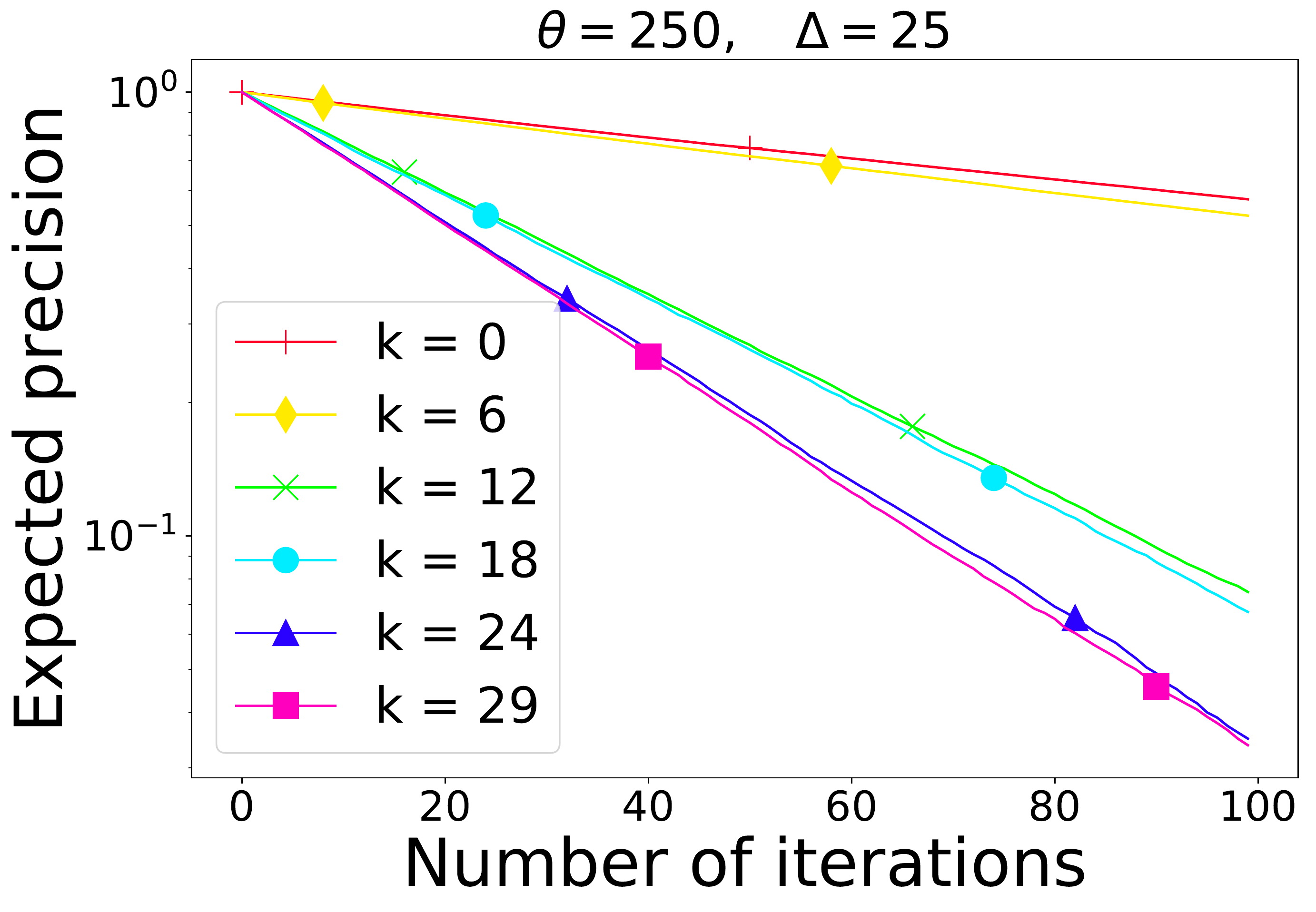}}
	\subfloat{\includegraphics[width=0.25\textwidth]{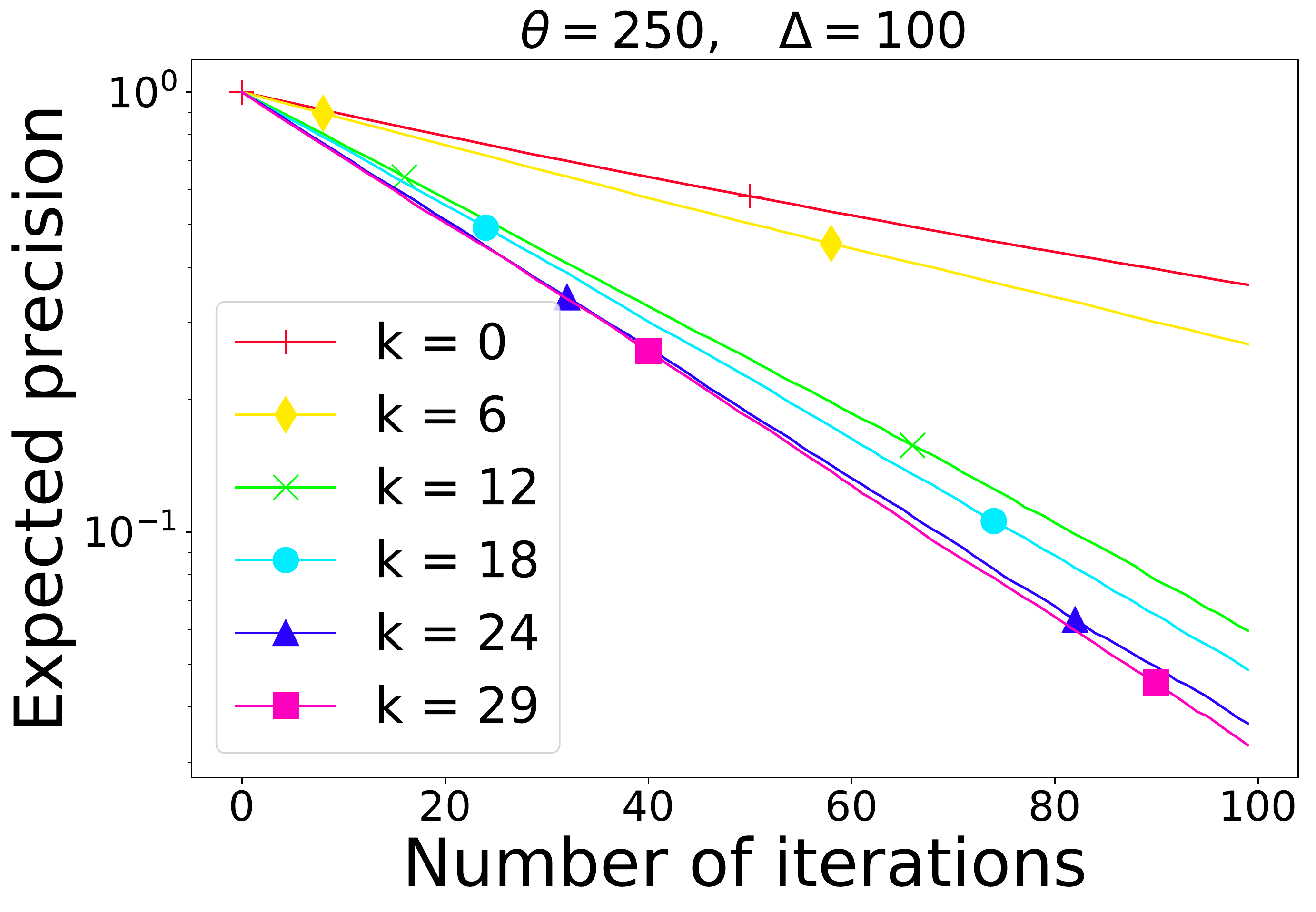}}\\
	\subfloat{\includegraphics[width=0.25\textwidth]{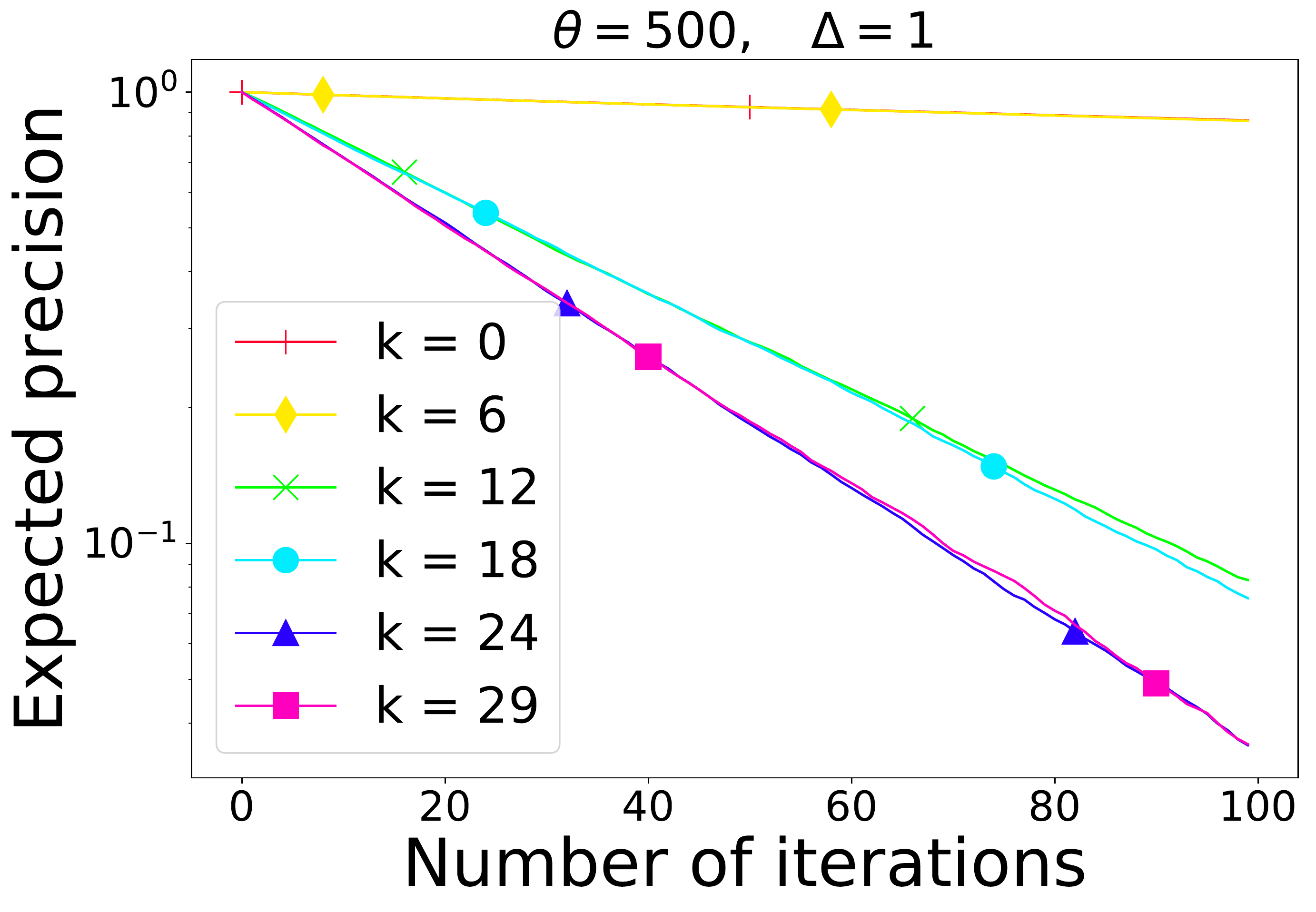}}
	\subfloat{\includegraphics[width=0.25\textwidth]{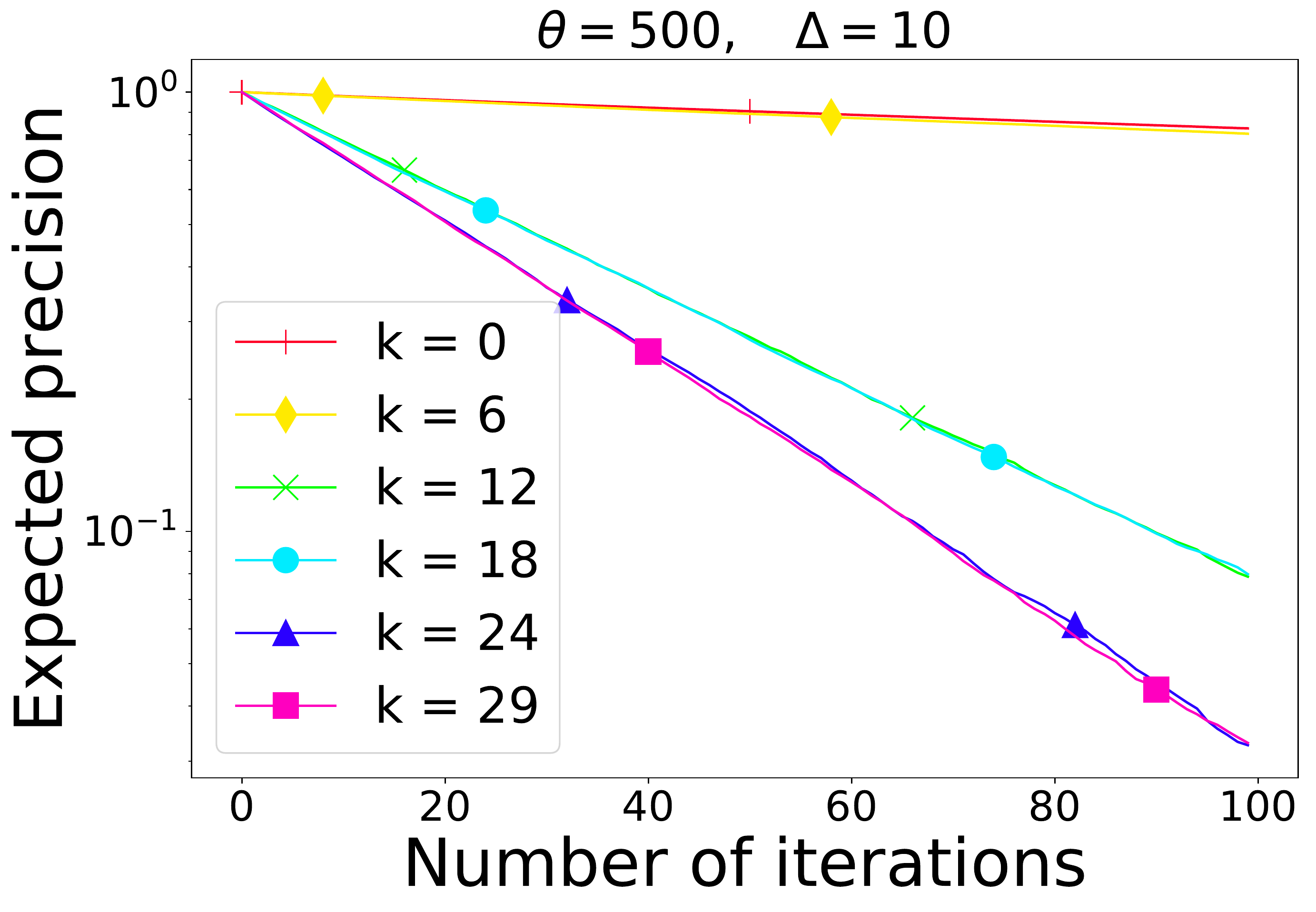}}
	\subfloat{\includegraphics[width=0.25\textwidth]{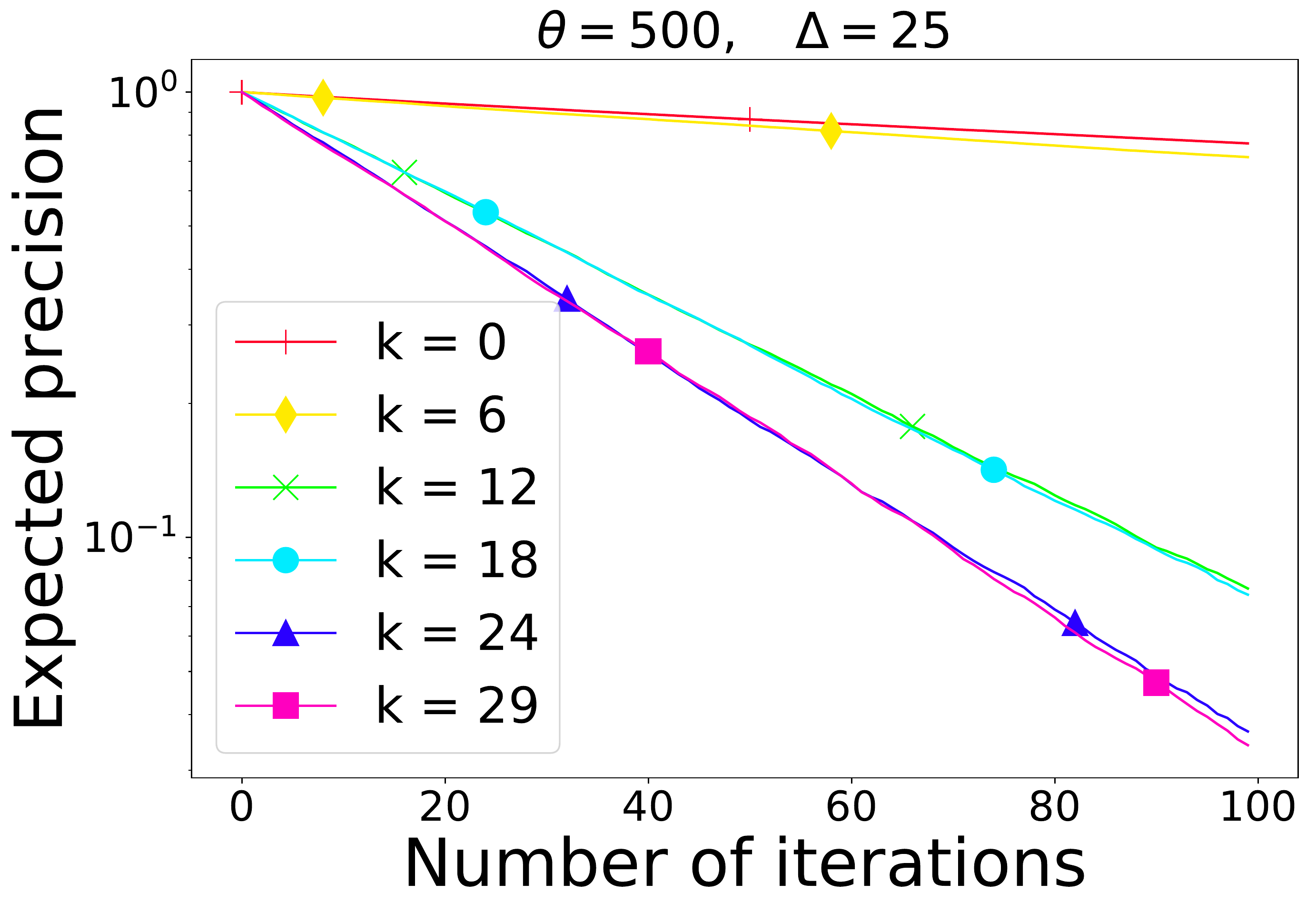}}
	\subfloat{\includegraphics[width=0.25\textwidth]{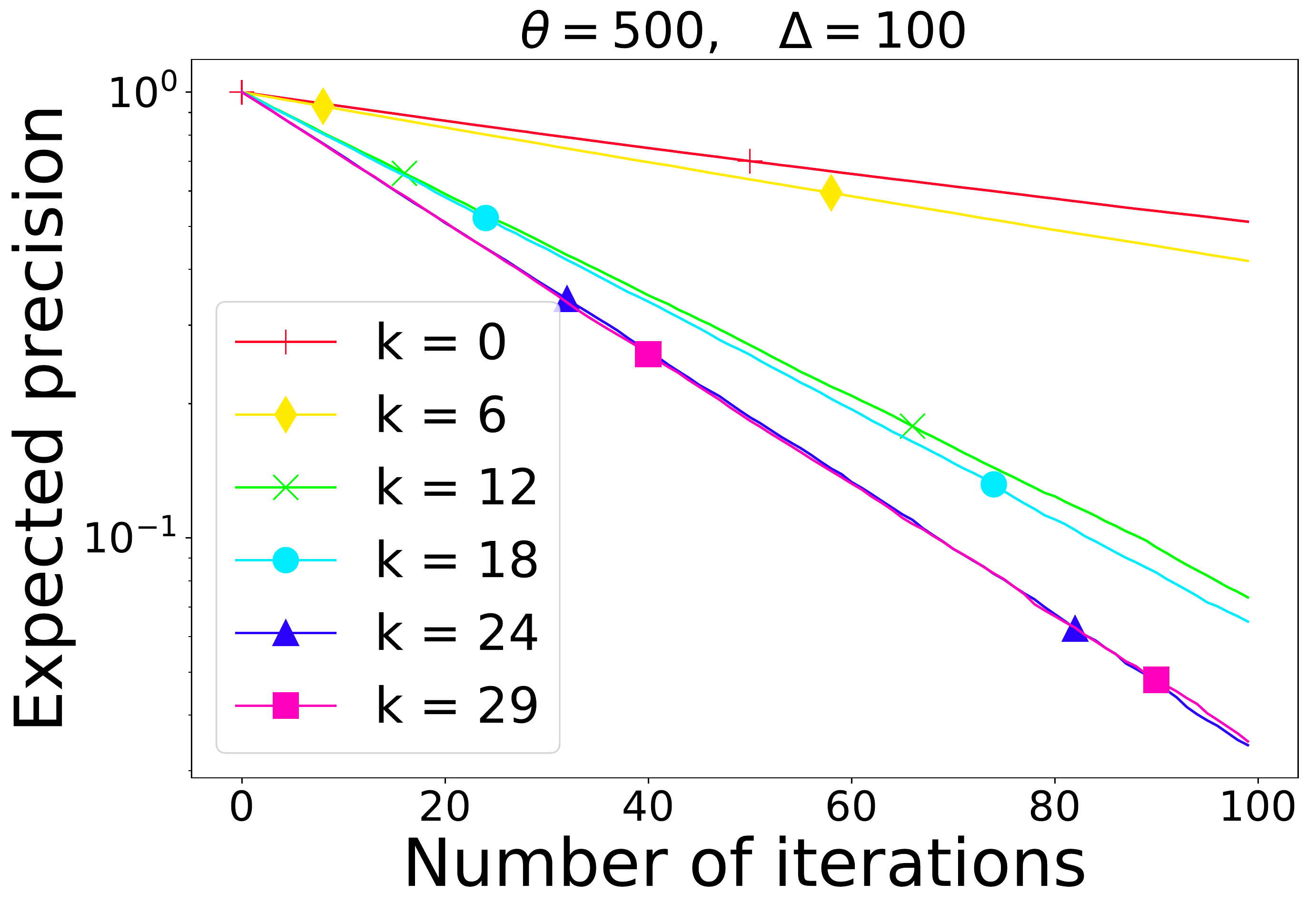}}\\
	\subfloat{\includegraphics[width=0.25\textwidth]{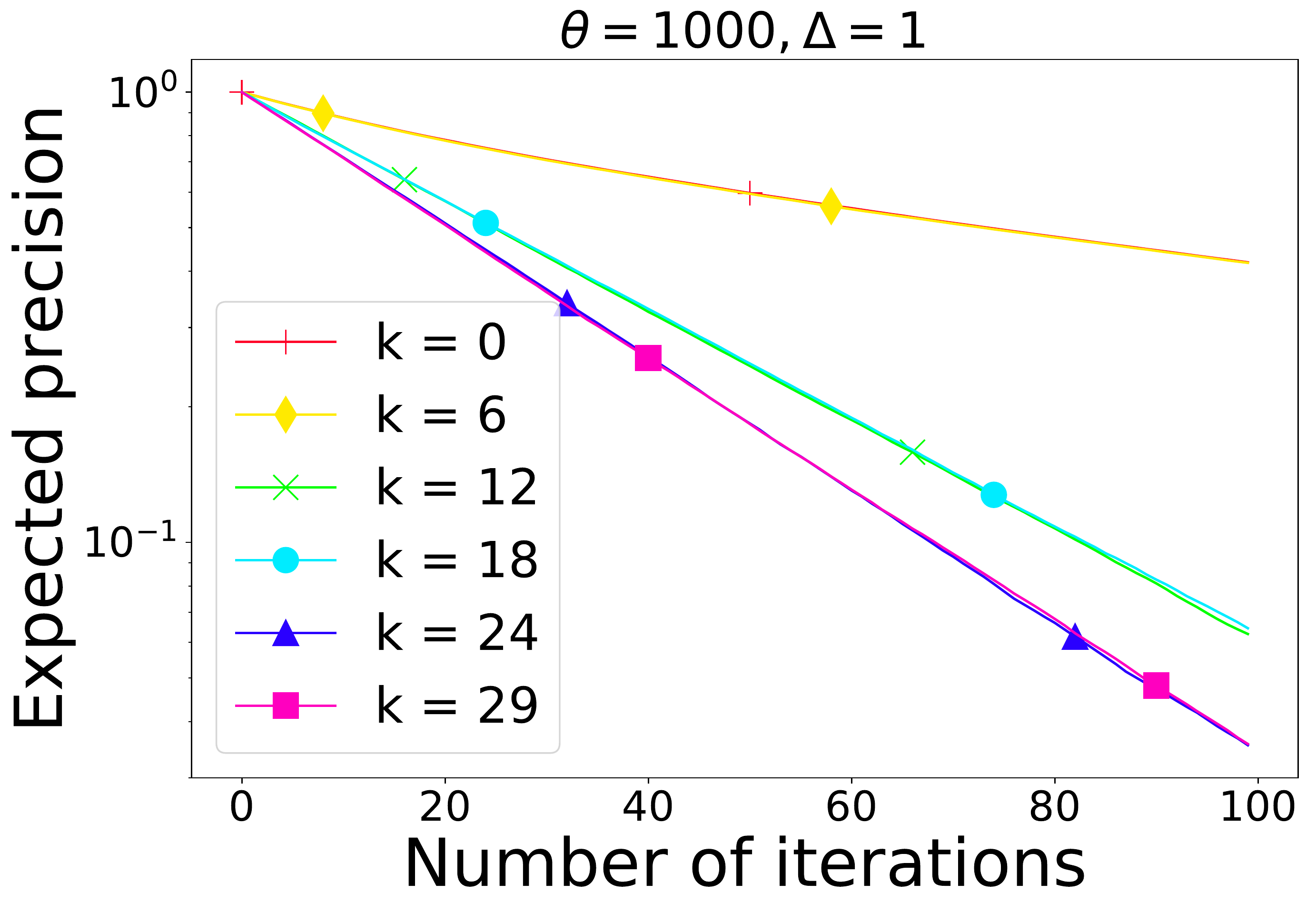}}
	\subfloat{\includegraphics[width=0.25\textwidth]{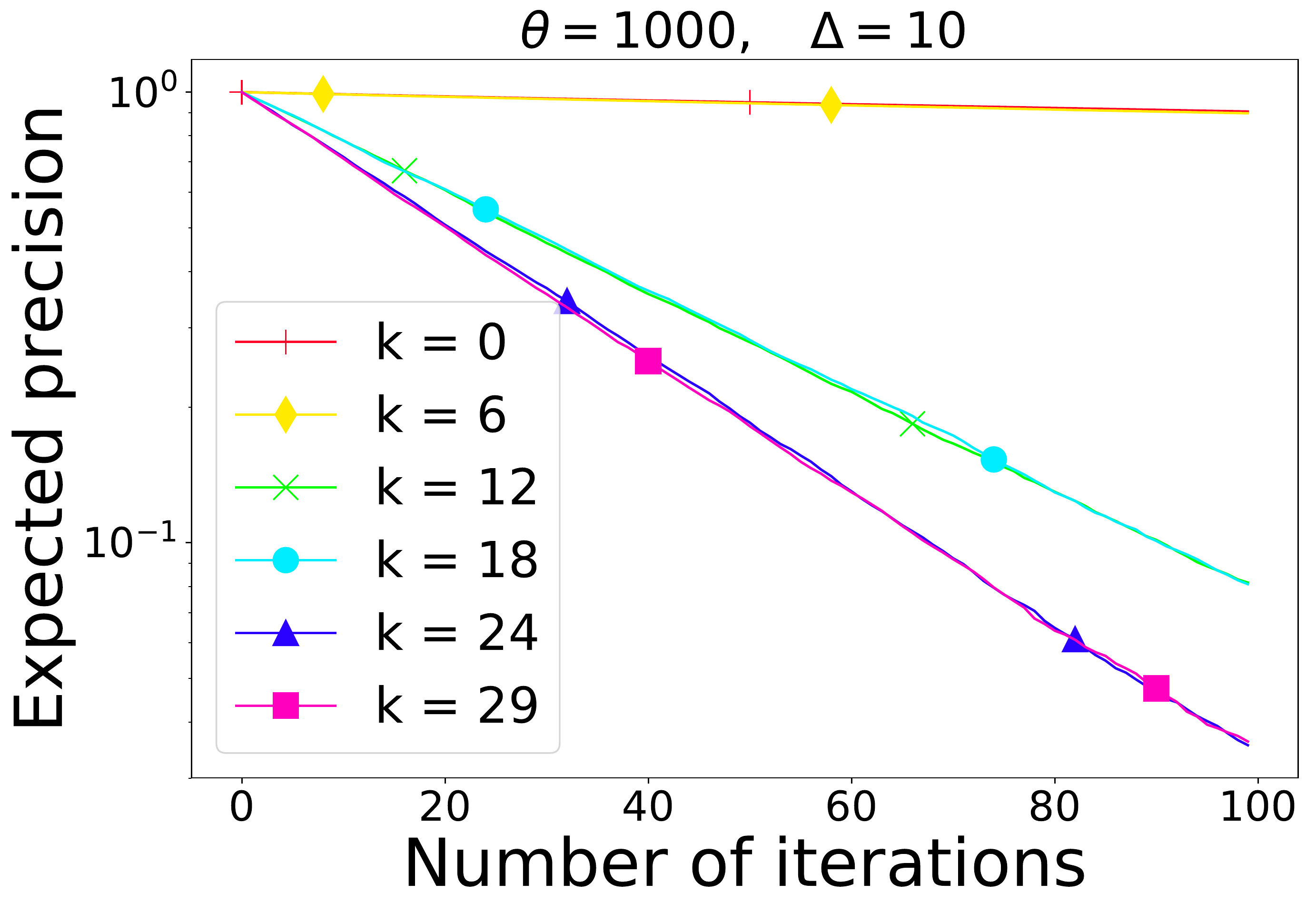}}
	\subfloat{\includegraphics[width=0.25\textwidth]{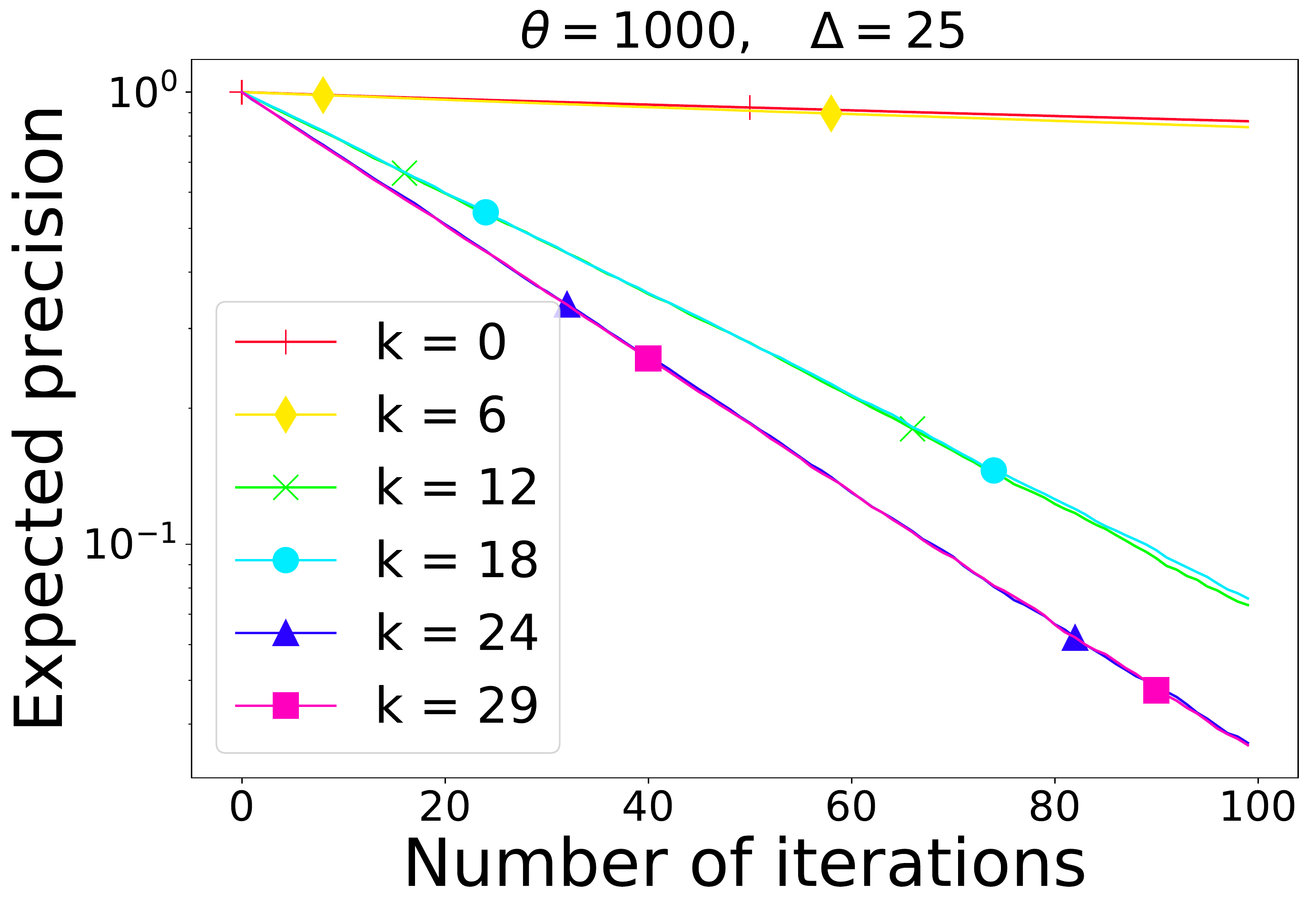}}
	\subfloat{\includegraphics[width=0.25\textwidth]{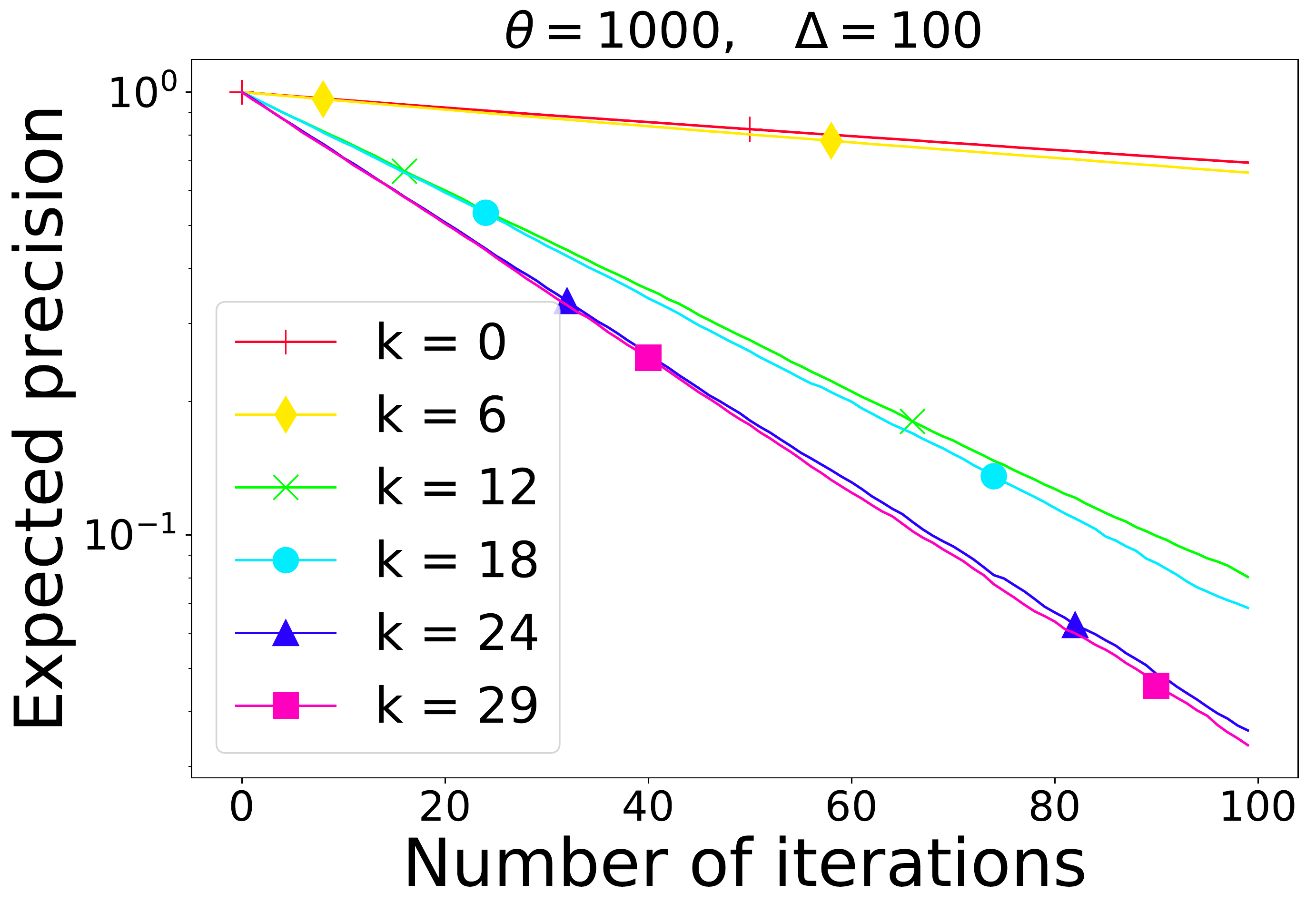}}\\
	\caption{Expected precision $\Exp\left[\frac{||x_t - x_\ast||^2_\mA}{||x_0 - x_\ast||^2_\mA}\right]$ versus the number of iterations of SSCD for symmetric positive definite matrices $\mA$ of size $30\times 30$ with different structures of spectrum. The spectrum of $\mA$ consists of 3 equally sized clusters of eigenvalues; one in the interval $(10, 10 + \Delta)$, the second in the interval $(\theta, \theta + \Delta)$ and the third in the interval $(2\theta, 2\theta + \Delta)$. We show results for 16 combinations of  $\theta$ and $\Delta$: $\Delta \in \{1, 10, 25, 100\}$ and  $\theta\in \{100, 250, 500, 1000\}$. }
	\label{fig:3clusters}
\end{figure*}

In Figure~\ref{fig:3clusters} we report on experiments similar to those performed in Section~\ref{sec:8hs98h89dhfffKK}, but on data matrix $\mA\in \R^{30\times 30}$ whose eigenvalues belong to three clusters, with 10 eigenvalues in each.  We can observe that the  SSCD methods  can be grouped into three categories: slow, fast, and very fast, depending on whether  $k$ corresponds to the smallest 10 eigenvalues,  the next cluster of 10 eigenvalues, or the 10 largest eigenvalues. That is, there are {\em two phase transitions}.

\subsection{Exponentially decaying eigenvalues}

We now consider matrix $\mA\in \R^{10\times 10}$ with eigenvalues $2^0, 2^1, \dots, 2^{9}$. We apply SSCD with increasing values of $k$  (see Figure~\ref{fig:expdecay}). 

\begin{figure*}[!h]
	\centering
	\includegraphics[width=0.5\textwidth]{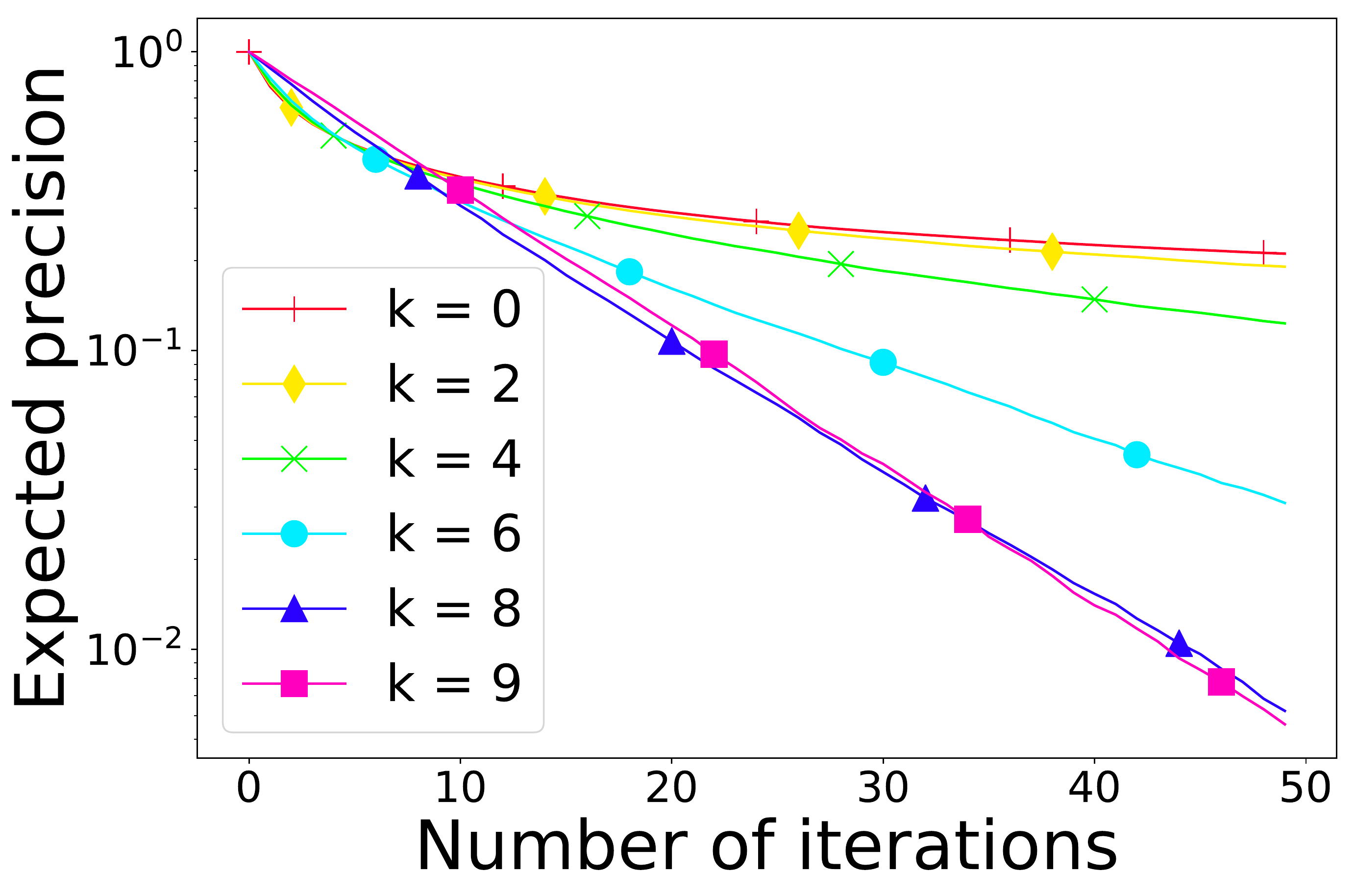}
	\caption{Expected precision $\Exp\left[\frac{||x_t - x_\ast||^2_\mA}{||x_0 - x_\ast||^2_\mA}\right]$ versus the number of iterations of SSCD for symmetric positive definite matrix $\mA$ of size $10\times 10$.}
	\label{fig:expdecay}
\end{figure*}

We can see that the {\em performance boost accelerates as $k$ increases}. So, while one may not expect much speed-up for very small $k$, there will be substantial speed-up for moderate values of $k$.  
This is {\em predicted} by our theory. Indeed, consulting Table~\ref{tbl:regimes} (last column), we have $\alpha=1/2$, and hence for $k=0$ the theoretical rate is 
$\tilde{\cO}(\tfrac{1}{\alpha^{9}}) $. For general $k$ we have $\tilde{\cO}(\tfrac{1}{\alpha^{9-k}})$. So, the speedup for value $k>0$ compared to the baseline case of $k=0$ (=RCD) is $2^k$, i.e.,  {\em exponential}.

\section{Proofs}

In this section we provide proofs of the statements from the main body of the paper. Table~\ref{tbl:guide} provides a guide on where the proof of the various results can be found.

\begin{table}[!h]
\begin{center}
\begin{tabular}{|c|c|}
\hline
Result & Section \\
\hline
\hline
Lemma~\ref{lem:rate_of_SD} & \ref{sec:98s8hf} \\
Theorem~\ref{thm:SSD} & \ref{sec:ih89s9f09} \\
Theorem~\ref{thm:n=2} & \ref{sec:iu98dg99sf} \\
Theorem~\ref{thm:unif_prob_can_be_opt} &  \ref{sec:i890f9hd900s7}\\
Theorem~\ref{thm:imp_prob_can_be_bad} & \ref{sec:i8h889gs8783b}\\
Theorem~\ref{thm:opt_probs_are_bad_UPPER} &  \ref{sec:h98gshhg729kdIJ}\\
Theorem~\ref{thm:opt_probs_are_bad_LOWER} &  \ref{sec:bujfYVyvd69Bj} \\
 Theorem~\ref{thm:SSCD} & \ref{sec:SSCD_proof} \\
Lemma~\ref{lem:rate_of_parallel_SD}& \ref{sec:j89d8ihd9JJGGF}\\
 Theorem~\ref{thm:Par_SSCD} & \ref{app:parallel} \\
\hline
\end{tabular}
\end{center}
\caption{Proof of lemmas and theorems stated in the main paper.}
\label{tbl:guide}
\end{table}

\subsection{Proof of Lemma~\ref{lem:rate_of_SD}} \label{sec:98s8hf}

The result follows from Theorem 4.8(i) in \cite{richtarik2017stochastic}  with the choice $\mB=\mA$. Note that since $x_*=\mA^{-1}b$ is the unique solution of $\mA x= b$, it is equal to the projection of $x_0$ onto the solution space of $\mA x = b$, as required by the assumption in Theorem 4.8(i). It only remains to check that Assumption 3.5 (exactness) in \cite{richtarik2017stochastic} holds. In view of Theorem 3.6(iv) in \cite{richtarik2017stochastic}, it suffices to check that the nullspace of $\Exp[\mH]$ is trivial.  However, this is equivalent to the assumption in Lemma~\ref{lem:rate_of_SD} that $\Exp[\mH]$ be invertible.

Finally,  observe that \begin{eqnarray*}\tfrac{1}{2}\|x-x_*\|_{\mA}^2 &=& \tfrac{1}{2}(x-x_*)^\top \mA (x-x_*) \quad = \quad \tfrac{1}{2}x^\top \mA x + \tfrac{1}{2}x_*^\top \mA x_* - x^\top \mA x_* \\
&= &  \tfrac{1}{2}x^\top \mA x + \tfrac{1}{2}x_*^\top \mA x_* - x^\top \mA \mA^{-1}b \quad \overset{\eqref{eq:quad_opt} }{=} \quad f(x) +\tfrac{1}{2}x_*^\top \mA x_*\\
&=& f(x) - f(x_*).\end{eqnarray*}

\subsection{Proof of Theorem~\ref{thm:SSD}} \label{sec:ih89s9f09}

We will break down the proof into three steps.

\begin{enumerate}
\item First, let us show that Algorithm~\ref{alg:SSD} is indeed SSD, as described in \eqref{alg_lin_sd}, i.e.,  
$
x_{t+1} = x_t -   \frac{ s_t^\top(\mA x_t - b)}{ s_t^\top \mA  s_t}s_t  .
$
We known that $s_t = u_i$ with probability $1/n$. Since $\mA u_i =\lambda_i u_i$, and assuming that at iteration $t$ we have $s_t=u_i$, we get
\begin{eqnarray*}
x_{t+1}&=& x_t - \frac{ u_i^\top(\mA x_t - b)}{ u_i^\top \mA  u_i} u_i  \quad = \quad   x_t - \frac{ u_i^\top(\mA x_t - b)}{ \lambda_i} u_i \\
&=& x_t - \frac{ \lambda_i u_i^\top x_t - u_i^\top b}{ \lambda_i} u_i  \quad = \quad x_t - \left( u_i^\top x_t - \tfrac{u_i^\top b}{\lambda_i} \right)   u_i.
\end{eqnarray*}

\item We now need to argue that the assumption that $\Exp[\mH]$ is invertible is satisfied. 
\begin{equation}\label{eq:8g9g8db98hsg8s}\Exp[\mH] \overset{\eqref{eq:H}}{=} \sum_{i=1}^n \frac{1}{n} \frac{u_i u_i^\top}{u_i^\top \mA u_i}  =\sum_{i=1}^n \frac{1}{n} \frac{u_i u_i^\top}{\lambda_i}. \end{equation}
Since $\Exp[\mH]$ has positive eigenvalues $1/(n \lambda_i)$, it is invertible.

\item Applying Lemma~\ref{lem:rate_of_SD}, we get
\[ (1-\lambda_{\max}(\mW))^t \Exp[\|x_0-x_*\|_{\mA}^2] \leq \Exp[\|x_{t}-x_*\|_{\mA}^2] \leq (1-\lambda_{\min}(\mA))^t \Exp[\|x_0-x_*\|_{\mA}^2].\]
It remains to show that $\lambda_{\min}(\mW) = \lambda_{\max}(\mW) = \tfrac{1}{n}$. In view of \eqref{eq:8g9g8db98hsg8s}, and since $\mA^{1/2}u_i = \sqrt{\lambda_i} u_i$, we get
\[\mW \overset{\eqref{eq:W}}{=} \mA^{1/2} \Exp[\mH] \mA^{1/2} \overset{\eqref{eq:8g9g8db98hsg8s}}{=}  \mA^{1/2} \sum_{i=1}^n \frac{1}{n} \frac{u_i u_i^\top}{\lambda_i} \mA^{1/2} =  \sum_{i=1}^n \frac{1}{n} \frac{\mA^{1/2} u_i u_i^\top \mA^{1/2}}{\lambda_i}  = \frac{1}{n} \mI. \]

\end{enumerate}

\subsection{Proof of Theorem~\ref{thm:n=2}} \label{sec:iu98dg99sf}

Let $\mA$ be a $2\times2$ symmetric  positive definite matrix:
\[
\mA =
\begin{pmatrix}
a & c \\
c & b \\
\end{pmatrix}.
\]
We know that $a, b > 0$, and $ab - c^2 > 0$. Assume that $s_t = e_1 = (1,0)^\top$ with probability $p>0$ and $s_t = e_2 = (0, 1)^\top$ with probability $q>0$, where $p+q=1$. Then
\[
\Exp[\mH] \overset{\eqref{eq:H}}{=} p \frac{e_1 e_1^\top}{e_1^\top \mA e_1} + q \frac{e_2 e_2^\top}{e_2^\top \mA e_2}  =
\begin{pmatrix}
\frac pa & 0\\
0 & \frac qb \\
\end{pmatrix},
\]
and therefore,
\[
\Exp[\mH]\mA =
\begin{pmatrix}
p & p \frac ca \\
q \frac cb & q \\
\end{pmatrix}.
\]
Note that $\Exp[\mH]
\mA$ has the same eigenvalues as $\mW = \mA^{1/2}\Exp[\mH] \mA^{1/2}$. We now find the eigenvalues of $\Exp[\mH]\mA$ by finding the zeros of the characteristic polynomial:
\[
\textrm{det}(\Exp[\mH]\mA - \lambda \mI) = \textrm{det}
\begin{pmatrix}
p - \lambda & p \frac ca\\
q \frac cb & q -\lambda \\
\end{pmatrix} = \lambda^2 - \lambda + pq \left(1 - \frac{c^2}{ab}\right)= 0
\]

It can be seen that

$$
\lambda_{\min}(\Exp[\mH]\mA) = \frac 12 - \frac 12 \sqrt{1 - 4pq\left(1 - \frac{c^2}{ab}\right)}
= \frac 12 - \frac 12 \sqrt{1 - 4p(1-p)\left(1 - \frac{c^2}{ab}\right)}
.$$

The expression $\lambda_{\min}(\Exp[\mH]\mA)$ is maximized for  $p = \frac 12$, independently of the values of $a, b$ and $c$.

\subsection{Proof of Theorem~\ref{thm:unif_prob_can_be_opt}} \label{sec:i890f9hd900s7}

Fix $n\geq 2$, and let $\Delta_n^+ \eqdef \{p \in \R^n \;:\; p>0, \; \sum_i p_i = 1\}$ be the (interior of the) probability simplex. Further, let $\mA = \textrm{Diag}(\mA_{11}, \mA_{22}, \dots, \mA_{nn})$ be a diagonal matrix with positive diagonal entries. 

The rate of RCD with any probabilities arises as a special case of Lemma~\ref{lem:rate_of_SD}. We therefore need to study the smallest eigenvalue of $\mW$ (defined in \eqref{eq:W}) as a function of $p=(p_1,\dots,p_n)$. We have
\[
\mH(p) \eqdef \Exp_{s\sim \cD}[\mH]  \overset{\eqref{eq:H}}{=} \sum_i \frac{p_i}{\mA_{ii}} e_i e_i^\top = \textrm{Diag}(p_1/\mA_{11}, p_2/\mA_{22}, \dots,p_n/\mA_{nn} ),
\] and hence
\begin{equation} \label{eq:89sg08b98}
\mW  \overset{\eqref{eq:W}}{=}  \mW(p) \eqdef \mA^{1/2}\mH(p)\mA^{1/2} = \sum\limits_{i=1}^n p_i e_i e_i^\top =
\begin{pmatrix}
p_1 & 0& \dots \\
0 & p_2 & \dots \\
\dots & \dots & \ddots \\
0 & 0 & \dots & p_n \\
\end{pmatrix} .
\end{equation}
Note that
$\lambda_{\min}(\mW(p)) \overset{\eqref{eq:89sg08b98}}{=}\lambda_{\min}({\rm Diag}(p_1,p_2,\dots,p_n)) =   \min_i p_i ,$
and thus
\[\max_{p\in \Delta_n^+} \lambda_{\min}(\mW(p))  = \frac{1}{n}.\]
Clearly, the optimal probabilities are uniform: $p_i^* =\tfrac{1}{n}$ for all $i$.

\subsection{Proof of Theorem~\ref{thm:imp_prob_can_be_bad}} \label{sec:i8h889gs8783b}

We continue from the proof of Theorem~\ref{thm:unif_prob_can_be_opt}.

\begin{enumerate}
\item Consider probabilities proportional to the diagonal elements: $p_i = \mA_{ii}/{\rm Tr}(\mA)$ for all $i$. Choose $\mA_{11} \eqdef t$, and $\mA_{22} = \cdots = \mA_{nn} = 1$.  Then

$$
\lambda_{\min}(\mW(p)) \leq p_2 =  \frac{\mA_{22}}{{\rm Tr}(\mA)}= \frac{1}{t + n - 1} \longrightarrow 0 \ \text{as} \ t \longrightarrow \infty
.$$

\item  Consider probabilities proportional to the squared row norms: $p_i = \|\mA_{i:}\|^2/{\rm Tr}(\mA^\top \mA)$ for all $i$. Choose $\mA_{11}\eqdef t$, and $\mA_{22} = \cdots = \mA_{nn} = 1$. Then
$$
\lambda_{\min}(\mW(p)) \leq p_2 = \frac{\mA_{22}}{{\rm Tr}(\mA^\top \mA)} =\frac{1}{t^2 + n - 1} \longrightarrow 0 \ \text{as} \ t \longrightarrow \infty.
$$
\end{enumerate}

In both cases, $\frac{\lambda_{\min}(\mW(p))}{\lambda_{\min}(\mW(p^*))}$ can be made arbitrarily small by a suitable choice of $t$.

\subsection{Proof of Theorem~\ref{thm:opt_probs_are_bad_UPPER} } \label{sec:h98gshhg729kdIJ}

The rate of RCD with any probabilities arises as a special case of Lemma~\ref{lem:rate_of_SD}. We therefore need to study the smallest eigenvalue of $\mW$ (defined in \eqref{eq:W}). Since we wish to show that the rate can be bad, we will first prove a lemma bounding $\lambda_{\min}(\mW)$ from above.

\begin{lemma} \label{lem:ineq_smallest_eigen} Let $0<\lambda_1\leq \lambda_2 \leq \dots \leq \lambda_n$ be the eigenvalues of $\mA$. Then
\begin{equation} \label{eq:ineq_smallest_eigen} \lambda_{\min}(\mW) \leq \frac{1}{n} \left(
	\prod\limits_{k=1}^{n} \frac{\lambda_k}{\mA_{kk}}\right)^{1/n} .\end{equation}
\end{lemma}
 \begin{proof}
We have
\[
\mW \overset{\eqref{eq:W}}{=}
\mA^{\frac{1}{2}} \E{\mH} \mA^{\frac{1}{2}} \overset{\eqref{eq:H}}{=}
\mA^{\frac{1}{2}}
\left(
\sum\limits_{k=1}^{n}
\frac{p_k e_ke_k^\top}{\mA_{kk}}
\right)
\mA^{\frac{1}{2}}=
\mA^{\frac{1}{2}}
\mathrm{Diag}\left(\frac{p_k}{\mA_{kk}} \right)
\mA^{\frac{1}{2}}.
\]

From the above we see that the determinant of $\textbf{W}$ is given by
\begin{equation} \label{eq:detW}
\det (\textbf{W}) =
\det(\mA) \prod\limits_{k=1}^{n}\frac{p_k}{\mA_{kk}}.
\end{equation}
On the other hand, we have the trivial bound
\begin{equation} \label{eq:det_bound}\det (\textbf{W}) = \prod\limits_{k=1}^{n} \lambda_k(\textbf{W})\geq
(\lambda_{\min}(\textbf{W}))^n.
\end{equation}
Putting these together, we get
an upper bound on $\lambda_{\min}(\textbf{W})$ in terms of the eigenvalues  and  diagonal elements of $\mA$:
\begin{eqnarray*}
\lambda_{\min}(\textbf{W})  \overset{\eqref{eq:det_bound}}{ \leq }  \sqrt[n]{\det(\textbf{W})} & \overset{\eqref{eq:detW}}{=} &
\sqrt[n]{\det(\mA) } \cdot \sqrt[n]{ \prod\limits_{k=1}^n \frac{p_k}{\mA_{kk}} }\\
& = &
\sqrt[n]{\det(\mA) } \cdot \sqrt[n]{\prod_{k=1}^n \frac{1}{\mA_{kk}}}  \cdot \sqrt[n]{\prod\limits_{k=1}^n p_k}\\
& \overset{(*)}{\leq} &
\sqrt[n]{\det(\mA) } \cdot \sqrt[n]{\prod_{k=1}^n \frac{1}{\mA_{kk}}}  \cdot\frac{\sum_{k=1}^n p_k}{n}\\
& = &
\frac{\sqrt[n]{\det(\mA) }}{n} \cdot \sqrt[n]{\prod\limits_{k=1}^n \frac{1}{\mA_{kk}}} \\
&\overset{\eqref{eq:det_bound}}{=}&
\frac{1}{n} \sqrt[n]{
	\prod\limits_{k=1}^{n} \frac{\lambda_k}{\mA_{kk}}},
\end{eqnarray*}
where (*) follows from the arithmetic-geometric mean inequality.
 \end{proof}
 
\paragraph{The Proof:}

Let $\lambda_1, \dots, \lambda_n$ are any positive real numbers.
We now construct matrix $\mA = \mM \Lambda \mM^\top$, where
$
\Lambda \eqdef \mathrm{Diag}(\lambda_1, \ldots, \lambda_n)$  and
\[ \mM \eqdef
\begin{pmatrix}
1 / \sqrt{2}& 1/ \sqrt{2} & 0&\cdots & 0\\
-1/ \sqrt{2}&1/ \sqrt{2} & 0&\cdots & 0\\
0&0&1& \cdots & 0\\
\vdots&\vdots&\vdots& \ddots & 0\\
0&0&0&\cdots& 1
\end{pmatrix} \in \R^{n \times n}.
\]
Clearly, $\mA$ is symmetric. Since $\mM$ is orthonormal,  $\lambda_1, \dots, \lambda_n$ are, by construction, the eigenvalues of $\mA$. Hence, $\mA$ is symmetric and positive definite. Further, note that the diagonal entries of $\mA$ are related to its eigenvalues as follows:  \begin{equation}\label{eq:iu9g98g98ss}\mA_{kk} = \begin{cases}
\frac{\lambda_1 + \lambda_2}{2},&k = 1,2;\\
\lambda_k,& \text{otherwise.}
\end{cases}
\end{equation}
Applying  Lemma~\ref{lem:ineq_smallest_eigen}, we get the bound
\begin{eqnarray*}
\lambda_{\min}(\mW) &\overset{\eqref{eq:ineq_smallest_eigen} }{\leq}& \frac{1}{n} \left(
	\prod\limits_{k=1}^{n} \frac{\lambda_k}{\mA_{kk}}\right)^{1/n} \\
	&=& \frac{1}{n} \left(
	\prod\limits_{k=1}^{2} \frac{\lambda_k}{\mA_{kk}} \cdot \prod\limits_{k=3}^{n} \frac{\lambda_k}{\mA_{kk}} \right)^{1/n} \\
	&\overset{\eqref{eq:iu9g98g98ss}}{=}&  \frac{1}{n} \left(
	\prod\limits_{k=1}^{2} \frac{\lambda_k}{\mA_{kk}}  \right)^{1/n} \\
	&\overset{\eqref{eq:iu9g98g98ss}}{=}& \frac{1}{n}  \left( \frac{4 \lambda_1 \lambda_2}{(\lambda_1+\lambda_2)^2}
	  \right)^{1/n}.
\end{eqnarray*}

Let $c>0$ be such that $\lambda_1 = c \lambda_2$. Then $\frac{4 \lambda_1 \lambda_2}{(\lambda_1 + \lambda_2)^2} = \frac{4c}{(1+c)^2}$. If choose $c$ small enough so that  $\frac{4c}{(1+c)^2} \leq \left(\frac{n}{T}\right)^n$, then $\lambda_{\min}(\mW) \leq \frac{1}{T}$. The statement of the theorem follows.



\subsection{Proof of Theorem~\ref{thm:opt_probs_are_bad_LOWER} } \label{sec:bujfYVyvd69Bj}

Let $\mW = \mU{\bf \Lambda} \mU^{\top}$ be the eigenvalue decomposition of $\mW$, where $\mU=[u_1,\ldots, u_n]$ are the eigenvectors, $\lambda_{1}(\mW) \leq \ldots \leq \lambda_{n}(\mW)$ are the eigenvalues and ${\bf \Lambda} = \diag{\lambda_{1}(\mW), \ldots, \lambda_{n}(\mW)}$.
From Theorem~4.3 of \cite{richtarik2017stochastic} we get
\begin{equation}\label{eq:thm_4.3_reference}
\E{\mU^{\top} \mA^{1/2}(x_t - x_*)} =
(\mI - {\bf \Lambda})^t \mU^{\top} \mA^{1/2} (x_0 - x_*).
\end{equation}
Now we use Jensen's inequality and get
\begin{eqnarray}\label{eq:jensen_ineq}
\E{\norm{x_t - x_*}_{\mA}^2} &=&
\E{\norm{\mU^{\top} \mA^{1/2}(x_t - x_*)}_2^2} \geq 
\norm{\E{\mU^{\top} \mA^{1/2}(x_t - x_*)}}_2^2 \overset{\eqref{eq:thm_4.3_reference}}{=}
\norm{(\mI - {\bf \Lambda})^t \mU^{\top} \mA^{1/2} (x_0 - x_*)}_2^2\\
&=&\Sum{i=1}{n} (1-\lambda_i(\mW))^{2t}
\left(
u_i^{\top} \mA^{1/2}(x_0 - x_*)
\right)^2 
\geq
(1-\lambda_1(\mW))^{2t}
\left(
u_1^{\top} \mA^{1/2}(x_0 - x_*)
\right)^2.
\end{eqnarray}

Now we take an example of matrix $\mA$,
for which we set $\lambda_{\min}(\mW) \leq \frac{1}{T}$ for arbitrary $T>0$, like we did in Section~\ref{sec:h98gshhg729kdIJ}.
We also choose $x_0 = x_* + \mA^{-1/2} u_1$. For this choice of $\mA$ and $x_0$ we get
$\norm{x_0 - x_*}_{\mA}^2 = \norm{u_1}_2^2$ and
\begin{equation}
\E{\norm{x_t - x_*}^2_{\mA}} \geq (1-\lambda_{1}(\mW))^{2t} \norm{u_1}_2^2
\geq \left(1 - \frac{1}{T}\right)^{2t} \norm{u_1}^2_2 = 
\left(1 - \frac{1}{T}\right)^{2t} \norm{x_0 - x_*}_{\mA}^2.
\end{equation}

\subsection{Proof of Theorem~\ref{thm:SSCD}}\label{sec:SSCD_proof}

We divide the proof into several steps.

\begin{enumerate}
\item Let us first show that SSCD converges with a linear rate for any choice of $\alpha>0$ and nonnegative $\{\beta_i\}$. Since SSCD arises as a special case of SD, it suffices to apply Lemma~\ref{lem:rate_of_SD}. In order to apply this lemma, we need to argue that $\cD=\cD(\alpha, \beta_1,\dots,\beta_n)$ is a proper distribution. Indeed,

\begin{align}
\Exp_{s\sim \cD}[\mH]  & \overset{\eqref{eq:H} }{=}  \sum_{i=1}^n p_i \frac{e_i e_i^\top}{e_i^\top \mA e_i} + \sum_{i=1}^k p_{n+i} \frac{u_i u_i^\top}{ u_i^\top \mA u_i} \notag \\
 &=  
\frac{1}{C_k}\left(\alpha\mI + \sum\limits_{i=1}^ku_iu_i^\top\frac{\beta_i}{\lambda_i} \right) \label{fam:expect_H} \\
&\succeq   \frac{\alpha}{C_k} \mI \quad \succ \quad 0.  \notag
\end{align}

\item For the specific choice of parameters $\alpha=1$ and $\beta_i = \lambda_{k+1}-\lambda_i$ we have 
\[
\Exp_{s\sim \cD}[\mH] =
\frac{1}{C_k}
\left(
\mI + \Sum{i=1}{k}u_iu_i^\top
\frac{\lambda_{k+1} - \lambda_i}{\lambda_i}
\right),
\]
and
$
C_k = (k+1)\lambda_{k+1} + \Sum{i=k+2}{m}\lambda_i.
$
Therefore,
\[
\Exp_{s\sim \cD}[\mA\mH] = \frac{1}{C_k}
\left(
\Sum{i=1}{k} \lambda_{k+1} u_i u_i^\top +
\Sum{i=k+1}{n} \lambda_i u_i u_i^\top
\right).
\]
The minimal eigenvalue of this matrix, which has the same spectrum as $\mW$, is
\[
\lambda_{\min}(\Exp_{s\sim \cD}[\mA\mH]) = \frac{\lambda_{k+1}}{C_k} =
\frac{\lambda_{k+1}}{(k+1)\lambda_{k+1} + \Sum{i=k+2}{n}\lambda_i}.
\]
The main statement follows by applying Lemma~\ref{lem:rate_of_SD}.

\item We now show that the rate improves as $k$ increases.
Indeed,
\[
k + \frac{1}{\lambda_{k+1}}\Sum{i=k+1}{m}\lambda_i =
k + 1 + \frac{1}{\lambda_{k+1}} \Sum{i=k+2}{m}\lambda_i \geq
k + 1 + \frac{1}{\lambda_{k+2}} \Sum{i=k+2}{m}\lambda_i.
\]
By taking reciprocals, we get
\[
\frac{\lambda_{k+2}}{(k+1)\lambda_{k+2} + \Sum{i=k+2}{m}\lambda_i} \geq
\frac{\lambda_{k+1}}{k\lambda_{k+1} + \Sum{i=k+1}{m}\lambda_i}.
\]

\item It remains to establish optimality of the specific parameter choice $\alpha=1$ and $\beta_i=\lambda_{k+1}-\lambda_i$. Continuing from \eqref{fam:expect_H}, we get
\begin{equation}\label{fam:expect_AH}
\Exp_{s\sim\cD}[\mA\mH] \overset{\eqref{fam:expect_H}}{=} \frac{1}{C_k}\left(\sum\limits_{i=1}^nu_iu_i^\top \alpha\lambda_i + \sum\limits_{i=1}^ku_iu_i^\top\beta_i\right) = \frac{1}{C_k}\left(\sum_{i=1}^{k}(\alpha\lambda_i+\beta_i)u_iu_i^\top + \sum\limits_{i=k+1}^n\alpha\lambda_iu_iu_i^\top\right).
\end{equation}
The eigenvalues of $\Exp_{s\sim\cD}[\mA\mH] $ are $\{ \frac{\alpha \lambda_i + \beta_i}{C_k}\}_{i=1}^k \cup \{\frac{\alpha \lambda_i}{C_k}\}_{i=k+1}^n$. Let $\gamma$ be the smallest eigenvalue, i.e.,  $\gamma\eqdef \lambda_{\min}(\Exp_{s\sim\cD}[\mA\mH]) =  \frac{\theta}{C_k}$, and $\Omega$ be the largest eigenvalue, i.e.,  $\Omega\eqdef \lambda_{\max}(\Exp_{s\sim\cD}[\mA\mH]) = \frac{\Delta}{C_k}$, where  $\theta$ and $\Delta$ are appropriate constants.  There are now two options.
\begin{enumerate}
	\item \textbf{$\gamma = \frac{\alpha \lambda_{k+1}}{C_k}.$} Then $\alpha\lambda_i+\beta_i \geq \alpha\lambda_{k+1}$ for $i\in\{1,\ldots,k\}$. In this case we obtain:
	\begin{equation}
	C_k = \alpha\trace{\mA} + \sum_{i=1}^k\beta_i = \sum\limits_{i=1}^k(\alpha\lambda_i+\beta_i) + \alpha\sum\limits_{i=k+1}^n\lambda_i\geq \alpha\left(k\lambda_{k+1}+\sum\limits_{i=k+1}^n\lambda_i\right)
	\end{equation}
	and therefore
	\begin{equation}
	\gamma \leq \frac{\lambda_{k+1}}{k\lambda_{k+1}+\sum\limits_{i=k+1}^n\lambda_i}.
	\end{equation}
	\item $\gamma = \frac{\alpha\lambda_j+\beta_j}{C_k} = \frac{\theta}{C_k}$ for some $j\in\{1,\ldots,k\}.$ Then
	\begin{equation}
	C_k = \alpha\trace{\mA} + \sum_{i=1}^k\beta_i = \sum\limits_{i=1}^k(\alpha\lambda_i+\beta_i) + \alpha\sum\limits_{i=k+1}^n\lambda_i \geq k\theta + \alpha\sum\limits_{i=k+1}^n\lambda_i
	\end{equation}
	whence
	\begin{equation}
	\gamma \leq \frac{\theta}{k\theta+\alpha\sum\limits_{i=k+1}^n\lambda_i}.
	\end{equation}
	Note that the function $f(\theta) = \frac{\theta}{k\theta+\alpha\sum\limits_{i=k+1}^n\lambda_i}$ increases monotonically:
	\begin{equation}
	f'(\theta) = \frac{1}{k\theta+\alpha\sum\limits_{i=k+1}^n\lambda_i} - \frac{k\theta}{(k\theta+\alpha\sum\limits_{i=k+1}^n\lambda_i)^2} = \frac{\alpha\sum\limits_{i=k+1}^n\lambda_i}{(k\theta+\alpha\sum\limits_{i=k+1}^n\lambda_i)^2} > 0.
	\end{equation}
	From this and inequality $\alpha\lambda_{k+1}\geq\theta$ we get
	\begin{equation}
	\gamma \leq \frac{\alpha\lambda_{k+1}}{\alpha(k\lambda_{k+1}+\sum\limits_{i=k+1}^n\lambda_i)} = \frac{\lambda_{k+1}}{k\lambda_{k+1}+\sum\limits_{i=k+1}^n\lambda_i}.
	\end{equation}
\end{enumerate}

In both possible cases we have shown that \[\lambda_{\min}(\Exp_{s\sim\cD}[\mA\mH]) \leq \frac{\lambda_{k+1}}{k\lambda_{k+1}+\sum\limits_{i=k+1}^n\lambda_i}.\] So, it is the optimal rate in this family of methods.  Optimal distribution is unique and it is:
\begin{equation}\label{fam:optimal_distrib}
	s\sim\cD \quad \Leftrightarrow \quad s = \begin{cases}
	e_i & \text{with probability\;}  p_i = \frac{\mA_{ii}}{C_k} \quad i=1,2,\ldots,n\\
	u_i & \text{with probability\;} p_{n+i} = \frac{\lambda_{k+1}-\lambda_i}{C_k} \quad i=1,2,\ldots,k,
	\end{cases}
\end{equation}
where $C_k = k\lambda_{k+1}+\sum\limits_{i=k+1}^n\lambda_i$. 


\end{enumerate}

\subsection{Proof of Lemma~\ref{lem:rate_of_parallel_SD}} \label{sec:j89d8ihd9JJGGF}

The steps are  analogous to the proof of Lemma~\ref{lem:rate_of_SD}.

\subsection{Proof of Theorem~\ref{thm:Par_SSCD}} \label{app:parallel}

Let $C_k = (k+1)\lambda_{k+1}+\sum\limits_{i=k+2}^n\lambda_i$ $\gamma = \frac{\theta}{C_k}$~--- the minimal eigenvalue of the matrix $\mW$ and $\Omega = \frac{\Delta}{C_k}$~--- the maximal eigenvalue of the matrix $\mW$. The optimal rate of the method \cite{richtarik2017stochastic} is
\begin{equation}
	r(\tau) = \frac{\gamma}{\frac{1}{\tau} + \left(1 - \frac{1}{\tau}\right)\Omega} = \frac{\theta}{\frac{1}{\tau}C_k + \left(1-\frac{1}{\tau}\right)\Delta}.
\end{equation}
From the Section~\ref{sec:SSCD_proof} we have 
\[
\Exp_{s\sim \cD}[\mA\mH] = \frac{1}{C_k}
\left(
\Sum{i=1}{k} \lambda_{k+1} u_i u_i^\top +
\Sum{i=k+1}{n} \lambda_i u_i u_i^\top
\right).
\]
There are two options.
\begin{enumerate}
	\item \textbf{$\gamma = \frac{\alpha \lambda_{k+1}}{C_k}.$} Then $\alpha\lambda_i+\beta_i \geq \alpha\lambda_{k+1}$ for $i\in\{1,\ldots,k\}$ and $\Delta \geqslant \alpha\lambda_n$. In this case we obtain:
	\begin{equation}
	C_k = \alpha\trace{\mA} + \sum_{i=1}^k\beta_i = \sum\limits_{i=1}^k(\alpha\lambda_i+\beta_i) + \alpha\sum\limits_{i=k+1}^n\lambda_i\geq \alpha\left(k\lambda_{k+1}+\sum\limits_{i=k+1}^n\lambda_i\right)
	\end{equation}
	and therefore
	\begin{equation}
	r(\tau) \leq \frac{\alpha\lambda_{k+1}}{\frac{\alpha}{\tau}\left(k\lambda_{k+1}+\sum\limits_{i=k+1}^n\lambda_i\right) + \left(1-\frac{1}{\tau}\right)\alpha\lambda_n} = \frac{\lambda_{k+1}}{\frac{1}{\tau}\left(k\lambda_{k+1}+\sum\limits_{i=k+1}^n\lambda_i\right) + \left(1-\frac{1}{\tau}\right)\lambda_n}.
	\end{equation}
	\item $\gamma = \frac{\alpha\lambda_j+\beta_j}{C_k} = \frac{\theta}{C_k}$ for some $j\in\{1,\ldots,k\}.$ Then
	\begin{equation}
	C_k = \alpha\trace{\mA} + \sum_{i=1}^k\beta_i = \sum\limits_{i=1}^k(\alpha\lambda_i+\beta_i) + \alpha\sum\limits_{i=k+1}^n\lambda_i \geq k\theta + \alpha\sum\limits_{i=k+1}^n\lambda_i,\quad \Delta \geq \alpha\lambda_n
	\end{equation}
	whence
	\begin{equation}
		r(\tau) \leqslant \frac{\theta}{\frac{1}{\tau}\left(k\theta+\alpha\sum\limits_{i=k+1}^{n}\lambda_i\right) + \left(1-\frac{1}{\tau}\right)\alpha\lambda_n}.
	\end{equation}
	Note that the function $f(\theta) = \frac{\theta}{\frac{1}{\tau}\left(k\theta+\alpha\sum\limits_{i=k+1}^{n}\lambda_i\right) + \left(1-\frac{1}{\tau}\right)\alpha\lambda_n}$ increases monotonically:
	\begin{equation}
		\begin{array}{cc}
			f'(\theta) = \frac{1}{\frac{1}{\tau}\left(k\theta+\alpha\sum\limits_{i=k+1}^{n}\lambda_i\right) + \left(1-\frac{1}{\tau}\right)\alpha\lambda_n} - \frac{\frac{k}{\tau}\theta}{\left(\frac{1}{\tau}\left(k\theta+\alpha\sum\limits_{i=k+1}^{n}\lambda_i\right) + \left(1-\frac{1}{\tau}\right)\alpha\lambda_n\right)^2}\\
			= \frac{\frac{\alpha}{\tau}\sum\limits_{i=k+1}^n\lambda_i + \left(1-\frac{1}{\tau}\right)\alpha\lambda_n}{\left(\frac{1}{\tau}\left(k\theta+\alpha\sum\limits_{i=k+1}^{n}\lambda_i\right) + \left(1-\frac{1}{\tau}\right)\alpha\lambda_n\right)^2} > 0.
		\end{array}
	\end{equation}
	From this and inequality $\alpha\lambda_{k+1}\geq\theta$ we get
	\begin{equation}
		r(\tau) \leq \frac{\alpha\lambda_{k+1}}{\frac{1}{\tau}\left(\alpha k\lambda_{k+1}+\alpha\sum\limits_{i=k+1}^n\lambda_i\right) + \left(1-\frac{1}{\tau}\right)\alpha\lambda_n} = \frac{\lambda_{k+1}}{\frac{1}{\tau}\left(k\lambda_{k+1}+\sum\limits_{i=k+1}^n\lambda_i\right) + \left(1-\frac{1}{\tau}\right)\lambda_n}.
	\end{equation}
\end{enumerate}
For both possible cases we shown that $r(\tau) \leq \frac{\lambda_{k+1}}{\frac{1}{\tau}\left(k\lambda_{k+1}+\sum\limits_{i=k+1}^n\lambda_i\right) + \left(1-\frac{1}{\tau}\right)\lambda_n}$. So, it is the optimal rate in this family of methods. Note that $\alpha$ could be any positive number. Optimal distribution is unique and it is:
\begin{equation}\label{fam:optimal_distrib_1}
s\sim\cD \quad \Leftrightarrow \quad s = \begin{cases}
e_i & \text{with probability\;}  p_i = \frac{\mA_{ii}}{C_k} \quad i=1,2,\ldots,n\\
u_i & \text{with probability\;} p_{n+i} = \frac{\lambda_{k+1}-\lambda_i}{C_k} \quad i=1,2,\ldots,k,
\end{cases}
\end{equation}
where $C_k = k\lambda_{k+1}+\sum\limits_{i=k+1}^n\lambda_i$. For $k=0$ we obtain mRCD, for $k=n-1$ we get the optimal rate $\frac{\frac{1}{n}}{\frac{1}{\tau}+(1-\frac{1}{\tau})\frac{1}{n}}$ and rate increases when $k$ increases.

\section{Results mentioned informally in the paper}

\subsection{Adding ``largest'' eigenvectors does not help}

In Section~\ref{sec:iugd8998ds} describing the SSCD method we have argued, without supplying any detail, that it does not make sense to consider replacing the $k$ ``smallest'' eigenvectors with a few ``largest'' eigenvectors. Here we make this statement precise, and prove it.

 Fix $k\in\{0,1,\ldots,n-1\}$ and consider running stochastic descent with the distribution $\cD$ defined via
\begin{equation}
s\sim\cD \quad \Leftrightarrow \quad s = \begin{cases}
e_i & \text{with probability\;}  p_i = \frac{\alpha\mA_{ii}}{C_k} \quad i=1,2,\ldots,n\\
u_i & \text{with probability\;} p_{n-k+i} = \frac{\beta_i}{C_k} \quad i=k+1,k+2,\ldots,n,
\end{cases}
\end{equation}
where $C_k = \alpha\trace{\mA} + \sum\limits_{i=k+1}^n\beta_i$ and for $\beta_i \geq 0$ for $i\in\{1,2,\ldots,k\}$.

That is, we consider ``enriching'' RCD with a collection of a $n-k$ eigenvectors corresponding to the $n-k$ largest eigenvectors of $\mA$.  We have the following negative result, which loosely speaking says that it is not worth enriching RCD with such vectors.

\begin{theorem} \label{thm:last_eigs} The optimal parameters of the above method are $k=n$ or $\beta_i=0$ for all $i=k+1,\dots,n$.  
\end{theorem}
\begin{proof} We follow similar steps as in the proof of Theorem~\ref{thm:SSCD}. In this setting we have
\[ \Exp_{s\sim\cD}[\mH] = \frac{1}{C_k}\left(\alpha\mI + \sum\limits_{i=k+1}^{n}\frac{\beta_i}{\lambda_i}u_iu_i^\top\right),\]
whence
\[ \mA\Exp_{s\sim\cD}[\mH] = \frac{1}{C_k}\left(\alpha\mA + \sum\limits_{i=k+1}^n\beta_iu_iu_i^\top\right) = \frac{1}{C_k}\left(\sum\limits_{i=1}^k\alpha\lambda_{i}u_iu_i^\top + \sum\limits_{i=k+1}^n(\beta_i+\alpha\lambda_i)u_iu_i^\top\right) \]
and
\[
\lambda_{\min}\left(\mA\Exp_{s\sim\cD}[\mH]\right) = \frac{\alpha\lambda_1}{C_k} \leq \frac{\alpha\lambda_1}{\alpha\trace{\mA}} = \frac{\lambda_1}{\trace{\mA}}.
\]

It means that the best rate in this family of methods is obtained when $k=n$ or $\beta_i=0$ for all $i=k+1,\ldots,n$. 
\end{proof}

So, to use spectral information about $n-k$ last eigenvectors we should use more complicated distributions (for instance, one may need to replace $\alpha$ by $\alpha_i$).

\subsection{Stochastic Conjugate Descent}

The lemma below was referred to in Section~\ref{sec:8ys089h0df}. As explained in that section, this lemma can  be used to argue that stochastic conjugate descent achieves the same rate as SSD: $\cO(n \log \tfrac{1}{\epsilon})$.

\begin{lemma}\label{lemma:A_orthogonal_basic_method}
	Let $\Set{v_1 \ldots v_n}$ be an $\mA$-orthonormal system:
	\[
	v_i^\top \mA v_j = \begin{cases}1 & i=j\\
	0 & i\neq j \end{cases}.
	\]
	If distribution $\cD$ consists of vectors $v_i$ chosen with uniform probabilities, then $\lambda_{\min}(\mW) = \frac{1}{n}$
\end{lemma}
\begin{proof}
	That is,
	\begin{equation}\label{eq:W_matr}
	\mW =
	\mA^{1/2} \Exp[\mH] \mA^{1/2} = \frac{1}{n}\Sum{i=1}{n} \frac{\mA^{1/2} v_i v_i^\top \mA^{1/2}}{v_i^\top \mA v_i}= \frac{1}{n}\Sum{i=1}{n} \mA^{1/2} v_i v_i^\top \mA^{1/2}.
	\end{equation}
	Making a substitution $u_i = \mA^{1/2}s_i$, we get
	\begin{equation}
	\mW =
	\frac{1}{n} \Sum{i=1}{n} u_iu_i^\top = \frac{1}{n} \mI,
	\end{equation}
	because $\Set{u_1 \ldots u_n}$ is orthonormal system.
\end{proof}

\section{Inexact Stochastic Conjuagate Descent}\label{sec:inexact_methods}

In Section~\ref{sec:8ys089h0df} we stated, that we can achieve an optimal rate of stochastic descent by using uniform distribution over a set of $n$ $\mA$-conjugate directions. In this section we consider the case when $\mA$-conjugate directions are computed approximately.

More formally, we consider a system of vectors $v_1, \ldots, v_n$, which satisfies $\abs{v_i^{\top} \mA v_j} \leq \varepsilon$ for $i \neq j$ and $v_i^{\top} \mA v_i = 1$ for some parameter $\varepsilon>0$. Further we'll call such vectors $\varepsilon$-approximate $\mA$-conjugate vectors.

Now we formalize the idea of using approximate $\mA$-conjugate directions in Stochastic Conjugate Descent, which leads to Algorithm~\ref{alg:iSconD}.
\begin{algorithm}[H]
	\caption{Inexact Stochastic Conjugate Descent (iSconD)}
	\label{alg:iSconD}
	\begin{algorithmic}
		\STATE {\bfseries Initialize:} $x_0 \in \R^n$; $v_1, \ldots, v_n$: $\varepsilon$-approximate
		$\mA$-conjugate directions
		\FOR{$t=0,1,2,\dots$}
		\STATE Choose $i\in [n]$ uniformly at random
		\STATE Set $x_{t+1} = x_t - v_i^\top \left(\mA x_t - b\right) v_i$
		\ENDFOR
	\end{algorithmic}
\end{algorithm}
For this algorithm we are going to obtain rate $\cO(n \log \tfrac{1}{\epsilon})$, the optimal rate for stochastic descent.

\subsection{Lemma}

	\begin{lemma}\label{lemma:lambda_bound}
		Let $\mS = [v_1, \ldots, v_n]$, where $v_1, \ldots, v_n$ are $\varepsilon$-approximate $\mA$-conjugate vectors.\\
		If $\varepsilon$ satisfies
		\begin{equation}\label{eq:eps_bound}
		\varepsilon < \frac{1}{n-1}
		\end{equation}
		then $\tilde{\mI}\eqdef \mS^\top \mA \mS$ is positive definite matrix and
		\begin{equation}\label{eq:lambda_min_bound}
		\lambda_{\min}(\mA^{1/2}\mS\mS^\top \mA^{1/2}) \geq
		1 - \varepsilon(n-1)\frac{1 + \varepsilon(n-1)}{1 - \varepsilon(n-1)}
		\end{equation}
		\begin{equation}\label{eq:lambda_max_bound}
		\lambda_{\max}(\mA^{1/2}\mS\mS^\top\mA^{1/2}) \leq
		1+\varepsilon(n-1)\frac{1 + \varepsilon(n-1)}{1 - \varepsilon(n-1)}
		\end{equation}
	\end{lemma}
	\begin{proof}
	For unit vector $x$ we can write
	\begin{eqnarray*}
		x^\top \tilde{\mI} x = \Sum{i,l}{} x_i x_l \tilde{\mI}_{il} =
		1 + \Sum{i,l: i\neq l}{} x_i x_l \tilde{\mI}_{il}  \geq
		1 - \varepsilon \Sum{i,l: i\neq l}{} \frac{1}{2}(x_i^2 + x_l^2) =
		1 - \varepsilon(n-1).
	\end{eqnarray*}
	Under condition \eqref{eq:eps_bound} we get
	$x^\top \tilde{\mI} x > 0$ for any $x$, which proves the first part of lemma.
	
	Since $\mS^\top \mA \mS$ is positive definite, vectors $\mA^{1/2}v_1, \ldots,  \mA^{1/2}v_n$ are linearly independent.
	Any unit vector $x$ may be represented as $x = \mA^{1/2} \mS \alpha$ with normalization
	condition:
	\begin{eqnarray}
		1 = x^{\top} x = \alpha^\top \tilde{\mI} \alpha =
		\alpha^\top\alpha + \Sum{i,l: i\neq l}{}\tilde{\mI}_{il} \alpha_i \alpha_l,
	\end{eqnarray}
	or
	\begin{equation}\label{eq:alpha_norm}
		\alpha^\top \alpha = 1 - \Sum{i,l: i\neq l}{}\tilde{\mI}_{il} \alpha_i \alpha_l.
	\end{equation}
	
	Now we can analyse spectrum of matrix $\mA^{1/2}\mS\mS^\top\mA^{1/2}$.
	\begin{eqnarray*}
		x^\top\mA^{1/2}\mS\mS^\top\mA^{1/2} x =
		\alpha^\top  \mS^\top \mA \mS \mS^\top \mA\mS \alpha =
		\alpha^\top \tilde{\mI}^2 \alpha = \norm{\tilde{\mI}\alpha}_2^2 =
		\Sum{i=1}{n}
		\left(
			\Sum{l=1}{n} \tilde{\mI}_{il} \alpha_l
		\right)^2 =\\=
		\Sum{i=1}{n}
		\left(
			\alpha_i + \Sum{l:l\neq i}{} \tilde{\mI}_{il} \alpha_l
		\right)^2=
		\Sum{i=1}{n}
		\left(
			\alpha_i^2 + 2\alpha_i  \Sum{l:l\neq i}{} \tilde{\mI}_{il} \alpha_l
			+
			\left(
				\Sum{l:l\neq i}{} \tilde{\mI}_{il} \alpha_l
			\right)^2
		\right).
	\end{eqnarray*}
	Using \eqref{eq:alpha_norm} we get
	\begin{equation} \label{eq:quadr_form}
		x^\top\mA^{1/2}\mS\mS^\top \mA^{1/2} x = 1 +
		\underbrace
		{
			\Sum{i,l:l\neq i}{} \tilde{\mI}_{il} \alpha_i \alpha_l
		}_{R_1} +
		\underbrace
		{
	 		\Sum{i=1}{n}
			\left(
				\Sum{l:l\neq i}{} \tilde{\mI}_{il} \alpha_l
			\right)^2
		}_{R_2}=
		1 + R_1 + R_2
	\end{equation}
	To estimate $\abs{R_1}$ and $\abs{R_2}$ we need to estimate $\alpha^\top\alpha$ using \eqref{eq:alpha_norm}:
	\begin{eqnarray*}
		\alpha^\top\alpha \leq 1 +
		 \varepsilon\Sum{i,l:i\neq l}{} \frac{\alpha_i^2 + \alpha_l^2}{2} = 1+ 
		 \varepsilon(n-1) \alpha^\top\alpha,
	\end{eqnarray*}
	which under condition \eqref{eq:eps_bound} implies that
$
		\alpha^\top\alpha \leq \frac{1}{1 - \varepsilon(n-1)}.
$
	Now we can estimate $\abs{R_1}$ and $\abs{R_2}$.
	\begin{equation}\label{eq:residual_1}
		R_1 \leq \varepsilon\Sum{i,l:i\neq l}{} \frac{\alpha_i^2 + \alpha_l^2}{2} = \varepsilon(n-1) \alpha^\top\alpha
		\leq
		\frac{\varepsilon(n-1)}{1-\varepsilon(n-1)}
	\end{equation}
	\begin{equation}\label{eq:residual_2}
		R_2 \leq \Sum{i=1}{n} (n-1) \Sum{l:l\neq i}{} \alpha_l^2 \varepsilon^2 =
		\varepsilon^2(n-1)^2 \alpha^\top\alpha
		\leq
		\frac{\varepsilon^2(n-1)^2}{1-\varepsilon(n-1)}
	\end{equation}
	Finally from \eqref{eq:quadr_form}, \eqref{eq:residual_1} and \eqref{eq:residual_2} we get
	\begin{equation}
		\lambda_{\min}(\mA^{1/2}\mS\mS^\top\mA^{1/2}) \geq  1 - \frac{\varepsilon(n-1)+\varepsilon^2(n-1)^2}{1-\varepsilon(n-1)}=
		1-\varepsilon(n-1)\frac{1 + \varepsilon(n-1)}{1 - \varepsilon(n-1)}
	\end{equation}
	\begin{equation}
		\lambda_{\max}(\mA^{1/2}\mS\mS^\top\mA^{1/2}) \leq 1 + \frac{\varepsilon(n-1)+\varepsilon^2(n-1)^2}{1-\varepsilon(n-1)}=
		1 + \varepsilon(n-1)\frac{1 + \varepsilon(n-1)}{1 - \varepsilon(n-1)}
	\end{equation}

\end{proof}

\begin{corollary}\label{corollary:cond_number}
	If $\varepsilon < \frac{\sqrt{2}-1}{(n-1)}$ then $\lambda_{\min}(\mA^{1/2}\mS\mS^\top\mA^{1/2}) > 0$ and condition number of $\mA^{1/2}\mS\mS^\top\mA^{1/2}$ has the following bound:
	\begin{equation}\label{eq:cond_number_bound}
		\frac{\lambda_{\max}(\mA^{1/2}\mS\mS^\top\mA^{1/2})}{\lambda_{\min}(\mA^{1/2}\mS\mS^\top\mA^{1/2})} <
		\frac{1 + \varepsilon^2(n-1)^2}{1 -2\varepsilon(n-1) - \varepsilon^2(n-1)^2}
	\end{equation}
\end{corollary}

\subsection{Rate of convergence}

The following theorem gives the rate of convergence of iSconD.

\begin{theorem}\label{thm:iSconD}
	Let $\mS = [v_1 \ldots v_n]$, where $\Set{v_1 \ldots v_n}$ is $\varepsilon$-approximate $\mA$-conjugate system. If $\varepsilon \leq \frac{1}{3(n-1)}$ then $\lambda_{\min}(\mW) > \frac{1}{3n}$, which means that the rate of  iSconD is $\cO(n \log \tfrac{1}{\epsilon})$.
\end{theorem}

\begin{proof}
As in Lemma \ref{lemma:A_orthogonal_basic_method}, we can show that
$
	\mW = \frac{1}{n} \mA^{1/2} \mS\mS^\top \mA,
$
where $\mS = [v_1 \ldots v_n]$. Using bound \eqref{eq:lambda_min_bound} and Corollary~ \ref{corollary:cond_number}, we get
\begin{equation}
	\lambda_{\min}(\mW) > \frac{1}{n}
	\left(
		1 - \varepsilon(n-1)\frac{1 + \varepsilon(n-1)}{1 - \varepsilon(n-1)}
	\right)
\end{equation}
for small enough $\varepsilon$ (see Corollary \ref{corollary:cond_number}).
For $\varepsilon = \frac{1}{3(n-1)}$ we get
$
	\lambda_{\min}(\mW) > \frac{1}{3n}.
$
\end{proof}

\subsection{Experiment}
Figure \ref{fig:inexactness} illustrates the theoretical results about iSonD.
For this experiment we generate
random orthogonal matrix $\mV$ and random symmetric positive definite matrix $\tilde{\mI}$, which satisfies $\tilde{\mI}_{ii} = 1$, $\abs{\tilde{\mI}_{ij}} \leq \varepsilon$ for $i\neq j$. Columns of matrix $\mA^{-1/2}\mV\tilde{\mI}^{1/2}$ were taken as approximate $\mA$-conjugate vectors.

\begin{figure}[H]
	\centering
	\includegraphics[width=0.50\linewidth]{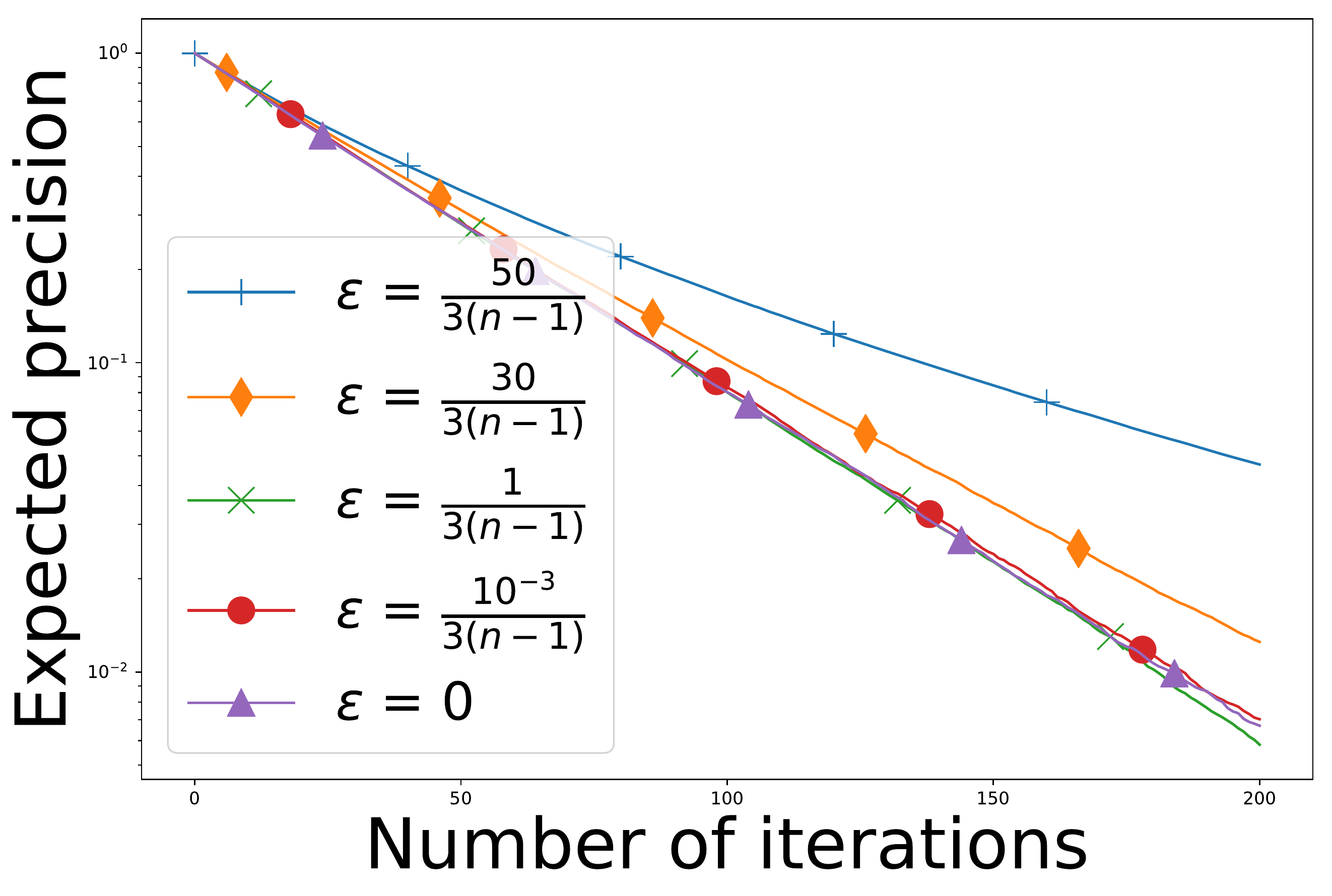}
	\caption{Expected precision $\Exp \left[\frac{||x_t - x_\ast||^2_\mA}{||x_0 - x_\ast||^2_\mA}\right]$ vs. the number of iterations of iSconD with different choices of parameter $\varepsilon$.}
	\label{fig:inexactness}
\end{figure}

\subsection{Approximate solution without iterative methods}
Note that the problem \eqref{eq:quad_opt} is equivalent to the following problem of finding $x$ such that
\begin{equation}\label{eq:linsys}
\mA x = b.
\end{equation}
Let $\mS=[v_1\ldots v_n]$ be a set of $\mA$-conjugate vectors, i.e.,
$
\mS^\top\mA\mS = \mI.
$
We can now find the solution to the linear system \eqref{eq:linsys}. Since
$
	\mS^\top b=\mS^\top\mA x = \mS^\top \mA \mS \mS^{-1}x = \mS^{-1}x,
$
we conclude that 
\begin{equation}\label{eq:solution_via_A_ort_vec}
x = \mS\mS^\top b.
\end{equation}

We will now show that unlike in the exact case, using formula \eqref{eq:solution_via_A_ort_vec} with $\varepsilon$-approximate $\mA$-conjugate vectors does {\em not} lead to a precise solution of our problem.

\begin{lemma}\label{lemma:solution_via_eps_A_ort_vec}
	Let $\mS=[v_1\ldots v_n]$ be an $\varepsilon$-$\mA$-orthonormal system.
	Let $x_*=\mA^{-1} b$ be the solution of the linear system  \eqref{eq:linsys}.
	Let $\hat{x}$ be an estimate of the solution, calculated with formula  \eqref{eq:solution_via_A_ort_vec}
	using $\varepsilon$-approximate $\mA$-conjugate vectors:
$
	\hat{x} = \mS\mS^\top b.
$
	If $\varepsilon< 1/(n-1)$, then
	\begin{equation}
	\norm{\hat{x} - x_*}_{\mA} \leq
	\varepsilon(n-1)\frac{1 + \varepsilon(n-1)}{1 - \varepsilon(n-1)}
	\norm{x_*}_{\mA}
	\end{equation}
\end{lemma}

\begin{proof} Note that $
		\mA^{1/2}\hat{x}  = \mA^{1/2}\mS\mS^\top \mA^{1/2} \mA^{1/2} x_* = \hat{\mI}\mA^{1/2}x_*,
$
	where $\hat{\mI} = \mA^{1/2}\mS\mS^\top \mA^{1/2}$. From Lemma~\ref{lemma:lambda_bound} we now get that
\begin{equation}
\abs{\lambda_i(\hat{\mI} - \mI)} \leq \varepsilon(n-1)\frac{1 + \varepsilon(n-1)}{1 - \varepsilon(n-1)},
\end{equation}
and hence
\begin{equation}
\norm{\hat{\mI} - \mI}_2 \leq  \varepsilon(n-1)\frac{1 + \varepsilon(n-1)}{1 - \varepsilon(n-1)}.
\end{equation}
Therefore,
\begin{eqnarray*}
	\norm{\hat{x} -x_*}_{\mA} = \norm{\mA^{1/2}(\hat{x} - x_*)}_2=
	\norm{(\hat{\mI} - \mI)\mA^{1/2}x_*}_2 \leq
	\norm{\hat{\mI} - \mI}_2 \norm{\mA^{1/2}x_*}_2
	\leq
	 \varepsilon(n-1)\frac{1 + \varepsilon(n-1)}{1 - \varepsilon(n-1)}
	 \norm{x_*}_{\mA}.
\end{eqnarray*}
\end{proof}

If we choose $\varepsilon=\frac{1}{3(n-1)}$, like we did in Theorem~ \ref{thm:iSconD}, we get the following precision:
\begin{equation}
\norm{\hat{x} - x_*}_{\mA} \leq \frac{2}{3} \norm{x_*}_{\mA},
\end{equation}
which is rather poor. However if we use Algorithm \ref{alg:iSconD}, we can get approximate solution with any precision (after enough iterations).

\section{Inexact SSD: a method that is not a special case of stochastic descent} \label{sec:iSSD}

In Section~\ref{sec:SSD} we defined Stochastic Spectral Descent (Algorithm~\ref{alg:SSD}).
We now design a new method which will ``try''  to use the same iterations, but with {\em  inexact} eigenvectors of $\mA$. We  call $w$ an inexact eigenvector of $\mA$ if
\begin{equation}\label{eq:inexact_eigenvector}
\mA w = \lambda w + \varepsilon
\end{equation}
for some $\varepsilon$ and $\lambda > 0$ (inexact eigenvalue). Clearly, {\em any} vector can be written in the form \eqref{eq:inexact_eigenvector}. This idea leads to Algorithm~\ref{alg:iSSD}.

\begin{algorithm}[h]
	\caption{Inexact Stochastic Spectral Descent (iSSD)}
	\label{alg:iSSD}
	\begin{algorithmic}
		\STATE {\bfseries Initialize:} $x_0 \in \R^n$; $(w_1,\lambda_1), \dots (w_n,\lambda_n)$: inexact eigenvectors and eigenvalues of $\mA$
		\FOR{$t=0,1,2,\dots$}
		\STATE Choose $i\in [n]$ uniformly at random
		\STATE Set $x_{t+1} = x_t - \left(w_i^\top x_t - \frac{w_i^\top b}{ \lambda_i}\right) w_i$
		\ENDFOR
	\end{algorithmic}
\end{algorithm}

Note that the above method is {\em not equivalent} to applying stochastic descent $\cD$ being the uniform distribution over the inexact eigenvectors. This is because in arriving at SSD, we have used some properties of the eigenvectors and eigenvalues to simplify the calculation of the stepsize. The same simplifications do {\em not} apply for inexact eigenvectors. Nevertheless, we can formally run SSD, as presented in Algorithm~\ref{alg:SSD},  and replace the exact eigenvectors and eigenvalues by  inexact versions thereof, thus capitalizing on the fast computation of stepsize which positively affects the cost of one iteration of the method. This leads to Algorithm~\ref{alg:iSSD}. 

Hence, in order to analyze the above method, we need to develop a completely new approach. We will  show that Algorithm~\ref{alg:iSSD} converges only to a neighbourhood of the optimal solution.

\subsection{Lemmas}

Further we will use the following notation:
$\mS = [w_1 \ldots w_n]$ -- inexact eigenvectors matrix,
$\Lambda = \diag{\lambda_1 \ldots \lambda_n}$ -- inexact eigenvalues matrix,
$\mE = [\varepsilon_1 \ldots \varepsilon_n]$ -- error matrix,
$\tilde{\mA} = \mS \Lambda \mS^\top$ -- estimation of matrix $\mA$.
We also assume, that inexact eigenvectors are $\varepsilon$-approximate orthonormal for $\varepsilon < \frac{1}{n-1}$, i.e. $w_i^{\top} w_i = 1$, $\abs{w_i^{\top}w_j} \leq \varepsilon$ for $i \neq j$.

The following lemma gives an answer to the question: how precise is  $\tilde{\mA}$ as an estimate of matrix $\mA$?
\begin{lemma}\label{lemma:estimate_A}
$
	\tilde{\mA} = \hat{\mI} \mA - \mS \mE^\top,
$
	where matrix $\hat{\mI} = \mS\mS^\top$ satisfies
	\begin{equation}\label{eq:estimate_identity}
	\norm{\hat{\mI} - \mI}_2 \leq  \varepsilon(n-1)\frac{1 + \varepsilon(n-1)}{1 - \varepsilon(n-1)}.
	\end{equation}
\end{lemma}
\begin{proof}
	Indeed, the definition of inexact eigenvectors can be written in matrix form as 
$
	\mA\mS = \mS\Lambda + \mE,
	$
	from which follows that
	$
	\hat{\mI}\mA = \mS\mS^\top \mA = \mS\Lambda \mS^\top + \mS\mE^\top.
$
Equality~\eqref{eq:estimate_identity} follows immediately from Lemma \ref{lemma:lambda_bound}.
\end{proof}

The next lemma gives a general recursion capturing one step of  iSSD, shedding light on the convergence of the method.

\begin{lemma}\label{lemma:inexact_SSD_conv}
	Sequence of $\{x_t\}$ generated by inexact SSD satisfies equality
	\begin{eqnarray*}
		\Exp \norm{x_{t+1} - x_*}^2_{\mA} &
		= &
		\left(1 - \frac{1}{n}\right)
		\Exp \norm{x_t - x_*}^2_{\mA} +
		\frac{1}{n}\Exp\left[
		(x_t - x_*)^\top \Gamma (x_t - x_*)\right]\\
		& + &
		\frac{1}{n}
		\left(
		\Exp \norm{x_t}^2_{\mE \Lambda^{-1} \mE^\top}
		+
		x_*^\top\mE \Lambda^{-2} \mC \mE^\top x_*\right)
		- 
		\frac{2}{n}\Exp\left[(x_t - x_*)^\top \mS \mC \Lambda^{-1}\mE^\top x_*\right],
	\end{eqnarray*}
	where
	$\Gamma = (\mI - \hat{\mI})\mA - \mS\mE^\top - \mE\Lambda^{-1}\mE^\top + \mS \mC \mS^\top$ and 
\begin{equation}\label{eq:C}\mC = \diag{w_1^\top\varepsilon_1 \ldots w_n^\top\varepsilon_n}.\end{equation}
\end{lemma}
\begin{proof}
	\begin{eqnarray*}
		\norm{x_{t+1} - x_*}_{\mA}^2 &=&
		\norm
		{
			x_t - x_* - \omega w_tw_t^\top(x_t-x_*) +
			\omega\frac{\varepsilon_t^\top x_*}{\lambda_t}w_t
		}_\mA^2 \\
		&=& 
		\norm{x_t - x_*}_{\mA}^2 +
		\omega^2 w_t^\top\mA w_t
		\left(
		w_t^\top(x_t - x_*) - \frac{\varepsilon_t^\top x_*}{\lambda_t}
		\right)^2 \\
		&& \qquad  + 
		2\omega (x_t - x_*)^\top \mA w_t
		\left(
		\frac{\varepsilon_t^\top x_*}{\lambda_{t}} -
		w_t^\top(x_t - x_*)
		\right)  \\ &=&
		\norm{x_t - x_*}_{\mA}^2 +
		\omega^2(\lambda_t + w_t^\top \varepsilon_t)
		\left(
		w_t^\top(x_t - x_*) - \frac{\varepsilon_t^\top x_*}{\lambda_t}
		\right)^2 \\
		&& \qquad +
		2\omega (x_t - x_*)^\top (\lambda_t w_t + \varepsilon_t)
		\left(
		\frac{\varepsilon_t^\top x_*}{\lambda_{t}} -
		w_t^\top(x_t - x_*)
		\right)\\ &=&
		\norm{x_t - x_*}_{\mA}^2
		-
		\omega(2 - \omega)
		(x_t - x_*)^\top \lambda_t w_t w_t^\top (x_t - x_*) + 
		\omega^2 \frac{x_*^\top \varepsilon_t \varepsilon_t^\top x_*}{\lambda_t} \\
		&& \qquad +
		2\omega \frac{(x_t - x_*)^\top \varepsilon_t \varepsilon_t^\top x_*}{\lambda_t} +
		\omega^2 w_t^\top \varepsilon_t
		\left(
		w_t^\top(x_t - x_*) -
		\frac{\varepsilon_t^\top x_*}{\lambda_{t}}
		\right)^2 \\
		&& \qquad 
		+ 2 (\omega - \omega^2) (x_t - x_*)^\top w_t\varepsilon_t^\top x_* 
		-2\omega (x_t - x_*)^\top w_t \varepsilon_t^\top (x_t - x_*)\\ &=&
		\norm{x_t - x_*}_{\mA}^2
		-
		\omega(2 - \omega)
		(x_t - x_*)^\top \lambda_t w_t w_t^\top (x_t - x_*)
		+
		\norm
		{
			x_*(\omega - 1) + x_t
		}_{\frac{\varepsilon_t \varepsilon_t^\top}{\lambda_{t}}}\\
		&& \qquad 
		-
		\norm
		{
			x_t - x_*
		}_{\frac{\varepsilon_t \varepsilon_t^\top}{\lambda_{t}}} + 
		2\omega (x_t - x_*)^\top w_t \varepsilon_t^\top
		(x_*(2-\omega) - x_t)+
		\omega^2 w_t^\top \varepsilon_t
		\left(
		w_t^\top(x_t - x_*) -
		\frac{\varepsilon_t^\top x_*}{\lambda_{t}}
		\right)^2.
	\end{eqnarray*}
	Now we can take conditional expectation $\Exp[\;\cdot\mid x_t]$.
	\begin{eqnarray*}
		\Exp[\norm{x_{t+1} - x_*}^2_{\mA}\mid x_t] =
		\norm{x_t - x_*}^2_{\mA}-
		\frac{\omega(2 - \omega)}{n}
		\norm{x_t - x_*}^2_{\tilde{\mA}}+
		\frac{1}{n}\norm{x_*(\omega - 1) + x_t}^2_{\Sigma} -
		\frac{1}{n}\norm{x_t - x_*}^2_{\Sigma} - \\ -
		\frac{2\omega}{n} (x_t - x_*)^\top	\mS\mE^\top
		(x_t - (2 - \omega) x_*) +
		\frac{\omega^2}{n}
		\Sum{i=1}{n} w_i^\top \varepsilon_i
		\left(
		w_i^\top(x_t - x_*) - \frac{\varepsilon_i ^ \top x_*}{\lambda_i}
		\right)^2,
	\end{eqnarray*}
	where
	$\Sigma = \mE \Lambda^{-1} \mE^\top$.
	
	Now we set $\omega = 1$ and use Lemma~\ref{lemma:estimate_A}.
	\begin{eqnarray*}
		\Exp[\norm{x_{t+1} - x_*}^2_{\mA}\mid x_t] =
		\norm{x_t - x_*}^2_{\mA}-
		\frac{1}{n}
		\norm{x_t - x_*}^2_{\tilde{\mA}}+
		\frac{1}{n}\norm{x_t}^2_{\Sigma} -
		\frac{1}{n}\norm{x_t - x_*}^2_{\Sigma} - \\ -
		\frac{2}{n}
		(x_t - x_*)^\top\mS\mE^\top(x_t-x_*) +
		\frac{1}{n}
		\Sum{i=1}{n} w_i^\top\varepsilon_i
		\left(
		w_i^\top(x_t - x_*) - \frac{\varepsilon_i ^ \top x_*}{\lambda_i}
		\right)^2 = \\ =
		\norm{x_t - x_*}^2_{\mA} \left(1 - \frac{1}{n}\right) +
		\frac{1}{n}
		(x_t - x_*)^\top
		\left((\mI - \hat{\mI})\mA + \mS\mE^\top - 2\mS \mE^\top\right)(x_t-x_*) +
		\frac{1}{n}\norm{x_t}^2_{\Sigma} -
		\frac{1}{n}\norm{x_t - x_*}^2_{\Sigma} + \\+
		\frac{1}{n}
		\Sum{i=1}{n} w_i^\top \varepsilon_i
		\left(
		w_i^\top(x_t - x_*) - \frac{\varepsilon_i ^ \top x_*}{\lambda_i}
		\right)^2 = \\=
		\left(1 - \frac{1}{n}\right)
		\norm{x_t - x_*}^2_{\mA} +
		\frac{1}{n}
		(x_t - x_*)^\top
		\left((\mI - \hat{\mI})\mA - \mS \mE^\top - \Sigma\right)(x_t - x_*) +
		\frac{1}{n}\norm{x_t}^2_{\Sigma} +  \\+
		\frac{1}{n}
		\Sum{i=1}{n} w_i^\top \varepsilon_i
		\left(
		w_i^\top(x_t - x_*) - \frac{\varepsilon_i ^ \top x_*}{\lambda_i}
		\right)^2 = \\=
		\left(1 - \frac{1}{n}\right)
		\norm{x_t - x_*}^2_{\mA} +
		\frac{1}{n}
		(x_t - x_*)^\top
		\left((\mI - \hat{\mI})\mA - \mS \mE^\top - \Sigma\right)(x_t - x_*) +
		\frac{1}{n}\norm{x_t}^2_{\Sigma} +  \\+
		\frac{1}{n}\norm{x_t - x_*}^2_{\mS \mC \mS^\top} +
		\frac{1}{n} x_*^\top\left(\mE \Lambda^{-2}\mC \mE^\top\right)x_* -
		\frac{2}{n} (x_t - x_*)^\top \mS \mC \Lambda^{-1}\mE^\top x_*,
	\end{eqnarray*}
	where $\mC = \diag{w_1^\top\varepsilon_1 \ldots w_n^\top\varepsilon_n}$.
	We get
	\begin{eqnarray*}
		\Exp[\norm{x_{t+1} - x_*}^2_{\mA}\mid x_t]
		&=&
		\left(1 - \frac{1}{n}\right)
		\norm{x_t - x_*}^2_{\mA} +
		\frac{1}{n}
		(x_t - x_*)^\top\Gamma(x_t - x_*)  \\
		&& \qquad + 
		\frac{1}{n}
		\left(
		\norm{x_t}^2_{\mE \Lambda^{-1} \mE^\top}
		+
		x_*^\top\mE \Lambda^{-2} \mC \mE^\top x_*
		-
		2(x_t - x_*)^\top \mS \mC \Lambda^{-1}\mE^\top x_*
		\right),
	\end{eqnarray*}
	where
	$\Gamma = (\mI - \hat{\mI})\mA - \mS\mE^\top - \mE\Lambda^{-1}\mE^\top + \mS \mC \mS^\top$.
\end{proof}

The following lemma describes which inexact eigenvalues are optimal for a fixed set of inexact eigenvectors.

\begin{lemma}\label{lemma:optimal_lambda}
	Let $w_i$ be fixed. Then the choice 
	\begin{equation}\label{eq:optimal_lambda}
	\lambda_i = w_i^\top \mA w_i
	\end{equation}
	minimizes $\norm{\varepsilon_i}_2$ in $\lambda$, where $\epsilon_i \eqdef \|\mA w_i - \lambda w_i\|_2$. 	Moreover, for this choice of $\lambda_i$ we get
$
	w_i^\top\varepsilon_i = 0.
$
\end{lemma}
\begin{proof} Minimizing $\norm{\mA w_i - \lambda w_i}_2^2$ in $\lambda$ gives \eqref{eq:optimal_lambda}. For this choice of $\lambda_i$ we get
$
	w_i^\top \varepsilon_i = w_i^\top \mA w_i - \lambda_i w_i^\top w_i = w_i^\top \mA w_i - w_i^\top \mA w_i = 0.
$
\end{proof}

\subsection{Convergence} \label{sec:iSSD-conv}

Choosing eigenvalues as defined in \eqref{eq:optimal_lambda}, and in view of \eqref{eq:C},  we see that $\mC = 0$.
From this and Lemma~\ref{lemma:inexact_SSD_conv} we get
\begin{equation}
\Exp \norm{x_{t+1} - x_*}^2_{\mA}
=
\left(1 - \frac{1}{n}\right)
\Exp \norm{x_t - x_*}^2_{\mA} +
\frac{1}{n}\Exp\left[
(x_t - x_*)^\top \Gamma (x_t - x_*)\right] +
\frac{1}{n}\Exp
\norm{x_t}^2_{\mE \Lambda^{-1} \mE^\top},
\end{equation}
where $\Gamma = (\mI - \hat{\mI})\mA - \mS\mE^\top - \mE\Lambda^{-1}\mE^\top$.
From the Cauchy–Schwarz inequality we get
\begin{eqnarray}
	\frac{1}{n}\Exp
	\norm{x_t}^2_{\mE \Lambda^{-1} \mE^\top} =
	\frac{1}{n}\Exp \norm{x_t - x_* + x_*}^2_{\mE \Lambda^{-1} \mE^\top} \leq
	\frac{2}{n} \Exp \norm{x_t - x_*}^2_{\mE \Lambda^{-1} \mE^\top} + \frac{2}{n} \Exp \norm{x_*}^2_{\mE \Lambda^{-1} \mE^\top} ,
\end{eqnarray}
which leads to
\begin{equation}\label{eq:iSSD_rec}
\Exp \norm{x_{t+1} - x_*}^2_{\mA}
\leq
\left(1 - \frac{1}{n}\right)
\Exp \norm{x_t - x_*}^2_{\mA} +
\frac{1}{n}\Exp\left[
(x_t - x_*)^\top \mQ (x_t - x_*)\right] +
\frac{2}{n}
\norm{x_*}^2_{\mE \Lambda^{-1} \mE^\top},
\end{equation}
where $\mQ = (\mI - \hat{\mI})\mA - \mS\mE^\top + \mE\Lambda^{-1}\mE^\top$. Inequality~\eqref{eq:iSSD_rec} implies that
\begin{equation}\label{eq:iSSD_rec_simple}
	\Exp \norm{x_{t+1} - x_*}^2_{\mA} \leq
	\Exp \norm{x_t - x_*}^2_{\mA}+
	\frac{q-1}{n}
	\Exp \norm{x_t - x_*}^2_{\mA}
	 + \frac{r_0}{n},
\end{equation}
where $q = \max \frac{z^{\top} \mQ z}{z^{\top} \mA z}$, $r_0 = 2
\norm{x_*}^2_{\mE \Lambda^{-1} \mE^\top}$.

If the  errors $\varepsilon_1, \ldots, \varepsilon_n$ and $\varepsilon$ are small enough,  we can make $q$ and $r_0$ arbitrarily small for fixed $x_*$. From \eqref{eq:iSSD_rec_simple} we can see that $\Exp \norm{x_{t+1} - x_*}^2_{\mA}$ is going to decrease as long as 
\begin{equation}\label{eq:iSSD_precision_limit}
	\Exp \norm{x_{t} - x_*}^2_{\mA} \geq \frac{r_0}{1-q}.
\end{equation}

Hence,  for small enough $\varepsilon_1, \ldots, \varepsilon_n$ and parameter $\varepsilon$, iSSD will converge to a neighborhood of the optimal solution, with limited precision \eqref{eq:iSSD_precision_limit}.

\end{document}